\documentclass[11pt,reqno]{amsart}
\usepackage{amsfonts,amsthm,amssymb}
\usepackage{amsmath}
\usepackage{pdfsync}

\usepackage[mathscr]{eucal}
\usepackage{times}
\usepackage{graphicx}

\newtheorem{thm}{Theorem}[section]
\newtheorem{defin}{Definition}[section]
\newtheorem{lemma}{Lemma}[section]
\newtheorem{prop}{Proposition}[section]
\newtheorem{cor}{Corollary}[section]
\newtheorem{ex}{Example}
\newtheorem{definition}{Definition}[section]
\newtheorem{remark}{Remark}[section]

\newtheorem*{thmA}{Theorem A}
\newtheorem*{thmB}{Theorem B}

\newcommand{\R}{\ensuremath{\mathbb{R}}}
\newcommand{\Q}{\ensuremath{\mathbb{Q}}}
\newcommand{\Z}{\ensuremath{\mathbb{Z}}}

\newcommand{\PP}{\ensuremath{\mathbb{P}}}
\newcommand{\EE}{\ensuremath{\mathbb{E}}}
\newcommand{\E}{\ensuremath{\mathbb{E}}}

\newcommand{\LL}{\ensuremath{\mathcal{L}}}

\newcommand{\HH}{\ensuremath{\mathcal{H}}}
\newcommand{\B}{\ensuremath{\mathcal{B}}}

\newcommand{\J}{\ensuremath{\mathcal{J}}}
\newcommand{\WW}{\ensuremath{\mathcal{W}}}

\newcommand{\w}{\ensuremath{\mathrm{w}}}

\newcommand{\n}{\ensuremath{N \rightarrow \infty}}

\newcommand{\nw}{\ensuremath{N,W \rightarrow \infty}}

\newcommand{\A}{\ensuremath{\mathcal{A}}}
\newcommand{\G}{\ensuremath{\mathcal{G}}}
\newcommand{\e}{\ensuremath{\mathcal{E}}}

\newcommand{\1}{\ensuremath{\mathbf{1}}}
\newcommand{\U}{\ensuremath{\mathcal{U}}}
\newcommand{\dd}{\ensuremath{{[d]}}}

\newcommand{\aB}{\ensuremath{\overline{\mathcal{B}}}}
\newcommand{\aq}{\ensuremath{\overline{q}}}
\newcommand{\aP}{\ensuremath{\overline{\psi}}}
\newcommand{\aO}{\ensuremath{\overline{\Omega}}}
\newcommand{\aZ}{\ensuremath{\overline{Z}}}

\newcommand{\aV}{\ensuremath{\overline{V}}}

\newcommand{\aE}{\ensuremath{\overline{\mathcal{E}}}}
\newcommand{\aA}{\ensuremath{\overline{A}}}

\newcommand{\aL}{\ensuremath{\overline{L}}}

\newcommand{\anu}{\ensuremath{\overline{\nu}}}

\newcommand{\amu}{\ensuremath{\overline{\mu}}}
\newcommand{\GG}{\ensuremath{\G \backslash \{g_0 \}}}

\newcommand{\bL}{\ensuremath{\bar{L}}}

\newcommand{\eps}{\ensuremath{\varepsilon}}
\newcommand{\la}{\ensuremath{\lambda}}
\newcommand{\La}{\ensuremath{\Lambda}}
\newcommand{\al}{\ensuremath{\alpha}}

\newcommand{\Ga}{\ensuremath{\Gamma}}

\newcommand{\de}{\ensuremath{\delta}}
\newcommand{\De}{\ensuremath{\Delta}}
\newcommand{\si}{\ensuremath{\sigma}}
\newcommand{\om}{\ensuremath{\omega}}
\newcommand{\Om}{\ensuremath{\Omega}}

\newcommand{\ls}{\ensuremath{\lesssim}}
\newcommand{\gs}{\ensuremath{\gtrsim}}
\newcommand{\subs}{\ensuremath{\subseteq}}

\newcommand{\eq}{\begin{equation}
\newcommand{\ee}{\end{equation}}}

{ \theoremstyle{claim} }
{ \theoremstyle{notation} }
{ \theoremstyle{convention} }

\numberwithin{equation}{subsection}

\begin{document}
\title{A Multidimensional Szemer\'edi Theorem in the primes\\via combinatorics}
\author{Brian Cook, \'Akos Magyar, Tatchai Titichetrakun}
\thanks{
The second author is supported by NSERC grant 22R44824 and ERC-AdG. 321104.}

\begin{abstract} Let $A$ be a subset of positive relative upper density of $\PP^d$, the $d$-tuples of primes. We
present an essentially self-contained, combinatorial argument to show that $A$ contains infinitely many affine copies of any finite set $F\subs\Z^d$. This provides a natural multi-dimensional extension of the theorem of Green and Tao on the existence of long arithmetic progressions in the primes.


\end{abstract}

\maketitle

\section{Introduction}

\subsection{Background.}
A celebrated theorem in additive combinatorics due to Green and Tao \cite{GT1} establishes the existence of arbitrary long arithmetic progressions in the primes. It is proved that if $A$ is a subset of the primes of positive relative upper density then $A$ necessarily contains infinitely many affine copies of any finite set of integers. As such, it might be viewed as a relative version of Szemer\'edi's theorem \cite{SZ} on the existence of long arithmetic progressions in dense subsets of the integers.
\\\\
Another fundamental result in this area is the multi-dimensional extension of Szemer\'edi's theorem originally proved by Furstenberg and Katznelson \cite{FK}. It states that if $A\subs\Z^d$ is of positive upper density then $A$ contains an affine copy of any finite set $F\subs\Z^d$. The proof in \cite{FK} uses ergodic methods however a more recent combinatorial approach was developed by Gowers \cite{Go1} and also independently by Nagel, R\"odl and Schacht \cite{Ro}.
\\\\
It is natural to ask if a multi-dimensional extension of the result of Green and Tao, or alternatively if a relative version of the Furstenberg-Katznelson theorem can be established. In fact, this question was raised already in \cite{T} where the existence of arbitrary constellations among the Gaussian primes was shown. A partial result was obtained earlier by the first two authors \cite{CoMa}, where it was proved that relative dense subsets of $\PP^d$ contain an affine copy of any finite set $F\subs\Z^d$ which is in \emph{general position}, in the sense that each coordinate hyperplane contains at most one point of $F$. Recently, the existence of arbitrary constellations in relative dense subsets of $\PP^d$ has been shown by Tao and Ziegler \cite{TZ}, based on an \emph{infinite} number of linear forms conditions to obtain a weighted version of the Furstenberg correspondence principle and a short proof by Fox and Zhao [4] has been obtained afterwards using sampling arguments.
\\\\
Both of the above proofs however rely on full force of the results of Green, Tao and
Ziegler developed in \cite{GT2, GT3, GT4} on the asymptotic number of prime solutions for systems linear
equations, a monumental work encompassing more than 200 pages. Thus it may be of interest to obtain an essentially self-contained, combinatorial proof using only the pseudo-randomness properties of the almost primes encapsulated by the so-called linear forms condition of Green and Tao.
\\\\
In this note we provide such an argument, based on developing a hypergraph removal lemma for weighted hypergraphs. A crucial new feature is that, in our settings, there are weights attached possibly to any lower dimensional  edge, and hence known transference arguments \cite{CFZ, GT1, T},  do not seem to apply. Our approach is different, we are not trying to remove the weights and hence to reduce the problem to previously known results, but to extend the proof of the hypergraph regularity and removal lemmas directly to the weighted setting.
\\\\
Finally, let us mention that the methods of [20], and [4] do not provide bounds, while from our approach one can
extract quantitative statements. The bounds, though recursive, are rather weak, iterated tower-exponential
type. Also, as we rely only on sieve-techniques, our approach is somewhat flexible, e.g. it can be modified to count the number of ``small" copies of a finite set $F$, of size $N^\eps$, in a set $A\subs [1,N]^d\cap \PP^d$ of positive relative density, however do not pursue such extensions here.
\\\\
\subsection{Main results.} Let us recall that a set $A\subs \PP^d$ is of positive relative upper density if
\[\limsup_{N\to\infty}\ \frac{|A\cap \PP_N^d|}{|\PP_N|^d}>0,\]
where $\PP_N$ denotes the set of primes up to $N$, and $|A|$ stands for the cardinality of a set $A$. If $F\subs \Z^d$ is a finite set, we say that a set $F'$ is an affine copy of $F$, or alternatively that $F'$ is a constellation defined by $F$, if
\[F'=x +t\cdot F =\{x+ty:\ y\in F\}.\]
We  call $F'$ \emph{non-trivial} if $t\neq 0$. Our main result is the following.

\begin{thm}\label{MainThm1} If $A$ is a subset of $\PP^d$ of positive upper relative density, then $A$ contains infinitely many non-trivial affine copies of any finite set $F\subs\Z^d$.
\end{thm}

\noindent Note that it is enough to show that the set $A$ contains at least one non-trivial affine copy of $F$, as deleting the set $F$ from $A$ will not affect its relative density. Also, by lifting the problem to a higher number of dimensions, it is easy to see that one can assume that $F$ forms the vertices of a $d$-dimensional simplex. Indeed, let $F=\{0,x_1,\ldots,x_k\}$, choose a set of $k$ linearly independent vectors $\{y_1,\ldots,y_k\}\subs\Z^k$, and define the set $\De:=\{0,(x_1,y_1),\ldots,(x_k,y_k),z_{k+1},
\ldots,z_{k+d}\}\subs\Z^{k+d}$ such that the vectors of $\De\backslash\{0\}$ form a basis of $\R^{k+d}$. If the set $A'=A\times \PP^k$ contains an affine copy of $\De$ then clearly $A$  contains an affine copy of the set $\pi (\De)\supseteq F$, where $\pi:\R^{d}\times \R^{k} \to \R^d$ is the natural orthogonal projection.
\\\\In the case when $\De\subs\Z^d$ is a $d$-dimensional simplex, we prove a quantitative version of Theorem \ref{MainThm1}. To formulate it we define the quantity
\eq\label{l-def}
l(\De):= \sum_{i=1}^d |\pi_i (\De)|,
\ee
$\pi_i:\R^d\to\R$ being the orthogonal projection to the $i$-th coordinate axis.

\begin{thm}\label{MainThm2} Let $\al>0$ and let $\De\subs\Z^d$ be a $d$-dimensional simplex. There exists a constant\\ $\de=\de(\al,\De)>0$ such that for any $N>1$ and any set $A\subs\PP_N^d$ such that $|A| \geq \al\, |\PP_N|^d$, the set $A$ contains at least $\,\de N^{d+1}\,(\log\,N)^{-l(\De)}$ affine copies of the simplex $\De$.
\end{thm}

\noindent Note that in Theorem 1.2 we do not require the copies of $\De$ to be non-trivial. Thus, without loss of generality, $N$ can be assumed to be sufficiently large with respect to $\al$ and $\De$. It is clear that Theorem 2 implies Theorem 1 as the number of trivial copies of $\De$ in $A$ is at most $N^d\,(\log\,N)^{-d}$.
\\\\
To see why the above lower bound is meaningful, note that there are $\approx N^{d+1}$ affine copies  of $\De$ in $[1,N]^d$, and for a fixed $i$ the probability that all the $i$-th coordinates of an affine copy $\De'$ are primes is roughly $(\log N)^{-|\pi_i(\De)|}$. Thus if the prime tuples behave randomly, the probability that $\De'\subs \PP_N^d$ is about $(\log N)^{-l(\De)}$.\\

In the contrapositive, Theorem 1.2 states that if a set $A \subseteq \mathbb{P}_N^d$ contains at most $\delta N^{d+1}(\log N)^{-l(\De)}$ affine copies  of $\De$ then its relative density is at most $\al$. To prove Theorem \ref{MainThm2} one formulates a statement involving a pseudo-random measure $\nu=\nu^{(n)}:[1,N] \rightarrow \R_+$, originally introduced in \cite{GT1}.

\subsection{The Green-Tao measure and the linear forms condition.} As previously indicated, we use the pseudo-random measure $\nu$ defined in \cite{GT1}, and (a slight variant of) the so-called linear forms condition, which we recall here briefly.\\\\
Let $w$ be a sufficiently large number
and let $W=\prod_{p\leq w} p$ be the product of primes up to $w$. For given $b$ relative prime to $W$ define the modified von Mangoldt function $\bar{\La}_b: \Z \rightarrow \R_{\geq 0}$ by
\[ \bar{\La}_b (n)= \left\{ \begin{array}{ll}
         \frac{\phi(W)}{W}\,\log (Wn+b) & \mbox{if $Wn+b$ is a prime.}\\
         0 & \mbox{otherwise}.\end{array} \right. \]
Here $\phi$ is the Euler function. Note that by Dirichlet's theorem on the distribution of primes in residue classes one has that $\sum_{n\leq N} \bar{\La}_b (n)=N (1+o(1))$. A crucial fact is that the function $\bar{\La}_b$ is majorized by the so-called Goldston-Yildirim divisor sum \cite{GT1, GY1}
\[\La_R(n)=\sum_{d|n,d\leq R} \mu(d)\,\log(R/d),\] $\mu$ being the Mobius function and $R=N^{d^{-1}2^{-d-5}}$. For given small parameters $0<\eps_1<\eps_2<1$ (whose exact values will be specified later), recall the Green-Tao measure
\[ \nu_b(n)= \left\{ \begin{array}{ll}
         \frac{\phi(W)}{W}\,\frac{\La_R(Wn+b)^2}{\log\,R} & \mbox{if $\eps_1N\leq n\leq\eps_2 N$},\\
         1 & \mbox{otherwise}.\end{array} \right. \]
Note that $\nu_b(n)\geq 0$ for all $n$ , and also it is easy to see that
for $N$ sufficiently large, one has that
\eq\nu_b(n)\geq d^{-1}2^{-d-6}\,\bar{\La}_b (n)\ee for all $\eps_1 N\leq n\leq \eps_2 N$.
This is trivial unless $Wn+b$ is a prime, and in that case, since $\eps_1 N>R$, $\La_R(Wn+b)=\log\,R=d^{-1}2^{-d-5}\log\,N$.\\\\

Let us briefly recall its definition and the pseudo-randomness properties: the so called linear forms condition we'll need in the proof. This is not exactly the same formulation as given in \cite{GT1}, however the proof works without any changes.
\begin{thmA} [Linear forms condition, \cite{GT1}]
Let $N$, $W$ and the measures $\nu_b$ be as above, and let $m_0,t_0,k_0\in\mathbb{N}$ be small parameters. Then the following holds.
\\\\
For given $m\leq m_0$ and $t\leq
t_0$, suppose that $\{l_{i,j}\}_{1\leq i\leq m, 1\leq
j\leq t}$ are arbitrary rational numbers with numerator and denominator at
most $k_0$ in absolute value, and that $\{b_i\}$ are arbitrary numbers relative prime to $W$. If the linear forms
\[L_i(x)=\sum_{j=1}^t l_{i,j}\,x_j,\] are non-zero and pairwise linearly independent over the rationals then
\eq\label{linform}
\mathbb{E}\,\left(\prod_{i=1}^m \nu_{b_i}(L_i(x));\ x\in\Z_N^t\right)=1+o_{\nw;\,m_0,t_0,k_0}(1),
\ee
where the $o(1)$ term is independent of the choice of the $b_i$'s.
\end{thmA}

\noindent In the above formula the linear forms $L_i(x)$ are considered as acting on $(\Z/N\Z)^t$ and the error term $o_{\nw;\,m_0,t_0,k_0}(1)$ denotes a quantity that tends to 0 as both $N\to\infty$ and $W\to\infty$ for any fixed choice of $m_0,t_0,k_0$, i.e. the error term can be made smaller than any given $\eps>0$ by choosing both $W$ and $N$ sufficiently large with respect to $m_0,t_0,k_0$, and $\eps$. This refinement is important in obtaining the quantitative lower bound in Theorem 1.2, see also the remarks in \cite{GT1}  (Sec.11). The error terms are independent of the choice of $b_i$'s, so we will write $\nu$ for $\nu_{b_i}$ for all $1 \leq i \leq d$ for simplicity of notations.
\\\\
With the aid of this measure, we define the weight of a finite set $S \subseteq \Z^d$ as

\eq\label{1.3.3}
w(S):= \prod_{i=1}^d \prod_{y \in \pi_i(\Delta)} \nu(y).
\ee
If $S=\{y\}$ we will write $w(y):=w(\{y\}).$ The point is that if $Wy+b \in \PP_N^d$ then

\eq\label{1.3.4}
w(y) \approx (\log N)^d.
\ee
The implicit constant depends only on $W$ which we will choose eventually sufficiently large but independent of $N$. Then also
\eq\label{1.3.5}
w(\Delta) \approx (\log N)^{l(\De)}
\ee\\
for a non-degenerate simplex $\De$ with $W\Delta +b\subseteq A \subseteq \PP_N^d$. Thus identifying $[1,N]$ with $\Z_N=\Z/N\Z$, it is a straight forward argument to show (see Section 5) that Theorem 1.2 follows from

\begin{thm}\label{Thm1.3}
Let $\Delta=\{v_0, \ldots, v_d\} \subseteq \Z^d$ be a d-dimensional simplex and let $\al>0$. Then there exists a constant $\de=\de(\al,\De)>0$ such that the following holds. If $N$ is a large prime and $A \subseteq \Z_N^d$ is a set for which

\eq\label{1.3.6}
\E_{x \in \Z_N^d, t \in \Z_N} \prod_{i=1}^d \1_A(x+tv_i)w(x+t\Delta) \leq \delta,
\ee
then
\[
\E_{x \in \Z_N^d}\1_A(x)w(x) \leq \al+o_{N,W \rightarrow \infty, \Delta}(1).
\]
\end{thm}

Note that the error term $o_{N,W \rightarrow \infty, \Delta}(1)$ can be omitted if both $N$ and $W$ are chosen sufficiently large with respect to $\al>0$.

\subsection{A Removal Lemma for weighted hypergraph systems.}

We will use the construction of a weighted hypergraph associated to a set $A \subseteq \Z_N^d$ and a simplex $\Delta=\{v_0, \dots, v_d\}$, given in \cite{T}. This will enable us to reduce Theorem 1.3 to prove a removal lemma for weighted hypergraphs.
\\\\
Let $J=\{0,1,\dots,d\}, \overline{\mathcal{H}}:=\{e:\ e \subseteq J\}$, $\mathcal{H}_d:=\{e:\ e \subseteq J, |e|=d\}$.
For a set $e \in \overline{\mathcal{H}}$, let $V_e=\mathbb{Z}_N^e=\prod_{j \in e}\mathbb{Z}_N$. Identify $V_e$ as the subspace of elements $x=(x_0,\ldots,x_d) \in V_{J}$ such that $x_j=0$ for all $j \notin e$ and let $\pi_e:V_{J} \rightarrow V_e$ denote the natural projection. For $e=\{j\}$ we write $V_j:=V_{\{j\}}$.

\begin{remark}
For a given $e \in \HH$, $|e|=d$ we think of a point $x_e\in V_e$ as a $d$-simplex (or clique) with vertices $\{x_j:j \in e\}$, and then a set $G_e\subs V_e$ may be viewed as a $d$-regular hypergraph with vertex sets $V_j$ $(j\in e)$. Similarly, a point $x \in V_{J}$ represents a $d+1$-simplex with faces $x_e:=\pi_e(x)$, for $e\in\HH_d$.
\end{remark}

\noindent
For a given $e\subseteq J$ define the $\sigma$-algebra $\mathcal{A}_e=\{\pi_e^{-1}(F):\ F \subseteq V_e\}.$ We assign a $d+1$-partite $d$-regular hypergraph to a given set $A \subs\Z_N^d$ as follows. For $e=J\backslash \{j\},$ define
\eq\label{1.4.1}
E_e=\{x \in V_{J}: \sum_{i=0}^dx_i(v_i-v_j) \in A\}.
\ee
Note that $E_e \in \mathcal{A}_e$ as the expression in \eqref{1.4.1} is independent of the coordinate $x_j$. We define $G:=\bigcap_{e\in\mathcal{H}_d} E_e\subs V_J$ to be the hypergraph associated to $A$. By our point of view the points $x\in G$ are $d+1$-simplices (or cliques) with faces $x_e=\pi_e(x)$ belonging to the hypergraphs $G_e=\pi_e(E_e)$.

\begin{ex}\label{ex1.1} Let $\De=\{v_0,v_1,v_2\}$ with $v_0=(0,0),\ v_1=(0,2),\ v_2=(1,1)$ and let $A\subs \Z_N^2$. Then by \eqref{1.4.1} we have
\[G_{01}=\{(x_0,x_1):\ (-x_0+x_1,-x_0-x_1)\in A\},\
G_{02}=\{(x_0,x_2):\ (-2x_0-x_2,x_2)\in A\},\]
\[G_{12}=\{(x_1,x_2):\ (2x_1+x_2,x_2)\in A\}, \quad\text{and}\]
\[G=\{(x_0,x_1,x_2):\ (-x_0+x_1,-x_0-x_1)\in A,\,(-2x_0-x_2,x_2)\in A,\,(2x_1+x_2,x_2)\in A\}\]
is the 3-partite graph associated to the set $A$.
\end{ex}
\noindent
To see how this hypergraph is related to detecting affine copies of the simplex $\De$ in a set $A\subs\Z_N^d$, consider the map $\Phi:\Z_N^{d+1} \rightarrow \Z_N^d\times\Z_N$, defined by
\eq \label{1.4.4}
\Phi(x)=(\sum_{i=0}^dx_iv_i,\, -\sum_{i=0}^dx_i) =:(y,t).
\ee
By \eqref{1.4.1} and \eqref{1.4.4} we have that $x \in E_e$ for $e=J \backslash \{j\} $ if and only if $y+tv_j \in A$, thus $x \in G=\bigcap_{e \in \HH_d}E_e$ exactly when $y+t\Delta \subseteq A$. Since $\Phi$ is one to one (as $v_1-v_0, \ldots, v_d-v_0$ are linearly independent), this gives a parametrisation of all affine copies of $\Delta$ contained in $A$.
\\\\
We now define a family of functions $\mu_e:V_e \rightarrow \R_+,\  (e\in\HH_d)$ and $\nu_e:V_e \rightarrow \R_+,\ (e\in \overline{\HH})$. For $e \in \HH_d,\ e=J \backslash \{j\}$
and $1\leq k \leq d$, let
\eq\label{1.4.2}
L_e^k(x):=\sum_{i=0}^dx_i(v_i^k-v_j^k)
\ee
where $v_i^k$ denotes the $k^{th}$-coordinate of the vector $v_i$. A crucial property is that the forms $L_e^k(x)$ are pairwise linearly independent over $\Z_N$ for $N$ being a sufficiently large prime. Indeed, for $e=J \backslash \{j\}$, the coefficients of these forms form the column vectors of the $d\times d$ matrix $D_j$ with row vectors $v_i-v_j\ (0\leq i\leq d,\,i\neq j)$. If any two columns would be linearly dependent over $\Z_N$, then $N$ would divide the determinant of $D_j$. Since the row vectors are linearly independent and has integer entries this is not possible for large enough $N$, which we will always assume.
\\\\
We define the functions
\eq\label{1.4.2.5}
\mu_e(x):=\prod_{k=1}^d \nu(L_e^k(x)).
\ee
We partition the family of forms
\[
\mathcal{L}:=\{L_e^k: |e|=d, 1 \leq k \leq d\}
\]
\\
according to which coordinates they depend on. For this, define the support of a linear form\\ $L(x)=\sum_{k=1}^d a_kx_k\ $ as $supp(L)=\{k:a_k \neq 0\}.$
For a given $e \in \overline{\HH}$, let
\eq \label{1.4.3}
\nu_e(x)=\prod_{\substack{\text{supp}(L)=e\\L \in \mathcal{L}}}\nu(L(x))
\ee
\hspace{1mm}
with the convention that $\nu_e \equiv 1$ if $\{L:\ supp(L)=e\}=\emptyset.$
\\\\
\begin{ex} In case of the hypergraph $G$ constructed in example \ref{ex1.1} one obtains the following weights and measures on the edges and vertices:
\[\mu_{01}(x_0,x_1)=\nu(-x_0+x_1)\nu(-x_0-x_1),\,\mu_{02}(x_0,x_2)=\nu(-2x_0-x_2)\nu(x_2),\,\mu_{12}(x_1,x_2)=\nu(2x_1+x_2)\nu(x_2)\]
\[\nu_{01}(x_0,x_1)=\nu(-x_0+x_1)\nu(-x_0-x_1),\,\nu_{02}(x_0,x_2)=\nu(-2x_0-x_2),\,\nu_{12}(x_1,x_2)=\nu(2x_1+x_2),\]
$\nu_2(x_2)=\nu(x_2)$, with the convention that $\nu_0(x_0)=\nu_1(x_1)\equiv 1$. The family of linear forms defining the weights is
\[\mathcal{L}=\{-x_0+x_1,-x_0-x_1,-2x_0-x_2,2x_1+x_2,x_2\}.\]
\end{ex}
\bigskip
\noindent
If $\Delta=\{v_0,\ldots,v_k\}$ is in general position, that is if $v_i^k \neq v_j^k$ for all $i \neq j$ and $k$, then $supp(L_e^k$) $=e$ for all $e \in \HH_d$, hence
\[
\mu_e(x)=\nu_e(x)=\prod_{k=1}^d \nu(L_e^k(x)).
\]
In general, we have
\eq\label{1.4.3.5}
\mu_e(x)=\prod_{f\subs e}\nu_f(x).
\ee
\noindent
This is because if for some $e'=J\backslash\{j'\}$ we have that $supp(L_{e'}^k)\subs e$, then by \eqref{1.4.2} $v_{j'}^k=v_j^k$ and $L_{e'}^k=L_e^k$. This ensures that the family $\mathcal{L}$ consists of pairwise-linearly independent forms. Also for $e=J\backslash\{j\}$, $\mu_e(x)$ is well-defined on $V_e$ as it depends only on the variables $\{x_i:\ i\neq j\}$.
We will refer the functions $\nu_e$ and $\mu_e$ as \emph{weights} and \emph{measures}. To emphasize this point of view we will often use the integral notation and write
\[\int_{V_e} F(x)\,d\mu_e(x):= \E_{x\in V_e} F(x)\mu_e(x),\ \ \textit{and}\ \ \mu_e(S):= \E_{x\in V_e} \1_S(x)\,\mu_e(x),\]
for functions $F:V_e\to\R$ and sets $S\subs V_e$.
\\\\
The weight system $\{\nu_e\}_{e\subs J}$ has very strong pseudo-randomness properties. First, for any hypergraph $\mathcal{G}\subs \bar{\HH}$ one has that: $\ \EE_{x\in V_j}\prod_{e\in\mathcal{G}}\nu_e(x_e)=1+o_{\nw}\,$. This is immediate from \eqref{1.4.3} and the fact the family $\LL$ is pairwise linearly independent. We will refer to this property as \emph{order-1 pseudo-randomness} of the weight system $\{\nu_e\}_{e\subs J}$.
However much more is true, to formulate it recall the notion of a hypergraph bundle, see also section 4.2. Let $K$ be a finite set and let $\pi:K\to J$ be an onto map. Define that graph $\bar{H}_{\pi}:=\{g\subs K:\ |\pi(g)|=|g|\}$. If $m$ is an integer such that $m_j:=|\pi^{-1}(j)|\leq m$ for all $j\in J$ then we say that the hypergraph bundle $\bar{H}_{\pi}$ is of order $m$. For given $k\in K$ set $V_k:=V_{\pi(k)}$, intuitively this means that we sample independently $m_j$ points from each vertex set $V_j$. Thus for $g\in \bar{H}_{\pi}$, $e=\pi(g)$ one has $V_g=V_e$ and can define the weights $\bar{\nu}_g(x_g):=\nu_e(x_g)$ on the points $x_g=(x_k)_{k\in g}$. Note that by \eqref{1.4.3} the weight $\bar{\nu}_g$ is defined via linear forms $L$ for which $supp(L)=g$ and hence the family of linear forms $\LL_{\pi}$ defining the weight system $\{\bar{\nu}_g\}_{g\in \bar{H}_{\pi}}$ consists of pairwise linearly independent linear forms in the $x_K=(x_k)_{k\in K}$ variables. Thus by the linear forms condition one has for any hypergraph $\mathcal{G}\subs \bar{H}_{\pi}$ of order $m$

\[\EE_{x_K\in V_K}\prod_{g\in \mathcal{G}} \bar{\nu}_g(x_g) = \EE_{x_K\in V_K} \prod_{g\in\mathcal{G}}\nu_{\pi(g)}(x_g) = 1+ o_{\nw;\,d,m}(1).\]
We will refer to the above property \emph{order-m pseudo-randomness} of the weight system $\{\nu_e\}_{e\subs J}$. For $m=2$ this is essentially equivalent to the linear forms condition on weighted hypergraphs given in \cite{T} (see def. 2.8). What is needed in our proof is that the weight system $\{\nu_e\}$ is pseudo-random at order $m$ where $m=m(\al)$ is a fixed but sufficiently large number depending on the density $\al$.
\\\\
The measures $\mu_e$ are closely related to the weights $w$. Indeed, for $e=J \backslash \{j\}$ and $(y,t)=\Phi(x)$, one has

\eq \label{1.4.5}
L_e^k(x)= \sum_{i=0}^d x_i(v_i^k-v_j^k) = \pi_k(y-tv_j),
\ee
where $\pi_k$ is the orthogonal projection to the $k^{th}$ coordinate axis. This implies that

\eq \label{1.4.6}
\mu_e(x)=\prod_{k=1}^d \nu(L_e^k(x))=w(y-tv_j)
\ee
and also
\eq \label{1.4.7}
\mu_{J}(x):=\prod_{L \in \mathcal{L}} \nu (L(x))=w(y-t\Delta).
\ee

Thus assumption \eqref{1.3.6} in Theorem \ref{Thm1.3}  in Theorem 1.3 translates to

\eq \label{1.4.8}
\E_{x \in V_{J} }\prod_{e \in \HH_d}\1_{E_e}(x)\mu_{J}(x)=\E_{(y,t) \in \Z_N^{d+1}}w(y-t\Delta) \leq \delta.
\ee

On the other hand, if $M=\{x \in V_{J}: x_0+\cdots +x_d=0\}$ then $x \in  M \bigcap_{e \in \HH_d}E_e$ if and only if $\Phi(x)=(y,0)$ with $ y \in A$, thus by \eqref{1.4.4}, \eqref{1.4.6}

\eq \label{1.4.9}
\E_{y \in A}\ w(y)= \E_{x \in M} \prod_{e \in \HH_d} \1_{E_e}(x)\mu_{e'}(x)
\ee
for any fixed $e' \in \HH_d$. Thus it is easy to see that Theorem 1.3 follows from a \emph{removal lemma} for weighted hypergraphs, which we first recall in the un-weighted case (where $\nu_f \equiv 1$ for all $f$).  See also \cite{Go1,Ro,T}.

\begin{thmB}\label{removal} (Simplex Removal Lemma\cite{T1}).
\noindent Let $E_e\in \A_e$ be given for $e\in \HH_d$, and let $\eps>0$. Also let $\mu$ and $\mu_e$ denote the normalized counting measures on $V$ and $V_e$. There exists $\de=\de(\eps)>0$ and for every index set $e\in\HH_d$ there exists a set $E'_e\in \A_e$ such that the following holds.
\\\\If
\[ \label{removal0}
\E_{x\in V_J} \prod_{e\in\HH_d} \1_{E_e}(x)\,d\mu(x) \leq \de,
\]
then
\[ \label{removal1}
\prod_{e\in\HH_d} \1_{E'_e}(x)=0\ \ \ \ \textit{for all}\ \ \ x\in V,\quad\textit{and}
\]

\[\label{removal2}
\mu_e (E_e\backslash E'_e) \leq \eps.
\]
\end{thmB}

\medskip

\noindent
Now let us observe the key properties of the family of linear forms $\mathcal{L}$ and introduce some terminology which we will use throughout the paper. Recall that for $e=J \backslash \{j\} $,  $e'=J \backslash \{j'\}$ we have that $supp(L_{e'}^k) \subseteq e$ if and only if $v_j^k=v_{j'}^k$ and that is equivalent to $L_{e'}^k=L_e^k$. We call such a family $\mathcal{L}$ \textbf{\emph{well-defined}}. Since for a given $e \in \HH_d$, the forms $\{L^k_e: 1 \leq k \leq d\}$ are linearly independent,
$\mathcal{L}$  is a \textbf{\emph{pairwise linearly independent}} family of linear forms. Also for $x \in M=\{x_0+\cdots +x_d=0\}$ we have $L_e^k(x)=L^k_{e'}(x)$ for all $e,e' \in \HH_d, 1 \leq k \leq d$. We call a family of linear forms $\mathcal{L}=\{L_e^k: e \in \HH_d, 1 \leq k \leq d\}$ satisfying this property\textbf{ \emph{symmetric}}.
\\\\
We now formulate a hypergraph removal lemma with respect to the weight system $\{\nu_e\}_{e\subs J}$ defined in \eqref{1.4.2}-\eqref{1.4.3}, and show that it implies Theorem 1.3.
\newpage

\begin{thm} \label{Thm1.4} (Weighted Simplex Removal Lemma)
Let $\{\nu_e\}_{e \subseteq J}, \{\mu_e\}_{e \subseteq J}$ be a system of weights and measures associated to a well-defined, pairwise linearly independent and symmetric family of linear forms  $\mathcal{L}$, as defined in \eqref{1.4.2}-\eqref{1.4.3}. Let $E_e \subseteq \mathcal{A}_e,\ g_e:V_e \rightarrow [0,1]$ be given for $e \in \HH_d$. Then for a given $\eps>0$ there exists an $\de=\de(\eps)>0$ such that the following holds. If

\eq \label{1.4.10}
\E_{x \in V_{J}} \prod_{e \in \HH_d}\1_{E_e}(x)\mu_{J}(x) \leq \delta
\ee
\\
then there exists a well-defined, symmetric family of linear forms $\tilde{L}=\{\tilde{L}_e^k: e \in \HH_d, 1 \leq k \leq d\}$, such that the associated system of weights and measures  $\{ \tilde{\nu}_e\}_{e \subseteq J}, \{ \tilde{\mu}_e\}_{e \subseteq J}  $  satisfies

\eq \label{1.4.11}
\E_{x \in V_{J}} \prod_{e \in \HH_d} \1_{E_e}(x)\tilde{\mu}_{J}(x)=\E_{x \in V_{J} }\prod_{e \in \HH_d}\1_{E_e}(x)\mu_{J}(x)+o_{N,W\to\infty}(1)
\ee
and for all $e \in \HH_d$,
\eq \label{1.4.12}
\E_{x \in V_{J}} g_e(x)\tilde{\mu}_e(x)= \E_{x \in V_{J}} g_e(x)\mu_e(x)+o_{N,W\to\infty}(1).
\ee
\\
In addition there exist sets $E'_e \in \mathcal{A}_e$ such that

\eq  \label{1.4.13}
\bigcap_{e \in \HH_d} \,(E_e\cap E_e')\, = \emptyset
\ee
and for all $e \in \HH_d$

\eq  \label{1.4.14}
\E_{x \in V_{J}} \1_{E_e \backslash E_e'}(x) \tilde{\mu}_e(x) \leq \eps +o_{N,W\to\infty}(1).
\ee
\end{thm}
\medskip
\noindent
Naturally one would like to establish \eqref{1.4.13}-\eqref{1.4.14} for $\{\tilde{\mu}_e\}=\{ \mu_e \}$ as that would easily imply Theorem \ref{Thm1.3} and hence our main result Theorem \ref{MainThm2}.  Formulas \eqref{1.4.11}-\eqref{1.4.12} mean that the measures $\tilde{\mu}_{J},\tilde{\mu}_e$ are small perturbation of the measures $\mu_{J},\mu_e$ with respect to the functions $g=\prod_{e \in \HH_d}\1_{E_e}$ and $g_e$. The reason why we have to perturb the measures is the existence of weights $\mu_e$ on lower dimensional edges $1<|e|<d$ when the configuration $\delta$ is not in general position. Removing these weights does not seem amenable to known `` transference arguments", see \cite{CFZ,T}.
\\\\
Let us remark that if one drops the ``symmetry" requirement $\mu_e(x)=\mu_{e'}(x),\ e,e'\in\HH_d,\,x\in M$, then one could formulate a simpler, purely combinatorial form of the weighted removal lemma, requiring only that the system of weights $\{\nu_e\}_{e\subs J}$ to be pseudo-random at an order $m=m(\eps)$, without referring to the underlying linear forms. In that case one could construct the perturbed weights from the weights on certain hypergraph bundles, emerging in our energy increment process. However, as we see below, we need to make sure that our weights and measures are invariant under certain linear transformations of the variables and have to go through the more tedious route of extending our underlying family of linear forms by introducing new variables, considered as parameters.
\\

{\bf{Proof}} [Theorem 1.4 implies Theorem 1.3]\\
\\
Let $\al>0$ and let $A$ be a subset of $\Z_N^d$. Set $\eps:=\al/(d+2)$ and let $G=\bigcap_{e\in\HH_d}E_e$ be the hypergraph associated to the set $A$ as defined above. For a given $e' \in \HH_d$ define the function $g_{e'}:V_{e'} \rightarrow [0,1]$ as follows. Let $\phi_{e'}:V_{e'} \rightarrow M$ be the inverse of the projection map $\pi_{e'}: V_{J} \rightarrow V_{e'}$ restricted to $M$, and for $z \in V_{e'}$ let
\[
 g_{e'}(z):= \prod_{e \in \HH_d}  \1_{E_{e}}(\phi_{e'}(z)).
 \]
Applying Theorem 1.4 to the system of weights $\{\nu_e\}$, the sets $E_e$ and the functions $\{g_e\}$ gives a constant $\de=\de(\eps)>0$, a system of measures $\tilde{\mu}_e$ and sets $E_e' \in A_e$ satisfying \eqref{1.4.11}-\eqref{1.4.14}.
\\\\
Now assume, as in Theorem \ref{Thm1.3}
\[
\E_{y,t} w(y-t\De)= \E_{x \in V_J} \prod_{e \in \HH_d}\1_{E_e}(x)\ \mu_{J}(x) \leq \delta.
\]
\noindent
By \eqref{1.4.4} we have that $x \in M\bigcap_{e \in \HH_d} E_e$ if and only if $\Phi(x)=(y,0)$ with $y \in A$. Moreover in that case $w(y)=\mu_e(x)$ for all $e \in \HH_d$ by \eqref{1.4.6}, thus for any given $e' \in \HH_{d}$

\begin{align*}
\E_{y \in \Z_N^d}\1_A(y)w(y)
=\E_{x \in M}\prod_{e \in \HH_d} \1_{E_e}(x)\mu_{e'}(x)
&=\E_{z \in V_{e'}}g_{e'}(z)\mu_{e'}(z) \\
&=\E_{z \in V_{e'}}g_{e'}(z)\tilde{\mu}_{e'}(z) +o_{N,W\to\infty}(1)\\
&=\E_{x \in M} \prod_{e \in \HH_d}\1_{E_e}(x)\tilde{\mu}_{e'}(x)+o_{N,W\to\infty}(1).
\end{align*}
By \eqref{1.4.13} clearly $ \prod_{e \in \HH_d}\1_{E_e}(x) \leq \sum_{e \in \HH_d}\1_{E_e \backslash E_e'}(x)$. Then the symmetry of the measures $
\tilde{\mu}_e$ (i.e. $\tilde{\mu}_e(x)=\tilde{\mu}_{e'}(x)$ for $x \in M$), \eqref{1.4.14} and the fact that $\1_{E_e \backslash E_e'}$ is constant on the fibers $\pi_e^{-1}(x)$ imply

\begin{align*}
E_{x \in M} \prod_{e \in \HH_d}\1_{E_e}(x)\tilde{\mu}_{e'}(x)
&\leq \sum_{e \in \HH_d}\E_{x \in M}\1_{E_e \backslash E'_e}(x)\tilde{\mu}_{e'}(x)\\
&=\sum_{e \in \HH_d}\E_{x \in V_{\J}} \1_{E_e \backslash E'_e}(x)\tilde{\mu}_e(x) \\
&\leq (d+1)\,\epsilon(\delta)+o_{N,W\to\infty}(1)\leq\al,
\end{align*}
\\
choosing $N,W$ sufficiently large with respect to $\al$ and Theorem \ref{Thm1.3} follows.
$\hspace{.5in}\Box$

\bigskip

\subsection{Weighted box norms and hypergraph regularity.}
The known proofs of the Simplex Removal Lemma rely on the so-called Hypergraph Regularity Lemma and the associated Counting Lemma \cite{Go1,Ro,T1}, and in particular the notion of a regular or pseudo-random hypergraph. This can be defined in different ways, we use a variant of Gowers's box norms \cite{Go1} adapted to our settings.
\\\\
Let $e \in \mathcal{H}_d$ be fixed. For a given $\omega \in \{0,1\}^e$ (i.e. $\omega: e \rightarrow \{0,1\}$), define the orthogonal projection
$\omega_e : V_e \times V_e \rightarrow V_e$ by
\eq\label{proj}
\omega_e(x_{e},q_{e})_i = \begin{cases} x_i & \text{if}\quad \omega_i =0 \\ q_i &\text{if}\quad \omega_i=1 \end{cases}
\ee
\\
for $i\in e$, and the weighted box norm of a function $F:V_e \rightarrow \R$, using the notation $x_f:=\pi_f(x)$ for $f\subs J$, as

\eq\label{boxnorm}
\left\| F \right\|_{\Box_{\nu_e}}^{2^d}=\E_{x,q \in V_e}\prod_{\omega \in \{0,1\}^e}F(\omega_e(x,q))\prod_{f \subs e}\ \prod_{\omega \in \{0,1\}^f}\nu_f(\omega_f(x_f,q_f)).
\ee
Note that if $\nu_f \equiv 1$ for all $f \subseteq e$, then $\left\|F \right\|_{\Box_{\nu_e}}=\left\| F\right\|_{\Box}$ is the usual box norm.

\begin{ex} Let $e=(1,2)$ and $F:V_e=V_1\times V_2\to\R$. Then
\begin{align*}\left\| F \right\|_{\Box_{\nu_e}}^4
&= \E_{x,q\in V_e}\ F(x_1,x_2)F(x_1,q_2)F(q_1,x_2)F(q_1,q_2)\\
&\times\nu_e(x_1,x_2)\nu_e(x_1,q_2)\nu_e(q_1,x_2)\nu_e(q_1,q_2)\,
\nu_1(x_1)\nu_2(x_2)\nu_1(q_1)\nu_2(q_2).
\end{align*}
\end{ex}
\medskip
\noindent
The points $\omega_e(x,q)$ and $\omega_f(x_f,q_f)$ may be viewed as the faces and edges of a $d$-dimensional octahedron $\mathcal{K}_d$ with vertices $\{x_j,q_j:\ j \in e\}$. The inner product in \eqref{boxnorm} represents the total weight of the octahedron obtained by multiplying the weights of all edges and vertices. The boxnorm itself is the weighted average of $F$ over all embeddings of the hypergraph $\mathcal{K}_d$.
\\\\
It is not hard to see that the $\Box_{\nu}$-norm is indeed a norm (for $d \geq 2$) and an appropriate version of the Gowers-Cauchy-Schwarz inequality holds (see the Appendix). The importance of this norm is that it controls weighted averages over $d+1$-dimensional simplices, which plays an important role in proving the counting lemma. More precisely one  has the following.

\begin{prop}\label{neumann}
(Weighted von Neumann inequality)
Let $F_e:V_e \rightarrow \R$ be a given functions, such that $|F_e|\leq 1$ for each $e\in\HH_d$. Then there is an absolute constant $C$ such that
\eq\label{neumannineq}
\big|\E_{x\in V_J} \prod_{e \in\mathcal{H}_d} F_e(\pi_e(x))\ \mu_J(x) \big| \leq C \min_{e \in \mathcal{H}_d}\left\| F_e \right\|_{\Box_{\nu_e}}+o_{N,W \rightarrow \infty}(1).
\ee
\end{prop}

\noindent The above inequality motivates the following
\begin{defin} \label{epsilon-reg}
Let $e \in \mathcal{H}_d$ and $\eps > 0$ be fixed and let $G_e\subs V_e$ be a $d$-regular hypergraph. We say that $G_e$ is $\eps$-regular with respect to the weight system $\{\nu_f\}_{f\subs e}$ if
\eq \label{epsilon-regeq}
\left\|\1_{G_e}-\mu_e(G_e)\ \1_{V_e}\right\|_{\Box_{\nu_e}} \leq \eps.
\ee
\end{defin}

\noindent It is easy to see from Proposition \ref{neumann} that if the sets $E_e\in\A_e$ are $\, c\eps\,$-regular for all $e \in \HH_d$ (\,$0<c\ll 1$), then Theorem \ref{Thm1.4} holds with $\tilde{\mu}_e=\mu_e$ and $\de=c\eps^{d+1}$. Indeed, writing $\ G_e=\pi_e(E_e)$, $1_{G_e}=\mu_e(G_e)+F_e\,$, and substituting this decomposition into the left side of \eqref{1.4.10} we get $2^{d+1}-1$ error terms each of which is bounded by $c'\eps^{d+1}$ (with $0<c'\ll 1$), and a main term of the form $\ \prod_{e\in\HH_d} \mu_e(G_e)\mu_J(V_J)$. The forms defining the measure $\mu_J$ are pairwise linearly independent and hence $\mu_J(V_J)=1+o(1)$. Thus by assumption \eqref{1.4.10} of Theorem \ref{Thm1.4}  $\mu_J(E_e)=\mu_e (G_e)(1+o_{N,W\to\infty}(1))\leq \eps$ (see section 2), for at least one $e\in \HH_d$ and the sets $E'_e:=\emptyset$, $E'_{e'}:=E_e$ ($e'\neq e$) satisfy the conclusion of Theorem \ref{Thm1.4}.
\\\\
In general the hypergraphs $G_e$ are not sufficiently regular and the bulk of our argument is to obtain a \emph{regularity lemma} in our weighted setting. This roughly means to partition the sets $G_e$ into sufficiently regular hypergraphs with respect to a system of measures $\tilde{\mu}_e$, which are small perturbations of the initial measures $\mu_e$.
Our proof is motivated by the iterative process described in \cite{T1} however we need to modify the entire argument because of the presence of weights on the lower dimensional edges. During the process we construct parametric families of weight systems $\{\nu_{q,e}\}_{e\in\bar{\HH},q\in\Omega}$ which for most values of the parameter $q$ will give rise to small perturbations of the initial system $\{\nu_e\}_{e\in\bar{\HH}}$.
\\\\
\subsection{Sketch of the proof in dimension 2.} We describe in detail the main elements of the proof in dimension $d=2$ (and illustrate some of the constructs on the configuration $\De=\{(0,0),(2,0),(1,1)\}$). We do this as in full generality both the notations and the arguments are quite complex, but the essential ideas can already be illustrated in this case.
\\\\
\emph{1. Energy increment.} Let $\eta>0$ and assume that for some edge $e$, say $e=(1,2)$, the graph $G_e=\pi_e(E_e)$ is not $\eta$-regular. This means
\eq\label{1.6.1}
\left\|F \right\|_{\Box_{\nu_e}} \geq \eta,
\ee
where $F=\1_{G_e}-\mu_e(G_e)\,\1_{V_e}$. In view of definition \eqref{boxnorm}, we may write
\eq \label{1.6.2}
\left\| F \right\|_{\Box_{\nu_e}}^4=\int_{V_e}\int_{V_e}\,  F(x)\,u_{q}^1(x_1) u_{q}^2(x_2)\nu_e(x_1,q_2)\nu_e(q_1,x_2)\,d\mu_e(x)\,d\mu_e(q) \geq \eta^4,
\ee
where $x=(x_1,x_2)$, $q=(q_1,q_2)$, $u^1_{q}(x_1)=F(x_1,q_2)$, and $u^2_{q}(x_2)=F(q_1,x_2)F(q_1,q_2)$.
We consider the variables $q$ as parameters and define the measures $\mu_{q,e}$, depending on $q$, by
\[
\mu_{q,e}(x):=\nu_e(x_1,q_2)\nu_e(q_1,x_2)\mu_e(x).
\]

\begin{ex} If $\De=\{(0,0),(2,0),(1,1)\}$ then $\nu_e(x)=\nu(2x_1+x_2)$, $\mu_e(x)=\nu(2x_1+x_2)\nu(x_2)$ and $\mu_{q,e}(x)=w(q,x)\mu_e(x)$ with
$w(q,x)=\nu(2x_1+q_2)\nu(2q_1+x_2)$. Note that both forms $2x_1+q_2$ and $2q_1+x_2$ depend on both the $x$ and $q$ variables, which holds in general by \eqref{1.4.3}.
\end{ex}
\noindent
The inner integral in \eqref{1.6.2} may be viewed as the inner product
\eq\label{1.6.3}
\Gamma (q):= \left\langle F, u_{q}^1 \times u_{q}^2 \right\rangle_{\mu_{q,e}}=\int_{V_e} F(x)\,u_{q}^1(x_1) u_{q}^2(x_2)\,d\mu_{q,e}(x),
\ee
on the Hilbert space $L^2(V_e,\mu_{q,e})$. Thus \eqref{1.6.2} translates to $\E_{q\in V_e}\,\Gamma(q)\,\mu_e(q)\geq \eta^4$ while using the linear forms condition it is easy to see that
\[\E_{q\in V_e}\,\Gamma(q)^2\,\mu_e(q) \leq \E_{q,x,y}\ \mu_{q,e}(x)\mu_{q,e}(y)\mu_e(q)=1+o_{N,W\to\infty}(1),\]
as the underlying linear forms defining $\mu_{q,e}(x)$ are pairwise linearly independent and depend on at least one of the $x$ variables. Alternatively, it also follows from the pseudo-randomness of the weight system $\{\nu_{12},\nu_1,\nu_2\}$, applied to the hypergraph bundle obtained by sampling the variables $x_1,y_1,q_1$ from $V_1$ and $x_2,y_2,q_2$ from $V_2$. Thus
\eq\label{1.6.4}
\Gamma (q)\gs \eta^4,\ \ \textit{for}\ \ q\in\Om,\ee
for a set $\Om\subs V_e$ of measure $\mu_e(\Om)\gs\eta^8$.
\\\\
Fix $q\in\Om$. As the functions $u_q^i$ $(i=1,2)$ are bounded, without loss of generality we may assume that they are indicator functions of sets $U_q^i\subs V_i$. Let $\B_{q,e}:=\B_q^1\vee\B_q^2$ denote the $\si$-algebra generated by the sets $\pi_i^{-1}(U_q^i)$ $(i=1,2)$ on $V_e$, and let $\EE_{\mu_{q,e}}(\1_{G_e}|\B_q)$ be the conditional expectation function of $\1_{G_e}$ with respect to this $\si$-algebra and the measure $\mu_{q,e}$. Recall that if $(V,\B,\mu)$ is a finite atomic measure space, and $B\in\B_q$ is an atom then $\EE_{\mu}(f|\B)(x)=\frac{1}{\mu(B)}\int_B f\,d\mu$ for all $x\in B$. It is easy to see that $\EE_{\mu}(f|\B)$ is the orthogonal projection of $f$ to the subspace of $\B$-measurable functions in $L^2(V,\B,\mu)$.
\\\\
As $u_q^1\times u_q^2$ is measurable with respect to $\B_q$, the inner product
\[\langle \1_{G_e}-\EE_{\mu_{q,e}}(\1_{G_e}|\B_{q,e}),\,u_q^1\times u_q^2\rangle_{\mu_{q,e}} = 0.\]
\\
This together with \eqref{1.6.3} and \eqref{1.6.4} implies
\[\langle\,\EE_{\mu_{q,e}}(\1_{G_e}|\B_{q,e})-\EE_{\mu_e}(\1_{G_e}|\B_0)\,,\,u_q^1\, u_q^2\,\rangle_{\mu_{q,e}} \gs \eta^4,\]
where $\B_0=\{V_e,\emptyset\}$ is the trivial $\si$-algebra, and hence $\ \EE_{\mu_e}(\1_{G_e}|\B_0)=\mu_e(G_e)\,\1_{V_e}+o_{N,W\to\infty}(1)$.\\
Then by the Cauchy-Schwartz inequality, we have
\eq \label{1.6.5}
\| \E_{\mu_{q,e}}(\1_{G_e}|\B_{q,e} )-\E_{\mu_{e}}(\1_{G_e}|\B_0) \|_{L^2(\mu_{q,e})}^2 \gtrsim \eta^8.
\ee\\
By the Pythagoras theorem, if the second term on the left side would be a conditional expectation with respect to the measure $\mu_{q,e}$ then one would obtain an ``energy increment"
\[
\| \E_{\mu_{q,e}}(\1_{G_e}|\B_{q,e} )-\E_{\mu_{q,e}}(\1_{G_e}|\B_0 )\|_{L^2(\mu_{q,e})}^2
=\|\E_{\mu_{q,e}}(\1_{G_e}|\B_{q,e} )\|^2_{L^2(\mu_{q,e})}-\|\E_{\mu_{q,e}}(\1_{G_e}|\B_0)\|^2
_{L^2(\mu_{q,e})} \gtrsim \eta^8.
\]
To overcome this ``discrepancy", a key observation of our approach is that the
measures $\mu_{q,e}$ are negligible perturbations of the measure $\mu_e$, for almost all values of the parameter $q$, which can be shown using the linear forms conditions. To be more precise,
for any given set $B\subs V_e$ one has for almost every $q \in V_e$
\eq\label{1.6.6}
\mu_{q,e}(B)-\mu_e(B)=o_{N,W \rightarrow \infty}(1).
\ee
The detailed proof of this, and other similar statements will be given in Section 2, we sketch the argument below. First, it is enough to show that
\[\La:=\E_{q\in V_e}\ |\mu_{q,e}(B)-\mu_e(B)|^2 \,\mu_e(q)=o_{N,W \rightarrow \infty}(1).\]
Since $\mu_{q,e}(x)=w(q,x)\mu_e(x)$, with $w(q,x)=\nu_e(x_1,q_2)\nu_e(q_1,x_2)$, one estimates
\begin{align*}
\La\ &=\ \E_{q\in V_e}\E_{x,y\in V_e}\1_B(x)\1_B(y)(w(q,x)-1)(w(q,y)-1)\mu_e(x)\mu_e(y)
\mu_e(q)\\
 &\leq\ \E_{x,y\in V_e}\mu_e(x)\mu_e(y)\,|\E_{q\in V_e} (w(q,x)-1)(w(q,y)-1)\mu_e(q)|.
\end{align*}
\noindent
By the linear forms condition $\E_{x\in V_e}\mu_e(x)=1+o_{N,W\to\infty}(1)\leq 2$ thus by the Cauchy-Schwarz inequality
\[\La^2\ls\ \E_{x,y\in V_e}\E_{q,p\in V_e} (w(q,x)-1)(w(q,y)-1)(w(p,x)-1)(w(p,y)-1)\times\]
\[\quad \mu_e(x)\mu_e(y)\mu_e(q)\mu_e(p).\]
\\
\noindent
The above expression is an alternating sum of 16-terms, each being $1+o_{N,W\to\infty}(1)$ by the linear forms condition. Indeed, the linear forms defining the measures $\mu_e(x),\,\mu_e(y),\,\mu_e(q)\,\mu_e(p)$ and the weight functions $w(q,x),\,w(q,y),\,w(p,x),\,w(p,y)$ are pairwise linearly independent in the $(x,y,q,p)$ variables, as either their support is different or by the fact that the measure $\mu_e(x)$ and the weight function $w(q,x)$ are defined via pairwise independent forms. Alternatively it follows from order-4 pseudo-randomness of the weight system $\{\nu_{12},\nu_1,\nu_2\}$ applied to the hypergraph bundle obtained by sampling the variables $x_i,y_i,q_i,p_i$ from the vertex sets $V_i$, for $i=1,2$. This in turn implies that

\eq\label{1.6.6.1}
\|\E_{\mu_{q,e}}(\1_{G_e}|\B_0)-\E_{\mu_e}(\1_{G_e}|\B_0) \|_{L^2(\mu_{q,e})}=o_{N,W \rightarrow \infty}(1)
\ee
and
\eq\label{1.6.6.2}
\|\E_{\mu_{e}}(\1_{G_e}|\B_0) \|_{L^2(\mu_{e})}=\|\E_{\mu_{q,e}}(\1_{G_e}|\B_0)\|_{L^2(\mu_{q,e})}+o_{N,W \rightarrow \infty}(1).
\ee
\\
\noindent
Then from \eqref{1.6.5}, assuming $N,W$ are sufficiently large w.r.t. $\eta$, we have that
\eq\label{1.6.7}
\| \E_{\mu_{q,e}}(\1_{G_e}|\B_{q,e}) \|_{L^2(\mu_{q,e})}^2 \geq \|\E_{\mu_{e}}(\1_{G_e}|\B_0) \|_{L^2(\mu_{e})}^2 +c\,\eta^8,
\ee
uniformly for $q\in\Om_1$, for a set $\Om_1\subs\Om$ of measure $\nu_e(\Om_1)\geq c\eta^8$.
\\\\
If $F:V \rightarrow \R$ is a function and $(V,\B,\mu)$ is a measure space, the quantity $\left\|\E_{\mu}(F|\B)\right\|_{L^2(\mu)}^2$ will be referred to as the ``energy" of the function $F$ with respect to the measure space $(V,\B,\mu)$, so \eqref{1.6.7} is telling that if $G_e$ is not $\eta$-uniform with respect to the initial measure spaces $(V_e,\B_0,\mu_e)$ then its energy increases by $c\eta^8$, with an absolute constant $c>0$, when passing to the measure spaces $(V_e,\B_{q,e},\mu_{q,e})$ for $q\in\Om_1$.
\\\\
We symmetrize the forms as $\mathcal{L}_{q,e}$ as follows. First pull back these forms to the subspace $M$ via the map $\pi_e:M\to V_e$, i.e. define
$\mathcal{L}_{q,M}:=\{L\circ\pi_e:\ L\in\mathcal{L}_{q,e}\}$, and for an edge $e'$ set $\mathcal{L}_{q,e'}:=\{L\circ\pi_{e'}^{-1}:\ L\in\mathcal{L}_{q,M}\}$. If $e'=J\backslash\{j'\}$ this amounts to the change of variables $x_{j'}:=-\sum_{i\neq j'}x_i$. The family of linear forms $\mathcal{L}_q:=\bigcup_{e\in\HH_2}\mathcal{L}_{q,e}$ is symmetric and well-defined in the $x$-variables, and also pairwise linearly independent in the $(q,x)$-variables. It is at this point where we have to construct weights using linear transformations and not just by extending our weights to hypergraph bundles.

\begin{ex} In case $\De=\{(0,0),(2,0),(1,1)\}$ the measure $\mu_{q,e}\ $ for $\ e=(1,2)\ $ is defined via the family of linear forms $\mathcal{L}_{q,e}=\{2x_1+x_2,2q_1+x_2,2x_1+q_2,x_2\}$, which are transferred to edges $e_1=(0,2)$ and $e_2=(0,1)$ to obtain $\ \mathcal{L}_{q,e_1}=\{-2x_0-x_2,2q_1-x_2,-2x_0-2x_2+q_2,x_2\}\ $ and \\ $\mathcal{L}_{q,e_2}=\{x_1-x_0,-x_0-x_1,2x_1+q_2,2q_1-x_0-x_1,\}$.
\end{ex}
\noindent
For given $q$ we define the corresponding weights and measures as in \eqref{1.4.2} and \eqref {1.4.3}, i.e. set
\[\nu_{q,e}(x):=\prod_{L\in\mathcal{L}_q,\, supp_x(L)=e} \nu (L(q,x)),\quad \mu_{q,e}(x):=\prod_{f\subs e} \nu_{q,f}(x),\]
where $supp_x (L)$ is the support of a linear form $L(q,x)$ in the $x$-variables. Finally set $\mathcal{B}_{q,0}:=\mathcal{B}_0=\{V_0,\emptyset\}$, and note that for each $q\in\Om_1$ we have a partition $\mathcal{B}_{q,i}$ (considered as a $\si$-algebra) on the vertex set $V_i$ into at most 2 sets. This gives rise to a partition $\mathcal{B}_{q,e}=\mathcal{B}_{q,i}\vee \mathcal{B}_{q,j}$ for an edge $e=(i,j)$. By the stability property \eqref{1.6.5} we have
\[\| \E_{\mu_{q,e}}(\1_{G_e}|\B_{q,e}) \|_{L^2(\mu_{q,e})}^2 \geq \|\E_{\mu_{q,e}}(\1_{G_e}|\B_0) \|_{L^2(\mu_{q,e})}^2 = \|\E_{\mu_{e}}(\1_{G_e}|\B_0) \|_{L^2(\mu_{e})}^2 + o_{N,W\to\infty}(1).
\]
Thus one obtains a symmetric version of the energy increment \eqref{1.6.7}, namely for $q\in\Om_1$
\eq\label{1.6.8}\sum_{e\in\HH_2}  \E_{\mu_{q,e}}(\1_{G_e}|\B_{q,e}) \|_{L^2(\mu_{q,e})}^2 \geq \sum_{e\in\HH_2} \|\E_{\mu_{e}}(\1_{G_e}|\B_0) \|_{L^2(\mu_{e})}^2 +c\,\eta^8.\ee
\\
\emph{2. Koopman-von Neumann decomposition.}
The above energy increment argument applies with slight changes to the parametric system of weights and measures $\nu_{q,e}$ and $\mu_{q,e}$. To be more precise assume that there is an edge, say $e=(1,2)$ and set $\Om_1'\subs\Om_1$ of size
$c_1:=\mu_e(\Om_1')\geq \mu_e(\Om_1)/2$, such that for all  $q\in\Om_1'$ one has
\[\|\1_{G_e}-\E_{\mu_{q,e}}(\1_{G_e}|\mathcal{B}_{q,e})\|_{\Box_{\nu_{q,e}}}\,\geq\,\eta.\]
Writing $F_q:=\1_{G_e}-\E_{\mu_{q,e}}(\1_{G_e})$, we have as in \eqref{1.6.2}
\eq \label{1.6.9}
\left\| F_q \right\|_{\Box_{\nu_{q,e}}}^4=\int_{V_e}\int_{V_e}\,  F_q(x)\,u_{q,p}^1(x_1) u_{q,p}^2(x_2)\nu_{q,e}(x_1,p_2)\nu_{q,e}(p_1,x_2)\,d\mu_{q,e}(x)\,d\mu_{q,e}(p) \geq \eta^4,
\ee
where $x=(x_1,x_2)$, $q=(q_1,q_2)$, $p=(p_1,p_2)$ and $u^1_{q,p},\ u^2_{q,p}$ are functions bounded by 1. Define the measure
\eq\label{1.6.9.5}
\mu_{q,p,e}(x):=\nu_{q,e}(x_1,p_2)\nu_{q,e}(p_1,x_2)\mu_{q,e}(x),
\ee
and note that it is defined by a parametric family of linear forms $\mathcal{L}_{q,p}$ which are pairwise independent in the $(q,p,x)$-variables. The space of parameters $Z:=V_e\times V_e$ is equipped with the measure $\psi(q,p)=\mu_{q,e}(p)\mu_e(q)$.
Then the inner expression in \eqref{1.6.9} may be viewed as the inner product
\eq\label{1.6.10}
\Gamma (q,p):= \left\langle F_q, u_{q,p}^1 \times u_{q,p}^2 \right\rangle_{\mu_{q,p,e}},
\ee
on the Hilbert space $L^2(V_e,\mu_{q,p,e})$. Thus \eqref{1.6.9} translates to $\ \E_{p\in V_e}\,\Gamma(q,p)\,\mu_{q,e}(p)\geq \eta^4\,$ for $q\in\Om_1'$. By \eqref{1.6.5} one has that $\mu_{q,e}(V_e)=1+o_{N,W\to\infty}(1)\leq 2$ for ``almost all" $q\in\Om_1'$, thus
\[\int_{V_e} \1_{\{\Gamma(q,p)\geq\eta^4/4\}}\,\Gamma(q,p)\,d\mu_{q,e}(p)\, \geq \frac{\eta^4}{2}.
\]
On the other hand using the linear forms condition, similarly as in \eqref{1.6.4}, one has
\[\int_{V_e}\int_{V_e}\,\Gamma(q,p)^2\,d\mu_{q,e}(p)\,d\mu_e(q) = 1+o_{N,W\to\infty}(1) \leq 2,\]
for $N,W$ sufficiently large. If we set $\Om_2':=\{(q,p)\in\Om_1'\times V_e:\ \Gamma(q,p)\geq\eta^4/2\}$, then by Cauchy-Schwarz
\eq\label{1.6.11}
\frac{c_1^2\eta^8}{4}\leq \left(\int_{\Om_2'} \Gamma (q,p)\,d\mu_{q,e}(p)\,d\mu_e(q)\right)^2\leq 2\,\psi(\Om_2').
\ee
This means that $\Gamma (q,p)\geq \eta^4/4$ for $(q,p)\in\Om_2'$,
for a set $\Om_2'\subs \Om_1'\times V_e$ of measure $\psi(\Om_2')\geq 2^{-4}\eta^8\mu_e(\Om_1)^2$.
\\\\
Since for $q'=(q,p)\in\Om_2'$ the functions $u_{q'}^i\ (i=1,2)$ are bounded  by 1, one may assume $u_{q'}^i=\1_{U_{q'}^i}$ for some sets $U_{q'}^i\subs V_i$, hence \eqref{1.6.10} becomes
\eq\label{1.6.12}
\langle \1_{G_e}-\E_{\mu_{q,e}}(\1_{G_e}|\mathcal{B}_{q,e}), \1_{U_{q'}^1}\times \1_{U_{q'}^2}\rangle_{\mu_{q',e}} \geq \frac{1}{4}\,\eta^4.\ee
Since $\1_{U_{q'}^1}\times \1_{U_{q'}^2}$ is constant on the atoms of $\mathcal{B}_{q',e}$, generated by $\mathcal{B}_{q,e}$ and the sets $\pi_i^{-1}\1_{U_{q'}^i}$ ($\pi_i:V_e\to V_i$ being the natural projection), one has
\[\langle \1_{G_e}-\E_{\mu_{q',e}}(\1_{G_e}|\mathcal{B}_{q',e}), \1_{U_{q'}^1}\times \1_{U_{q'}^2}\rangle_{\mu_{q',e}} = 0,\]
hence, uniformly for $q'=(q,p)\in\Om_2'$
\eq\label{1.6.13}
\langle \E_{\mu_{q',e}}(\1_{G_e}|\mathcal{B}_{q',e})-\E_{\mu_{q,e}}(\1_{G_e}|\mathcal{B}_{q,e}), \1_{U_{q'}^1}\times \1_{U_{q'}^2}\rangle_{\mu_{q',e}} \geq \frac{1}{4}\,\eta^4.\ee
Then by the Cauchy-Schwarz inequality
\eq\label{1.6.14}
\|\E_{\mu_{q',e}}(\1_{G_e}|\mathcal{B}_{q',e})-\E_{\mu_{q,e}}(\1_{G_e}|\mathcal{B}_{q,e})\|_{L^2_{\mu_{q',e}}} \geq \frac{1}{8}\,\eta^4.\ee
The measure $\mu_{q,p,e}(x)$ defined in \eqref{1.6.9.5} is of the form $\mu_{q,p,e}(x)=w(q,p,x)\mu_{q,e}(x)$ with a weight function $w(q,p,x)$ defined by pairwise linear independent forms depending on \emph{both} the $x$ and $p$ variables, hence the stability properties \eqref{1.6.6.1}-\eqref{1.6.6.2} remain valid. Thus for a set $\Om_2\subs\Om_2'$ such that $\psi(\Om_2'\backslash\Om_2)=o_{N,W\to\infty}(1)$, one has for all $q'=(q,p)\in\Om_2$,
\eq\label{1.6.15}
\|\E_{\mu_{q',e}}(\1_{G_e}|\mathcal{B}_{q',e})\|_{L^2_{\mu_{q',e}}}^2-\|\E_{\mu_{q,e}}(\1_{G_e}|\mathcal{B}_{q,e})\|_{L^2_{\mu_{q',e}}} \geq 2^{-7}\eta^8.\ee
Symmetrizing the forms $\mathcal{L}_{q',e}$ one obtains an energy increment of $2^{-8}\eta^8$ of the total energy:
\eq\label{1.6.16}
\sum_{e\in\HH_2}\|\E_{\mu_{q',e}}(\1_{G_e}|\mathcal{B}_{q',e})\|_{L^2_{\mu_{q',e}}}^2 \geq \sum_{e\in\HH_2}\|\E_{\mu_{q,e}}(\1_{G_e}|\mathcal{B}_{q,e})\|_{L^2_{\mu_{q',e}}} + 2^{-8}\eta^8,\ee
uniformly for $q'=(q,p)\in\Om_2$, for a set $\Om_2\subs Z$ of measure $\psi(\Om_2)\geq 2^{-5}\eta^8\mu_2(\Om_1)^2$.
\\\\
It is clear that the above procedure can be iterated leading to a family of measure spaces $(V_e,\mathcal{B}_{q_k,e},\mu_{q_k,e})_{e\in\HH_2}$ depending on a parameter $q_k\in\Om_k$ with $\Om_k\subs\Z_N^k$ of measure $\psi_k(\Om_k)\geq (\eta/2)^{c_k}$ (for example with $c_k=2^{k+2}$), at the $k$-th step. The procedure must stop at a step $k\leq 2^{10} \eta^{-8}$ as the total energy, given in \eqref{1.6.16}, is increasing by at least $2^{-8}\eta^8$ in each step however is bounded by $3(1+o_{N,W\to\infty}(1))\leq 4$, for sufficiently large $W$ and $N$. At this step there is a set $\Om_k'\subs\Om_k$ such that for all $q_k\in\Om_k$ and $e\in\HH_2$, we have
\[\|\1_{G_e}-\E_{\mu_{q_k,e}}(\1_{G_e}|\mathcal{B}_{q_k,e})\|_{\Box_{\nu_{q_k,e}}}\,\leq\,\eta.\]
Given an edge $e=(i,j)$ this gives a \emph{Koopman-von Neumann} type decomposition $\1_{G_e}=f_{q_k,e}+g_{q_k,e}$ with $f_{q_k,e}$ being constant on the atoms
of $\mathcal{B}_{q_k,e}=\mathcal{B}_{q_k}^i\vee\mathcal{B}_{q_k}^j$, while $g_{q_k,e}$ has small $\Box_{\nu_{q_k,e}}$-norm. This is not quite enough to prove a removal lemma as the size of the atoms are too small with respect to the regularity parameter $\eta$.
\\\\
\emph{3. Regularity Lemma.} We will iterate the above decomposition to arrive to a partition of the vertex spaces so that on most atoms the graphs $G_e$ become sufficiently uniform. Originally such partitions of graphs were obtained by Szemer\'{e}di \cite{SZ}, we will follow the approach of Tao \cite{Treg} as it is more amenable to extensions to hypergraphs.
\\\\
Let $(V_e,\mathcal{B}_{q,e},\mu_{q,e})_{e\in\HH_2}$ be well-defined, symmetric system of measure spaces, depending on a parameter $q\in\Z^K$ and let $\Om\subs\Z_N^K$ be a set of measure $\psi(\Om):=c_0>0$. Assume that the $\si$-algebras $\mathcal{B}_{q,e}=\mathcal{B}_{q,i}\bigvee\mathcal{B}_{q,j}$ ($e=(i,j)$) have complexity (i.e. the least number of sets needed to generate them) bounded by $M_0$ uniformly in $q$.
\\\\
Given $\eps>0$ and a function $F(M,\eps)>0$, what we would like to achieve is a symmetric, well-defined extension $(V_e,\mu_{q',e})_{e\in\HH_2}$ and families of $\sigma$-algebras $\mathcal{B}_{q',i}\subs \mathcal{B}'_{q',i}\ (i=1,2,3)$  for which the following holds. There is a set $\Om'\subs\Om\times \Z_N^l$ of measure $\psi'(\Om')\geq c(F,M_0,c_0,\eps)>0\,$ such that uniformly for $q'=(q,p)\in\Om'$ one has
\\\\
$(1)$ The complexity of the $\si$-algebras $\mathcal{B}_{q',i}$ is at most $M=O(M_0,F,\eps)$.
\\\\
$(2)$\ For all $e\in\HH_2$,
\[\|\1_{G_e}-\E_{\mu_{q',e}}(\1_{G_e}|\mathcal{B}'_{q',e})\|_{\Box_{\nu_{q',e}}}\,\leq\,\frac{1}{F(M,\eps)}.
\]
$(3)$ For all $e\in\HH_2$,
\[\|\E_{\mu_{q',e}}(\1_{G_e}|\mathcal{B}_{q',e}) - \E_{\mu_{q',e}}(\1_{G_e}|\mathcal{B}'_{q',e})\|_{L^2_{\mu_{q',e}}}\,\leq\,\eps.\]
\noindent
The second property means that the uniformity parameter can be taken arbitrary small with respect to the complexity $M$ of the ``coarse" $\si$-algebras $\mathcal{B}_{q',i}$. The third states that the approximations on the coarse and fine scales to the graphs $G_e$ are sufficiently close in the $L^2$ sense, see \cite{Treg}.
\\\\
Apply the Koopman-von Neumann decomposition with uniformity parameter $\eta:=F(M_0,\eps)^{-1}$ to the original system $(V_e,\mathcal{B}_{q,e},\mu_{q,e})_{e\in\HH_2}$. This generates a well-defined, symmetric extension $(V_e,\mathcal{B}'_{q',e},\mu_{q',e})_{e\in\HH_2}$. Set $\mathcal{B}_{q',e}:=\mathcal{B}_{q,e}$ for $q'=(q,p)$. The new system satisfies $(1)$ and $(2)$ uniformly for $q'=(q,p)\in\Om'$ for a set $\Om'\subs \Om\times \Z_N^l$ of measure $\psi'(\Om_1)\geq \psi'(\Om')/2\geq c(F,\eps,M_0,c_0)>0$. There are two possibilities.
\\\\
There exists a set $\Om_1'\subs\Om'$ of measure $\psi'(\Om_1')\geq \psi'(\Om')/2$ such that $(3)$ also holds uniformly for $q'\in\Om_1'$. In this case all three properties hold for this extension and the $\sigma$-algebras $\mathcal{B}_{q',e}\subs\mathcal{B}'_{q',e}$.
\\\\
There exists a set $\Om_2'\subs\Om'$ of measure $\psi'(\Om_1')\geq \psi'(\Om')/2$ such that $(3)$ fails for all $q'\in\Om'_2$. Then, thanks to the stability condition \eqref{1.6.8} and the fact that $\mathcal{B}_{q',e}=\mathcal{B}_{q,e}$, we have for ``almost every" $q'=(q,p)\in\Om_2'$

\[\sum_{e\in\HH_2} \|\E_{\mu_{q',e}}(\1_{G_e}|\mathcal{B}'_{q',e})\|_{L^2_{\mu_{q',e}}}^2 - \sum_{e\in\HH_2}\|\E_{\mu_{q,e}}(\1_{G_e}|\mathcal{B}_{q,e})\|_{L^2_{\mu_{q,e}}}^2\]
\[= \sum_{e\in\HH_2} (\,\|\E_{\mu_{q',e}}(\1_{G_e}|\mathcal{B}'_{q',e})\|_{L^2_{\mu_{q',e}}}^2 - \|\E_{\mu_{q',e}}(\1_{G_e}|\mathcal{B}_{q,e})\|_{L^2_{\mu_{q',e}}}^2\,)-o_{N,W\to\infty}(1)\]
\[= \sum_{e\in\HH_2} \|\,\E(\1_{G_e}|\mathcal{B}'_{q',e}) - \E_{\mu_{q',e}}(\1_{G_e}|\mathcal{B}_{q',e})\|_{L^2_{\mu_{q',e}}}^2-o_{N,W\to\infty}(1)\,\geq\,\frac{\eps^2}{2}.\]

This means that the total energy of the extended system $(V_e,\mu_{q',e},\mathcal{B}'_{q',e})$ is at least $\frac{\eps^2}{2}$-larger than that of the system $(V_e,\mu_{q,e},\mathcal{B}_{q,e})$ for all $q'=(q,p)\in \Om_2'$. Repeating this step one generates systems $(V_e,\mu_{q_i,e})$ and $\sigma$-algebras
$ \mathcal{B}_{q_i,e}\subs\mathcal{B}'_{q_i,e}$ such that either (1)-(3) hold or the total energy of the system is increased by at least $\frac{\eps^2}{2}$ at the $i$-th step. Since the total energy is $O(1)$ the process must stop at a step $i=O(\eps^{-2})$.
\\\\
\emph{4. Removal lemma.} We will sketch below how to use the decomposition obtained by the regularity lemma to prove Theorem 1.4. Recall that given $\eps>0$, it is enough to show that there is a $\de=\de(\eps)>0$ and a $q\in\Om$ such that if $\mu(E_{12}\cap E_{13}\cap E_{23})\leq\de$ then \eqref{1.4.11}-\eqref{1.4.14} hold for some sets $E'_{q,12},E'_{q,13},E'_{q,23}$ for the measures $\tilde{\mu}_e=\mu_{q,e}$, that is
\eq\label{1.6.21}
\bigcap_{e\in\mathcal{H}_2} (E_e\cap E'_{q,e})=\emptyset,\quad \sum_{e\in\mathcal{H}_2} \mu_{q,e}(E_e\backslash E'_{q,e})\ls\eps.
\ee
Note that \eqref{1.4.11}-\eqref{1.4.12} hold for all $q\in\Om$ outside an exceptional set $\mathcal{E}$ of measure $\psi(\mathcal{E})=o_{N,W\to\infty}(1)$.
\\\\
So let us assume that we have $\si$-algebras $\mathcal{B}_{q,i}\subs\mathcal{B'}_{q,i}$ on the vertex sets $V_i$ obtained by the regularity lemma with respect to  $\eps_1:=\eps^4>0$ and a function $F(M,\eps)$; we choose $F(M,\eps):=\eps^{-8}\,2^{6M+4}$. Note that the complexity of the $\si$-algebras $\mathcal{B}_{q,i}\leq M$ uniformly for $q$ where $M=M(F,\eps)>0$ is a constant depending only on $F$ and $\eps$.
\\\\
We will work with the $\si$-algebras $\aB_{q,i}\subs\aB'_{q,i}$, $\aB_{q,i}:=\pi_i^{-1}(\mathcal{B}_{q,i})$ together with the $\si$-algebras $\aB_{q,e}:=\B_e$ generated by the single set $E_e$ on $V=V_1\times V_2\times V_3$. The atoms of the $\si$-algebra $\aB_q=\bigvee_{e\in\bar{\HH}} \aB_{q,e}$ are of the form $A_q=A_{12}\cap A_{13}\cap A_{23}\cap A_{q,1}\cap A_{q,2}\cap A_{q,3}$. The set $E_{12}\cap E_{13}\cap E_{23}$ is the disjoint union of all such atoms where $A_{ij}=E_{ij}$.
\\\\
It will be convenient to work with the single measure $\mu_q:=\prod_{f\in\bar{\HH}}\nu_{q,f}$ instead of the ensemble of measures $\mu_{q,e}=\prod_{f\subs e}\nu_{q,f}$, fortunately this can be done at essentially no cost as $\mu_q(A_{q,e})=\mu_q(\pi_e^{-1}(A_{q,e}))+o_{N,W\to\infty}(1)$ for all $e\in\bar{\HH}$, uniformly for $q$ outside a set $\mathcal{E}\subs\Om$ of measure $\psi(\Om)=o_{N,W\to\infty}(1)$. This is given in Lemma 2.2 in Section 2, the argument being similar to the one given after \eqref{1.6.6}.
\\\\
The proof of Theorem 1.4 is roughly as follows. First one shows that most atoms $A_q$ are ``large", in fact
\eq\label{1.6.22} \mu_q(A_q) \geq \frac{1}{F(M,\eps)},\ee
in the sense that all ``small" atoms $A_q$ are contained in sets of ``edges" $B_{q,e}\in\mathcal{A}_e$ of measure
\eq\label{1.6.23} \mu_q(B_{q,e}) =O(\eps).\ee\\
Let $\de:=\frac{1}{2}F(M,\eps)^{-1}$. If $\mu_q(E_{12}\cap E_{13}\cap E_{23})\leq\de$ then it cannot contain any large atoms. If one defines $E'_{q,e}:=V\backslash B_{q,e}$ then $\bigcap_{e\in\HH_2} E'_{q,e}$ does not contain any small atoms, hence $\bigcap_{e\in\HH_2} (E_e\cap E'_{q,e})$ cannot contain any atoms at all and must be empty and \eqref{1.6.21} follows \eqref{1.6.23}.
\\\\
Let $A_q=\cap_{e\in\HH_2} A_e\cap_i A_{q,i}$ be an atom of $\aB_q$. We say that $A_q$ is $(\eps,M)$-regular if the following holds for all $i=1,2,3$ and $e=(i,j)\in\HH_2$.

\eq\label{RI} \mu_q(A_{q,i})\geq\eps 2^{-M}.\ee

\eq\label{RII} \mu_q (A_e\cap A_{q,i}\cap A_{q,j}) \geq \eps\,\mu_q( A_{q,i}\cap A_{q,j}).\ee

\eq\label{RIII} \int_V |\EE_{\mu_q} (\1_{A_e}\mid\, \aB_{q,i}\bigvee \aB_{q,j}) - \EE_{\mu_q} (\1_{A_e}\mid\, \aB_{q,i}\bigvee \aB_{q,j})|^2
\1_{A_{q,i}} \1_{A_{q,j}}\,d\mu_q \,\leq\, \eps^4\, \int_V \1_{A_{q,i}} \1_{A_{q,j}}\,d\mu_q.\ee

Let $B_{q,e}$ be the union of all atoms $A_e\cap A_{q,i}\cap A_{q,j}$ for which at least one of \eqref{RI}-\eqref{RIII} fails. We show that $\mu_q(B_{q,e})=O(\eps)$. First note that $\aB_{q,i}$ has at most $2^M$ atoms as its complexity is bounded by $M$ thus the total measure of atoms $A_q$ when \eqref{RI} fails is $O(\eps)$. If \eqref{RII} fails then $\mu_q(A_q)\leq\eps \mu_q(A_{q,i}\cap A_{q,j})$ for some $e=(i,j)\in\HH_2$. Since the sets $A_{q,i}\cap A_{q,j}$ are disjoint, the total measure of such atoms is also $O(\eps)$. Finally, summing over all pairs $(i,j)$ when \eqref{RIII} fails, one estimates using property (3) of the regularity lemma

\[
\begin{split}
\sum_{i,j} \mu_q(A_{q,i}\cap A_{q,j})& \leq\eps^{-1}\, \|\EE_{\mu_q} (\1_{A_e}\mid\, \aB_{q,i}\bigvee \aB_{q,j}) - \EE_{\mu_q}(\1_{A_e}\mid\, \aB'_{q,i}\bigvee \aB'_{q,j})\|^2_{L^2_{\mu_q}}\\
& = \eps^{-1}\,\|\EE_{\mu_{q,e}} (\1_{A_e}\mid\, \mathcal{B}_{q,i}\bigvee \mathcal{B}_{q,j}) - \EE_{\mu_q}(\1_{A_e}\mid\, \mathcal{B}'_{q,i}\bigvee \mathcal{B}'_{q,j})\|^2_{L^2_{\mu_{q,e}}}\,=\,O(\eps^3).
\end{split}
\]
\\
A crucial fact is that the measure of the regular atoms can be approximately determined which is the content of the so-called \emph{Counting Lemma} \cite{Go1,T1}. To state it let us introduce the relative densities $\de_{q,e}:= \mu_q(A_e\cap A_{q,i}\cap A_{q,j})/\mu_q( A_{q,i}\cap A_{q,j})$ and set $\de_{q,i}:=\mu_q(A_{q,i})$. Then for an $(\eps,M)$-regular atom\\ $A_q=\cap_{e\in\HH_2} A_e\cap_i A_{q,i}$ we have
\eq\label{counting}
\mu_q(A_q)=(1+O(\eps))\,\de_{q,12}\de_{q,13}\de_{q,23}\de_{q,1}\de_{q,2}\de_{q,3} + O(\frac{1}{F(M,\eps)})+o_{N,W\to\infty}(1).
\ee
\\
The above also holds for regular atoms of any subgraph $\HH'\subs\HH_2$. In fact the proof proceeds by induction on the number of edges of $\HH'$. If $\HH'=\{1,2,3\}$ i.e., it has no edges, then one has the stronger approximation
\eq\label{1.6.24}
\mu_q(A_{q,1}\cap A_{q,2}\cap A_{q,3}) = \mu_q(A_{q,1})\mu_q(A_{q,2})\mu_q(A_{q,3}) + o_{N,W\to\infty}(1).\ee
This again follows, for $q$ outside a negligible set, from the linear forms condition by an argument similar to the one given after \eqref{1.6.6}, see Section 2. Fix an edge, say $e=(1,2)$ and assume that \eqref{counting} holds for regular atoms of the graph $\HH'=\HH_2\backslash \{e\}$. Write
\[\1_{A_e}=\de_{q,e}+b_{q,e}+c_{q,e},\quad\textit{where}\]
\[b_{q} = \EE_{\mu_q} (\1_{A_e}\mid\, \aB'_{q,1}\bigvee \aB'_{q,2}) - \EE_{\mu_q}(\1_{A_e}\mid\, \aB_{q,1}\bigvee \aB_{q,2}),\]
\[c_{q} = \1_{A_{q,e}} - \EE_{\mu_q} (\1_{A_e}\mid\, \aB'_{q,1}\bigvee \aB'_{q,2}).\]
Note that $\de_{q,e} = \EE_{\mu_q}(\1_{A_e}\mid\, \aB_{q,1}\bigvee \aB_{q,2})$ on the set $A_{q,1}\cap A_{q,2}$. Accordingly, integrating term by term

\[
\begin{split}
\int_V \prod_{f\in\HH} \1_{A_{q,e}}\,d\mu_q\, &= \,\int_V (\de_{q,e}+b_{q,e}+c_{q,e})\ \prod_{f\neq e} \1_{A_{q,e}}\,d\mu_q\\
                                              &= \, M_q + E_q^1 + E_q^2.
\end{split}
\]
For the main term, by induction for $\HH'=\HH\backslash\{e\}$, we have
\[M_q =\de_{q,e} (1+O(\eps))\,\prod_{f\neq e}\de_{q,f} + O(\frac{1}{F(M,\eps)}) +o_{N,W\to\infty}(1).\]
For the first error term, taking absolute values and discarding the factors $\1_{A_{13}}$, $\1_{A_{23}}$, we have
\[
\begin{split}
|E_q^1| &\leq \int_{x_1,x_2}|b_q(x_1,x_2)|\1_{A_{q,1}}(x_1) \1_{A_{q,2}}(x_2)\,\left(\int_{x_3} \1_{A_{q,3}}(x_3) w_q(x_1,x_2,x_3)\,d\mu_{q,3}\right)\,d\mu_{q,12}\\
        &=    \int_{x_1,x_2} I_q(x_1,x_2)\,J_q(x_1,x_2)\,d\nu_{q,12},
\end{split}
\]
where $w_q(x_1,x_2,x_3)=\nu_{q,13}(x_1,x_3)\nu_{q,23}(x_2,x_3)$. By \eqref{RIII}
\[\int I_q^2\, d\mu_{q,12}\, \leq \,\eps\,\int \1_{A_{q,1}} \1_{A_{q,2}}\, d\mu_{q,12}\,\ls\,\eps^4\de_{q,1}\de_{q,2}.\]
Expanding the square of the integral in $x_3$,
\[\begin{split}
\int J_q^2\, d\mu_{q,12}\,&=\,\int_{x_1,x_2,x_3,x_{3'}} \1_{A_{q,1}}(x_1) \1_{A_{q,2}}(x_2)\1_{A_{q,3}}(x_3) \1_{A_{q,3}}(x_{3'})\,d\mu'_q\\
                          &=\, (1+o_{N,W\to\infty}(1))\,\de_{q,1}\de_{q,2}\de_{q,3}^2,
\end{split}
\]
\\
where the second equality again holds outside an exceptional set $\mathcal{E}\subs\Om$. Thus by the Cauchy-Schwarz inequality, property \eqref{RII} of regular atoms and the fact that
$\mu_{q,12}(V_1\times V_2)=1+o_{N,W\to\infty}(1)$, we have
\[|E_q^1| \ls \eps^4\,\de_{q,1}\de_{q,2}\de_{q,3} \ls \eps\,\de_{q,12}\de_{q,12}\de_{q,23}\de_{q,1}\de_{q,2}\de_{q,3}.\]
The error term $E_q^2$ can be estimated by Proposition 1.1, the weighted von-Neumann inequality. Indeed, set $F_{q,12}:=c_q$, $F_{q,13}:=\1_{A_{13}}\1_{A_{q,1}}\1_{A_{q,3}}$ and $F_{q,23}:=\1_{A_{23}}\1_{A_{q,2}}$. Then by Proposition 1.1
\[
|E_q^2|\,\ls\,\|c_q\|_{\Box_{\nu_{q,12}}} + o_{N,W\to\infty}(1)\,=\,O\left(\frac{1}{F(M,\eps)}\right),
\]
by the regularity condition (3) uniformly, for $q\in \Om\backslash \mathcal{E}$ for set of measure $\psi(\mathcal{E})=o_{N,W\to\infty}(1)$. A detailed proof of the general case will be given in section 4.2. Various versions of weighted von-Neumann inequalities are obtained in \cite{CFZ,GT2}, for the sake of completeness we include a proof of Proposition 1.1 in an appendix.
This concludes the proof of Theorem 1.4 for graphs, i.e. when $d=2$.
\\


\subsection{Outline of the paper} In Section 2 we describe the type of parametric weight systems $\{\nu_{q,f} \}_{f \in \HH,\,q\in Z}$ we encounter later on.  Here we also discuss their basic properties such as stability and symmetry. Though the notations are more complex most argument are variants of the one described in section 1.6, see \eqref{1.6.5}, based only on the Cauchy-Schwarz inequality and the linear forms condition.
\\\
The main argument starts in Section 3 where we introduce the energy increment argument for parametric weight systems on hypergraphs, and prove a hypergraph regularity lemma. Section 4 is devoted to proving the counting and removal lemmas. Many of our arguments in Section 3 and Section 4 may be viewed as an extension of those in \cite{T1}. In the  last section we obtain our main results stated in the introduction. The basic properties of weighted box norms are discussed in an appendix.
\\\\
As for our notations, most of our variables are vector type, although we do not emphasize this. We think of the initial data $\De=\{v_0,\ldots,v_d\}$ being fixed throughout, and do not denote the dependence of various quantities on them. For example we write $Y=O(X)$ or $Y\ls X$ if $Y\leq C\,X$ for some constant $C>0$ depending only on the vectors $v_i$ or the dimension $d$. If $y_1,\ldots,y_s$ and $X$ are additional parameters we write $O_{y_1,\ldots,y_s}(X)$ for a quantity $Y$ bounded by $C(y_1,\ldots,y_s)X$ or equivalently $Y\ls_{y_1,\ldots,y_s}\,X$.
\\\\
Though most of our constructions depend on $N$, we will not indicate that to emphasize that all other terms in our estimates are independent of $N$ as well as the parameters appearing in them. We'll utilize the linear forms condition throughout the paper, giving rise to error terms which tend to 0 as both $N\to\infty$ and $W\to \infty$ for any fixed choice of the parameters $y_1,\ldots,y_s$ on which they may depend. The standard notation for such terms would be $o_{N,W\to\infty;y_1,\ldots,y_s}(1)$, for simplicity we will write $o_{y_1,\ldots,y_s}(1)$. Finally as all estimates in the linear forms condition involving the weights $\nu_b$ are independent of the choice of $b$ we write in certain places $\nu=\nu_b$ for the purpose of simplifying the notation.

\section{Basic properties of parametric weight systems and their extensions} In this section we define the type of parametric systems and associated families of measures we encounter later and  discuss their basic properties such as stability and symmetry. We also discuss the type of extensions of such systems which  arise in our induction process. Although our arguments are increasing in complexity they are in essence just variations of the argument described after \eqref{1.6.6} in the introduction, based on the Cauchy-Schwarz inequality and the linear forms condition.

\subsection{Parametric weight systems and stability properties}
Recall the family of measures $\{\mu_e \}_{e \in \HH}$ constructed in \eqref{1.4.1}
\[
\mu_e(x)=\prod_{L\in\LL,\,supp(L)\subs e}\nu(L(x)),
\]
where the family $\LL$ defined in \eqref{1.4.1} consists of pairwise linearly independent forms. The following statement is based on the linear forms condition and is a prototype of many of the arguments in this section.

\begin{lemma}\label{lemma2.1} For all $e \in \HH$ we have that

\eq \label{2.1.1}
\mu_e(V_e)=1+o(1),
\ee
moreover if $g:V_e \rightarrow [-1,1],$
\[
\EE_{x_e\in V_e}\,g(x_e)\,\mu_e(x_e)=\EE_{x\in V_J}\,g(\pi_e(x))\,\mu_J(x)+o(1),
\]
or equivalently
\eq \label{2.1.2}
\int_{V_e}g\, d\mu_e= \int_{V_J}(g \circ \pi_e)\,d\mu_J + o(1).
\ee
\end{lemma}

\begin{proof}
Note that the linear forms appearing on the right side of
\[
\mu_e(V_e)=\E_{x \in V_e} \prod_{supp(L)\subs e}\nu(L(x))
\]
\\
are pairwise linearly independent, and as they are supported on $e$ they remain pairwise independent when restricted to $V_e$. Thus
\eqref{2.1.1} follows from the linear forms condition.\\
To show $\eqref{2.1.2}$, let $e'=J\backslash e$ and write $x=(x_e,x_{e'})$ with $x_e=\pi_e(x),\ x_{e'}=\pi_{e'}(x)$. Then
\[
E:=\EE_{x\in V_J}(g \circ \pi_e)(x)\mu_J(x)-\EE_{x_e\in V_e}\,g(x_e)\,\mu_e(x_e) =\EE_{x_e\in V_e}\  g(x_e)\,\mu_e(x_e)\,\E_{x_{e'} \in V_{e'}}(\mathrm{w}(x_e,x_{e'})-1),
\]
where $\mathrm{w}(x_e,x_{e'})=\prod_{f \nsubseteq e}\nu_f(x_{e \cap f}, x_{e' \cap f})$.\\\\
By \eqref{2.1.1} we have that $\mu_e(V_e)\ls 1 $, and then by the Cauchy-Schwartz inequality
\[
|E|^2 \ls \E_{x_e \in V_e}\E_{x_{e'},y_{e'} \in V_{e'}}\ (\mathrm{w}(x_e,x_{e'})-1)(\mathrm{w}(x_e,y_{e'})-1)\,\mu_e(x_e).
\]
The right hand side of this expression is a combination of four terms and \eqref{2.1.2} follows from the fact that each term is $1+o(1).$ Indeed the linear forms appearing in the definition of the function $\mu_e(x_e)$ depend only on the variables $x_j$ for $j\in e$ and are pairwise linearly independent. All linear forms involved in $\mathrm{w}(x_e,x_{e'})$ depend also on some of the variables in $x_j$, $j\in e'$, while the ones in $\mathrm{w}(x_e,y_{e'})$ depend on the variables in $y_j$, $j\in e'$, hence these forms depend on different sets of variables. Thus the forms appearing in the expression $\mu_e(x_e)\mathrm{w}(x_e,x_{e'})\mathrm{w}(x_e,y_{e'})$ are pairwise linearly independent and \eqref{2.1.2} follows from the linear forms condition.
Note that the estimate is independent on the function $g$.
\end{proof}

\noindent This will allow us to consider sets $G_e \subseteq V_e$ as sets $\overline{G}_e =\pi_e^{-1}(G_e) \subseteq V_J$, changing their measure only by  a negligible amount
\eq\label{2.1.3}
\mu_J(\overline{G}_e)=\mu_e(G_e)+o(1).
\ee\\
Next we define weight systems and associated families of measures depending on parameters. Let
\[
\mathcal{L}_q:= (L^1(q,x),...,L^s(q,x))
\]
be a family of linear forms with integer coefficients depending on the parameters $q\in \Q^R$ and the variables $x \in \Q^D$. We call the family \emph{pairwise linearly independent} if no two forms in the family are rational multiples of each other. If $N$ is a sufficiently large prime with respect to the coefficients of the linear forms $L^i(q,x)$, then the forms remain pairwise linearly independent when considered as forms over $Z\times V$, $Z=\Z_N^R$, $V=\Z_N^D$. We refer to the set $Z=\Z_N^R$ as the \emph{parameter space} of the family $\mathcal{L}_q$. As our arguments will involve averaging over the parameter space $Z$, we call the family $\LL_q$ \emph{well-defined} if there is measure on $Z$ given by
\eq\label{2.1.4}
\int_Z g(q)\,d\psi(q)=\EE_{q\in Z}\ g(q)\,\psi(q),\ \ \ \ \psi(q)=\prod_{i=1}^t \nu(Y_i(q)),
\ee
for a family of pairwise linearly independent linear forms $Y_i$ defined over $Z$, and if all forms $L^i(q,x)$ depend on some of the $x$-variables.
\\
If $V=V_J$ then we define an associated system of weights $\{\nu_{q,e}\}_{q\in Z, e\in\HH}$ and measures $\{\mu_{q,e}\}_{q\in Z, e\in\HH}$ as follows. For a form $\ L^k(q,x)=\sum_i b_iq_i+\sum_j a_j x_j\ $ define its $x$-support as $\ supp_x (L)=\{j\in J:\ a_j\neq 0\}$. For $e\subs J$ and $q\in Z$, let

\eq\label{2.1.5}
\nu_{q,e}(x) := \prod_{\substack{L\in \LL_q\\ supp_x(L)=e}} \nu(L(q,x)),\ \ \ \ \ \ \mu_{q,e}(x) := \prod_{\substack{L\in \LL_q\\ supp_x(L)\subs e}} \nu(L(q,x)).
\ee
We use the convention that $\nu_{q,e}\equiv 1$ if there is no form $L\subs \LL_q$ such that $supp_x (L)=e$. Note that the $x$-support partitions the family of forms $\LL_q$ independent of the parameters $q$, thus for given $e\in\HH$
\[ \mu_{q,e}(x)=\prod_{f\subs e} \nu_{q,e}(x),\ \ \ \ \textit{for all}\ \ q\in Z.\]
\\
A crucial observation is that many of the properties of the measure system $\{ \mu_e\}$ still hold for well-defined measure systems $\{\mu_{q,f}\}$ for almost every value of the parameter $q \in Z$. In order to formulate such statements we say that the family $\LL$ has \emph{complexity} at most $K$ if the dimension of the space $Z$, the number of linear forms $L^j(q,x),\,Y_l(q)$, and the magnitude of their coefficients are all bounded by $K$. This quantity will control the dependence of the error terms in applications of the linear forms condition. We have the analogue of Lemma \ref{lemma2.1}.

\begin{lemma}\label{lemma2.2}
Let $\{\mu_{q,e}\}_{e\in\HH,q\in Z}$ be a well-defined parametric measure system of complexity at most $K$.\\
For every $e \in \HH$ there is a set $\e_e \subseteq Z$ such that $\psi(\e_e)=o_K(1)$, and for every $q \notin \e_e$

\eq \label{2.1.6}
\mu_{q,e}(V_e)=1+o_K(1).
\ee
\\
Moreover for every $e \in \HH$ there is a set $\e_{e} \subseteq Z$ of measure $\psi(\e_{e})=o(1)$, such that the following holds.
For any function $g:Z\times V_e \rightarrow [-1,1]$  and for every $q\notin \e_e$ one has the estimate

\eq \label{2.1.7}
\int_{V_e}\,g(q,x_e)\ d\mu_{q,e}(x_e)= \int_{V_J} g (q,\pi_e(x))\ d\mu_{q,J}(x) +o_K(1).
\ee
\end{lemma}

\medskip

\begin{proof}
To prove \eqref{2.1.6} consider the quantity
\begin{align*}
\La_e &:= \int_Z |\mu_{q,e}(V_e)-1|^2\ d\psi(q) \\
&=\int_Z  \E_{x_e,y_e}(\prod_{supp_x(L)\subs e}\nu(L(q,x_e))-1)(\prod_{supp_x(L) \subs e}\nu(L(q,y_e))-1)\ d\psi(q).
\end{align*}
The above expression is a combination of four terms and note that the family of linear forms \[\{Y_k(q),L^i(q,x_e),L^j(q,y_e) \}\] is pairwise linearly independent in the $(q,x_e,y_e)$ variables by our assumptions. Applying the linear forms condition gives that each term is $1+o_K(1)$ and so $\La_e=o_K(1)$ and \eqref{2.1.6} follows.\\\\
Now let $e'=J\backslash e$, write $x=(x_e,x_{e'})$ and arguing as in Lemma \ref{lemma2.1} we have
\begin{align*}
\La(q,e,g)&:=|\ \EE_{x\in V_J}\ g(q,\pi_e(x))\,\mu_{q,J}(x)-\EE_{x_e\in V_e}\ g(q,x_e)\,\mu_{q,e}(x_e)|\\
&=\ |\EE_{x_e\in V_e}\ g(q,x_e)\,\mu_{q,e}(x_e)\,\E_{x_{e'} \in V_{e'}}\ (\mathrm{w}_q(x_e,x_{e'})-1)|\\
&\leq\ \EE_{x_e\in V_e}\ \mu_{q,e}(x_e)\,|\E_{x_{e'} \in V_{e'}}\, (\mathrm{w}_q(x_e,x_{e'})-1)|,
\end{align*}
where $\mathrm{w}_q(x_e,x_{e'})=\prod_{f \nsubseteq e}\nu_{q,f}(x_{e \cap f}, x_{e' \cap f})$.\\\\
Notice that the right hand side of the above inequality is independent of the function $g$; if we denote it by $\La(q,e)$ then \eqref{2.1.7} would follow from the estimate $\EE_{q\in Z}\ \La(q,e)\,d\psi(q)=o_K(1)$. By the linear forms condition $\EE_{q,x_e}\,d\psi(q)\,d\mu_{q,e}(x_e)=1+o_K(1)\leq 2$, for $N$ sufficiently large with respect to $K$. Then by the Cauchy-Schwartz inequality one has\\
\[
\left( \EE_{q\in Z}\ \La(q,e)\,d\psi(q)\right)^2 \ls \E_{q\in Z,\,x_e \in V_e}\ \E_{x_{e'},y_{e'} \in V_e}\ (\mathrm{w}_q(x_e,x_{e'})-1)(\mathrm{w}_q(x_e,y_{e'})-1)\,d\mu_{q,e}(x_e)\,d\psi(q).
\]\\
This is a combination of four terms, however each term again is $1+o_{K}(1)$  as the linear forms defining $\psi$ depend on the variables $q$ while the ones defining $\mu_{q,e}$ depend also on the $x_e$ variables. On the other hand all linear forms appearing in the weight functions $\mathrm{w}_q(x_e,x_{e'})$ (resp. $\mathrm{w}_q(x_e,y_{e'})$) depend on the $x_{e'}$ (resp. $y_{e'}$) variables as well. Thus the family of all linear forms in the above expressions is pairwise linearly independent in the $(q,x_e,x_{e'},y_{e'})$ variables.
\end{proof}

\subsection{Extension of parametric systems}
During our iteration process we will encounter extensions of parametric families of forms depending on more and more parameters. Roughly speaking one extends a family by adding new parameters together with new forms depending also on the new parameters. More precisely  let $\mathcal{L}_{q_1}^1=\{L_1^1(q_1,x),...,L_1^{s_1}(q_1,x)\}$ and $\mathcal{L}_{q_2}^2=\{L_2^1(q_2,x),...,L_2^{s_2}(q_2,x)\}$ be two pairwise linearly indpendent families of linear forms defined on the parameter spaces $Z_1=\Z_N^{k_1}$ and respectively $Z_2=\Z_N^{k_2}$. Let $\psi_1$ and $\psi_2$ be measures on $Z_1$ and $Z_2$ defined by the families of linear forms $\{Y_1^1(q_1),\ldots, Y_{s_1}^1(q_1)\}$ and $\{Y_1^2(q_2),\ldots, Y_{s_2}^2(q_2)\}$.

\begin{defin} \label{def2.1}
We say that the family $\mathcal{L}_{q_2}^2$ is an extension of the family $\mathcal{L}_{q_1}^1$ if $Z_1\leq Z_2$ and the following holds. The family of forms $L_2^i(q_2,x),  Y_j^2(q_2)$ which depend only on the variables $q_1=\pi(q_2)$ is exactly the family of forms $L_1^i(q_1,x)$, $Y_j^1(q_1)$, where $\pi:Z_2 \rightarrow Z_1$ is the natural orthogonal projection.
\end{defin}

\medskip

\noindent If $V=V_J$ let $\mathbf{\mu}^1:=\{\mu_{{q_1},e}\}_{q_1\in Z_1,e\in\HH}\ $ and $\mathbf{\mu}^2:=\{\mu_{{q_2},f}\}_{q_2\in Z_2,f\in \HH}\ $ be the associated measure systems as defined in \eqref{2.1.5}. We say that the measure system $\mathbf{\mu}_2$ is an \emph{extension} of the system $\mathbf{\mu}_1$.
\\\\
Let us make a few immediate observations. Writing $Z_2=Z_1\times \Z$, $Z=\Z_N^r$ and $q_2=(q_1,q)$, we have
\eq \label{2.2.1}
\psi_2(q_1,q)=\psi_1(q_1)\cdot \varphi(q_1,q),
\ee
where $\varphi(q,q_1)=\prod_{i=1}^t\nu(Y_i(q_1,q))$. The linear forms $Y_i(q_1,q)$ defining $\varphi(q,q_1)$ depend on some of the variables of $q=(q_i)_{1\leq i\leq k}$ and are pairwise linearly independent. Similarly one may write for any $e \in \HH$
\eq \label{2.2.2}
\mu^2_{(q_1,q),e}(x_e)=\mu^1_{q_1,e}(x_e)\w_e(q_1,q,x_e),
\ee
where the linear forms $L_2^j(q_1,q,x_e)$ defining the function $\w_e(q,q_1,x_e)$ depend on (some of) the variables $q$ as well as on (all of) the variables $x_e$.

\noindent In the special case when $\mathcal{L}=(L^1(x),\ldots,L^s(x))$ is a family of linear forms, a parametric family $\mathcal{L}_{q}$ is called an extension of $\LL$ if the set of forms in $\mathcal{L}_{q}$ which are independent of $q$ is exactly the family $\mathcal{L}.$ Similarly, the associated system of weights $\{\nu_{q,e}\}$ and measures $\{\nu_{q,e}\}$ is referred to as an extension of $\{\nu_e\}$ and $\{\mu_e\}$.

\begin{lemma}\label{lemma2.3}
Let $\{\mu_f\}_{f \in \HH}$ be a well defined measure system, and let $\{\mu_{q,f}\}_{q\in Z, f\in\HH}$ be a well-defined parametric extension of $\{\mu_f\}_{f \in \HH}$ of complexity at most $K$. Then for any $f \in \HH$ and for any function $g:V_f \rightarrow [-1,1]$ there is a set $\e_{g,f} \subseteq Z$ of measure $\psi (\e_{g,f})=o_{K}(1)$, so that for all $q \notin \e_{g,f}$
\eq \label{2.2.3}
\int_{V_f} g\, d\mu_{q,f} - \int_{V_f} g\, d\mu_f = o_{K}(1).
\ee\\
Similarly if $\{\mu_{q_1,f}\}_{f \in \HH,q_1\in Z_1}$ is a well-defined parametric system and if $\{\mu_{q_2,f}\}_{f\in\HH,q_2\in Z_2}$ is an extension of complexity at most $K_2$, then to any function $g:Z_1 \times V_f \rightarrow [-1,1]$ there exists a set $\e_{g,f} \subseteq Z_2$ of measure $\psi_2(\e_{g,f})=o_{K_2}(1)$, such that for all $q_2=(q_1,q) \notin \e_{g,f}$
\eq \label{2.2.4}
\int_{V_f}g(q_1,x)\,d\mu_{q_2,f}(x) - \int_{V_f}g(q_1,x)\,d\mu_{q_1,f}(x) = o_{K_2}(1).
\ee
\end{lemma}

\begin{proof}
As $\mu_{q,f}=\mu_f(x_f)\w_f(q,x_f)$, the left side of \eqref{2.2.3} may be written as
\[
\La_{f,g}(q):= \int_{V_f}g(x)(\w_f(q,x)-1)\,d\mu_f(x).
\]
Consider
\[
\La_{f,g} := \int_Z|\La_{f,g}(q)|^2\, d\psi(q).
\]
Using the Cauchy-Schwartz inequality we estimate
\begin{align*}
\La_{f,g} &= \int_Z \int_{V_f} \int_{V_f} (\w_f(q,x)-1)(\w_f(q,y)-1)g(x)g(y)\,    d\mu_f(x)d\mu_f(y) d\psi(q) \\
&\leq \int_{V_f} \int_{V_f} \left|\int_Z (\w_f(q,x)-1)(\w_f(q,y)-1)d\psi(q)\right|\,   d\mu_f(x)d\mu_f(y).
\end{align*}
Now the Cauchy-Schwartz inequality and \eqref{2.1.1} give
\begin{align*}
|\La_{f,g}|^2 &\lesssim \int_{V_f}\int_{V_f}\int_Z\int_Z (\w_f(q,x)-1)(\w_f(q,y)-1) \\  &\times\ (\w_f(p,x)-1)(\w_f(p,y)-1)\ d\mu_f(x)d\mu_f(y)d\psi(q)d\psi(p).
\end{align*}
\\
This last expression is a combination of 16 terms where each term is $1+o_{K}(1)$ by the linear form conditions. Indeed the linear forms which can appear in any of these terms are $Y_{i_1}(q)$,$Y_{i_2}(p)$,$L^{i_3}(x)$,$L^{i_4}(y)$, $L^{i_5}(q,x)$, $L^{i_6}(q,y)$,  $L^{i_7}(p,x)$, $L^{i_8}(p,y)$. Note that the last 4 terms depend on both sets of variables (for example $L^i(q,x)$ depends both on $q \in Z$ and on $x \in V_f$), and hence the family of these forms are pairwise linearly independent in the $(q,p,x,y)$ variables. This Proves \eqref{2.2.3}.\\\\
The proof of \eqref{2.2.4} is essentially the same. Set
\[
\La_{f,g}(q_2):= \int_{V_f}g(q_1,x)d\mu_{q_2,f}(x)-\int_{V_f}g(q_1,x)d\mu_{q_1,f}(x)
\]and
\[
\La_{f,g} := \int_Z  |\La_{f,g}(q_2)|^2\,  d\psi_2(q_2).
\]
Write $Z_2=Z_1\times Z$, where $Z=\Z_N^k$, and  $q_2=(q_1,q)$ for $q_2\in Z_2$. By \eqref{2.2.1} we estimate as above
\begin{align*}
\La_{f,g} \lesssim \int_{V_f}\int_{V_f}\int_{Z_1}\,d\psi_1(q_1)d\mu_{q_1,f}(x)d
\mu_{q_1,f}(y)\,\left|
\EE_{q\in Z}\,(\w_f(q_1,q,x)-1)(\w_f(q_1,q,y)-1)\varphi(q_1,q)\right|.
\end{align*}
The linear forms condition gives
\[
\int_{V_f} \int_{V_f} \int_{Z_1} d\psi_1(q_1)d\mu_{q_1,f}(x)d\mu_{q_1,f}(y)=1+o_{K_2}(1),
\]
thus we have
\begin{align*}
|\La_{f,g}|^2 &\lesssim
\int_{V_f}\int_{V_f} \int_{Z_1} \E_{p,q\in Z}\ (\w_f(q_1,q,x)-1)(\w_f(q_1,q,y)-1) \\
&\times (\w_f(q_1,p,x)-1)(\w_f(q_1,p,y)-1)\ \varphi(q_1,q)\varphi(q_1,p)\,d\psi_1(q_1)d\mu_{q_1,f}(x)d\mu_{q_1,f}(y).
\end{align*}\\
The point is that any linear form $L^i_f(q_1,q,x)$ depends both on the variables $q$ and $x$. Thus again the left side is a combination of 16 terms, each being $1+o_{K_2}(1)$ by the linear forms condition as all the linear forms involved in any of these expressions are pairwise linearly independent in the $(x,y,q_1,q,p)$ variables.
\end{proof}

\noindent Lemma \ref{lemma2.3} is an example of what we refer to as a \emph{stability} property. It means that the extension measures $\mu_{(q_1,q),f}$ are small perturbations of the measures $\mu_{q_1,f}$ with respect to quantities which are independent of $q$.\\

As a first application of this principle we show that the weighted box norms, defined in \eqref{1.4.2}, remain essentially unchanged under parametric extensions of the weight systems defining the norms. Let $\LL_{q_1}$ be a pair-wise linearly independent family of forms defined on the parameter space $(Z_1,\psi_1)$ and let $\{\nu_{q_1,e}\}$ be the associated system of weights.

Let $g:Z_1\times V_e \rightarrow \R$ be a function and let
$e\in\HH$, $|e|=d'$. For a given $q_1\in Z_1$ recall the box norm of $g_{q_1}(x)=g(q_1,x)$

\eq \label{2.2.5}
\big\| g_{q_1} \big\|_{\Box_{\nu_{q_1,e}}}^{2^{d'}}
= \E_{p,x \in V_e}\prod_{\omega \in \{0,1\}^e} g(q_1,\om_e(p,x))\,\prod_{f \subs e}\
\prod_{\omega_f \in \{0,1\}^f} \nu_{q,f}(\omega_f(p_f,x_f)),
\ee
where $x_f=\pi_f(x),\ p_f=\pi_f (p)$, $\pi_f:V_e\to V_f$ being the natural projection. The inner product on the right side of \eqref{2.2.5} is defined by the parametric family of forms
\eq\label{2.2.6}
\tilde{\LL}_{q_1}=\bigcup_{f\subs e} \{L(q_1,\om_f(p_f,x_f)):\ L\in \LL_{q_1},\, supp_x (L)=f,\, \om\in \{0,1\}^f\}.
\ee
It is easy to see that this is a pairwise linearly independent family of forms defined over $\ Z_1\times V$, $V=V_e\times V_e$. Indeed, if we'd have that
\eq\label{2.2.7}
L'(q_1,\om'_{f'}(p_{f'},x_{f'}))=\la L(q_1,\om_f(p_f,x_f)),
\ee
\\
then restricting both forms to the subspace $\{p=x\}$ would imply that $L'(q_1,x_{f'})=\la L(q_1,x_f)$ and hence $f'=supp_x(L')=supp_x(L)=f$. Then, as $L$ and $L'$ depend exactly on the variables $x_j$ for which $j\in f$, we should have $\om'=\om$ and $L=L'$.
\\
If $\ \{\tilde{\mu}_{q_1,f}\}_{q\in Z_1, f\subs e}\ $ denotes the associated system of measures and
\eq
G(q_1,p,x):= \prod_{\omega \in \{0,1\}^e}\,g(q_1,\om_e(p,x)),
\ee
then for given $q_1\in Z_1$
\eq \label{2.2.9}
\big\| g_{q_1} \big\|_{\Box_{\nu_{q_1,e}}}^{2^{d'}}
= \E_{p,x \in V_e} G_{q_1}(p,x)\,\tilde{\mu}_{q_1,e}(p,x).
\ee\\
Now, if $\LL_{q_2}$ is a well-defined parametric extension of $\LL_{q_1}$ then \eqref{2.2.6} yields to a well-defined parametric extension $\tilde{\LL}_{q_2}$ of the family $\tilde{\LL}_{q_1}$. Then by Lemma \ref{lemma2.3}, and the simple observation that $|a^{2^{d'}}-b^{2^{d'}}|\leq \eps$ implies $|a-b|\leq \eps^{2^{-d'}}$ for $a,b\geq 0$, we obtain

\begin{lemma}\label{lemma2.4} Let $\{\nu_{q_1,f}\}_{f\in\HH,q_1\in Z_1}$ be a parametric weight system with a well-defined extension $\{\nu_{q_2,f}\}_{f\in\HH,q_2\in Z_2}$  of complexity at most $K_2$. Then to any $e \in \HH$ and to any function $g:Z_1 \times V_e \rightarrow [-1,1]$ there exists a set $\e=\e(g,e) \in Z_2$ of measure $\psi_2(\e)=o_{K_2}(1)$ such that for all $q_2=(q_1,p) \notin \e$

\eq \label{2.2.10}
\big\|g_{q_1} \big\|_{\Box_{\nu_{q_2,e}}} = \big\| g_{q_1} \big\|_{\Box_{\nu_{q_1,e}}}+o_{K_2}(1).
\ee\\
\end{lemma}


\noindent Let $(V,\B,\mu)$ be a measure space and let $g:V \rightarrow \R$ be a function. An important construction, the so-called conditional expectation function is defined as
\[
\E_{\mu}(g|\B)(x)=\frac{1}{\mu(B(x))}\ \E_{y \in V}\1_{B(x)}(y)g(y)d\mu(q)=\frac{1}{\mu(B(x))} \int_{B(x)} g(y) d\mu(y),
\]
where $B(x) \in \B$ is the atom containing $x$. If $\mu(B(x))=0$ then we set $\E_{\mu}(g|\B)(x)=1$.\\

\noindent The complexity of the $\sigma$-algebra $\B$, denoted by compl($\B$), is defined as the minimum number of elements of $\B$ which generates $\B$. Note that the number of atoms of $\B$ is at most $2^{\text{compl}(\B)}$. Next we compare the conditional expectation functions of parametric systems.

\begin{lemma}\label{lemma2.5}
Let $(\mu_{q_1,f})_{q_1\in Z_1,f\in\HH}$ be a well-defined parametric measure system with a well-defined extension $(\mu_{q_2,f})_{q_2\in Z_2,f\in\HH}\,$  of complexity at most $K_2$. For $q_1\in Z_1$ and $e \in \HH$, let $\B_{q_1,e}$ be a $\sigma$-algebra on $V_e$ such that compl($\,\B_{q_1,e}$)$\leq M\,$  for some fixed number $M$. For any function $g:Z_1 \times V_e \rightarrow [-1,1]$ there exists a set $\e=\e(\B,g) \subs Z_2$ of measure $\psi_2(\e)=o_{M,K_2}(1)$ such that for any $\ q_2=(q_1,q) \notin \e$\\

\begin{enumerate}
\item we have
\eq \label{2.2.11}
\big\|\E_{\mu_{q_2,e}}(g_{q_1}|\B_{q_1,e})-\E_{\mu_{q_1,e}}
(g_{q_1}|\B_{q_1,e})\big\|^2
_{L^2(\mu_{q_2,e})}=o_{M,K_2}(1)
\ee

\item and
\eq \label{2.2.12}
\big\|\E_{\mu_{q_2,e}}(g_{q_1}|\B_{q_1,e})\big\|^2_{L^2(\mu_{q_2,e})}
=\big\|\E_{\mu_{q_1,e}}(g_{q_1}|\B_{q_1,e})\big\|^2_{L^2(\mu_{q_1,e})}+o_{M,K_2}(1).
\ee\\
\end{enumerate}
\end{lemma}

\begin{proof}
Let $m=2^M$ and enumerate the atoms of $\B_{q_1,e}$ as $B_{q_1}^1,...,B_{q_1}^m$, allowing some of them to possibly  be empty. For a fixed $1 \leq i \leq m$ define the function
\[
b_i(q_1,x)=\1_{B^i_{q_1}(x)}=\begin{cases}1 &\text{if} \quad x \in B_{q_1}^i,\\ 0 &\text{otherwise}  \end{cases}
\]
and for $q_2=(q_1,q) \in Z_2$ define the quantities
\[
\mu_i(q_2,g) := \int_{V_e}g(q_1,x)b_i(q_1,x)d\mu_{q_2,e}(x),\quad\quad
\mu_i(q_2):= \mu_{q_2,e}(B^i_{q_1}),
\]
\[
\mu_i(q_1,g) := \int_{V_e}g(q_1,x)b_i(q_1,x)d\mu_{q_1,e}(x),\quad\quad
\mu_i(q_1):= \mu_{q_1,e}(B^i_{q_1}).
\]
By Lemma 2.3 we have that
\eq \label{2.2.13}
\mu_i(q_2,g)=\mu_i(q_1,g)+o_{K_2}(1),\quad \mu_i(q_2)= \mu_i(q_1)+o_{K_2}(1)
\ee
for all $q_2 \notin \e_i$ where $\e_i \subseteq Z_2$ is a set of $\psi_2$-measure $o_{K_2}(1).$ Let $\e=\bigcup_{i=1}^m \e^i$
then $\psi_2(\e)=o_{K_2,M}(1).$ The left hand side of \eqref{2.2.11} takes the form
\eq \label{2.2.14}
\sum_{i=1}^m \bigg(\frac{\mu_i(q_2,g)}{\mu_i(q_2)}-\frac{\mu_i(q_1,g)}
{\mu_i(q_1)}\bigg)^2 \mu_i(q_2),
\ee
with the convention that if $\mu_i(q_1)=0$ or $\mu_i(q_2)=0$ then $\mu_i(q_1,g)/\mu_i(q_1):=1$ or $\mu_i(q_2,g)/\mu_i(q_2):=1.$ \\\\
If $q_2=(q_1,q) \notin \e$ then by \eqref{2.2.13}
\eq \label{2.2.15}
\eps:=\sum_{i=1}^m|\mu_i(q_2,g)-\mu_i(q_1,g)|+|\mu_i(q_2)-\mu_i(q_1)| =o_{K_2,M}(1).
\ee
Now if $\,\mu_i(q_1) \leq 2\eps^{1/4}\,$ then $\,\mu_i(q_2) \leq 3\eps^{1/4}\ $ by \eqref{2.2.13}, hence the total contribution of such terms is bounded by $\ 12m\,\eps^{1/4}=o_{K_2,M}(1).$ \\\\
If $\mu_i(q_1) \geq 2\eps^{1/4}$ then $\mu_i(q_2) \geq \eps^{1/4}$, then we have the estimate

\[
\bigg|\frac{\mu_i(q_2,g)}{\mu_i(q_2)}-\frac{\mu_i(q_1,g)}{\mu_i(q_1)} \bigg| \leq \frac{4\eps(N)}{2\eps(N)^{1/2}} \leq 2\eps^{1/2}=o_{K_2,M}(1).
\]\\
This proves \eqref{2.2.11}. The proof of inequality \eqref{2.2.12} proceeds the same way, here one needs to estimate the quantity
\eq \label{2.2.16}
\sum_{i=1}^m \bigg|\frac{\mu_i(q_2,g)^2}{\mu_i(q_2)}-\frac{\mu_i(q_1,g)^2}{\mu_i(q_1)} \bigg|=\sum_{i=1}^m \bigg|\bigg( \frac{\mu_i(q_2,g)}{\mu_i(q_2)} \bigg)^2\mu_i(q_2)-\bigg(\frac{\mu_i(q_1,g)}{\mu_i(q_1)}\bigg)^2\mu_i(q_1) \bigg|.
\ee

\noindent If $\mu_i{q_1} \leq 2\eps^{1/4}$ then $\mu_i(q_2) \leq 3\eps^{1/4}$ for $q_2=(q_1,q)\notin\e\,$, thus the contribution of such terms to the right side of \eqref{2.2.16} is trivially estimated by
\[
3m\,\eps^{1/4}=o_{M,K_2}(1).
\]
The rest of the terms are bounded by $8\,\eps^{1/2}$ and \eqref{2.2.12} follows.
\end{proof}

We also need an analogue of the above result when the $\| \cdot \|_{L^2(\mu_{q,e})}$ norm is replaced by the more complicated $\| \cdot \|_{\Box_{\nu_{q,e}}}$ norms.

\begin{lemma}\label{lemma2.6}
Let $\{\nu_{q_2,f}\}_{f\in\HH,q_2\in Z_2}$ be a well-defined extension of the parametric weight system $\{\nu_{q_2,f}\}_{f\in\HH,q_1\in Z_1}$, of complexity at most $K_2$. For $q_1\in Z_1$ and $e\in\HH$, let $\B_{q_1,e}$ be a $\si$-algebra of complexity at most $M$, for some fixed constant $M>0$. Then
\eq \label{2.2.17}
\|\E_{\nu_{q_2,e}}(g_{q_1}|\B_{q_1,e})- \E_{\nu_{q_1,e}}(g_{q_1}|\B_{q_1,e})\|_{\Box_{\nu_{q_2,e}}}
=o_{M,K}(1),
\ee
for all $q_2=(q_1,q) \notin \e$, where $\e=\e(g,\B)\subs Z_2$ is a set of measure $\psi_2(\e)=o_{M,K_2}(1).$
\end{lemma}

\begin{proof}
First we show that for any family of sets $\A=(A_{q_1})_{q_1\in Z_1}$, $A_{q_1}\subseteq V_e$ there is a set $\e_1=\e_1(g,\A)$ of measure $\psi_2(\e_1)=o_{K_2}(1)$ such that for all $q_2 = (q_1,q) \notin \e_1$ we have
\eq  \label{2.2.18}
\| \1_{A_{q_1}} \|_{\Box_{\nu_{q_2,e}}}^{2^d} \leq \mu_{q_2,e}(A_{q_1})+o_{K_2}(1).
\ee
To see this, first note that for $q_2=(q_1,q)\in Z_2$ one has
\begin{align*}
\| \1_{A_{q_1}} \|_{\Box_{\nu_{q_2,e}}}^{2^d} &\leq \E_{x,p \in V_e} \1_{A_{q_1}}(x)\, \mu_{q_2,e}(x)\prod_{f \subseteq e}\prod_{\omega_f \neq 0}\nu_{q_2,f}(\omega_f(p_f,x_f)) \\
&=\mu_{q_2,e}(A_{q_1})+E(q_2),
\end{align*}
with
\[
E(q_2) \leq \E_{x \in V_e} \mu_{q_2,e}(x)|\E_{p \in V_e}(\WW(q_2,p,x)-1)|,
\]
where
\[
\WW(q_2,p,x)=\prod_{f \subseteq e}\prod_{\omega_f \neq 0}\nu_{q_2,f}(\omega_f(p_f,x_f)).
\]
Arguing as in Lemma 2.6, we see that
\[
\E_{q_2 \in Z_2} \E_{x,p,p' \in V_e} \,\psi_2(q_2)d\mu_{q_2,e}(x)\ (\WW(q_2,p,x)-1)(\WW(q_2,p',x)-1)=o_{M,K}(1)
\]
and \eqref{2.2.18} follows.
\\\\
Now let $\{ B_{q_1}^i \}_{i=1}^m$ ($m=2^M$) be the atoms of $\B_{q_1,e}$ and define the quantities $\mu_i(q_2,g),\mu_i(q_2),\mu_i(q_1,g),\mu_i(q_1)$  as in Lemma 2.4. The expression in \eqref{2.2.11} is then estimated
\begin{align*}\label{2.2.13}
\bigg\| \sum_{i=1}^m \bigg(\frac{\mu_i(q_2,g)}{\mu_i(q_2)}-\frac{\mu_i(q_1,g)}{\mu_i(q_1)} \bigg)\1_{B_{q_1}^i} \bigg\|_{\Box_{\nu_{q_2,e}}}
&\leq \sum_{i=1}^m \bigg|\frac{\mu_i(q_2,g)}{\mu_i(q_2)}-\frac{\mu_i(q_1,g)}{\mu_i(q_1)} \bigg| \big\| \1_{B_{q_1}^i} \big\|_{\Box_{\nu_{q_2,e}}} \notag \\
&\lesssim \sum_{i=1}^m \bigg|\frac{\mu_i(q_2,g)}{\mu_i(q_2)}-\frac{\mu_i(q_1,g)}{\mu_i(q_1)} \bigg|\ \mu_{q_2,e}(B_{q}^i)^{2^{-d}} + o_{M,K}(1),
\end{align*}
for $q_2=(q_1,q) \notin \e_1$, where $\e_1=\e_1(\B_{q_1,e},g)$ is a set of measure $o_{M,K}(1)$.
\\\\
Using the facts that $\mu_i(q_2,g)=\mu_i(q_1,g)+o_{K_2}(1)$ and $ \mu_i(q_2)=\mu_i(q_1)+o_{K_2}(1)$ outside a set of measure $o_{M,K_2}(1)$, and
\[
\sum_{i=1}^m \mu_{q_2,e}(B_{q_1}^i)=\mu_{q_2,e}(V_e)=1+o_{K_2}(1),
\]
it follows that the above expression is $o_{M,K_2}(1)$ by arguing as in Lemma 2.4. This completes the proof.
\end{proof}

\subsection{Symmetric extensions.} We will also need our parametric families of forms to be symmetric, to apply Theorem \ref{Thm1.3}, which we define as follows. Let for each $e\in \HH_d$, $\LL_{q,e}=\{L_e^1(q,x),...,L_e^s(q,x)\}$ be a pairwise linearly independent family of linear forms defined on $V_J$, depending on parameters $q \in Z$, such that $supp_x (L_e^j)\subs e$. We say that the family of forms $\LL_q=\bigcup_{e\in\HH_d} \LL_e$ is \emph{symmetric} if $L_e^j(q,x)=L_{e'}^j(q,x)$ for all $q\in Z$, $x\in M$, $e,\,e'\in\HH_d$ and $1\leq j\leq s$. Note that our initial family of forms defined in \eqref{1.4.2} has this property.

It is not hard to see that to a given family of forms $\LL_{q,e}$ , for a fixed $e\in \HH_d$, there is a unique symmetric family of forms $\LL_q$ such that $\LL_{q,e}=\{L\in \LL_q;\ supp_x(L)\subs e\}$. Indeed, if $\LL_q$ is such a family, then for $e'\in \HH_d$, $q\in Z$ and $x\in V_J$
\eq\label{2.3.1}
L_{e'}^j(q,x)=L_{e'}^j(q,\pi_{e'}(x))=L_{e'}^j(q,\phi_{e'}\circ \pi_{e'}(x))=
L_{e}^j(q,\phi_{e'}\circ \pi_{e'}(x)),
\ee
where $\phi_e:V_e\to M$ is the inverse of the projection $\pi_{e'}$ restricted to $M$. This shows the uniqueness of the family $\LL_q$. Conversely, define $L_{e'}^j(q,x)$ by the above equality, then it is clear that $supp_x(L_{e'}^j)\subs e'$, moreover if $x\in M$ then $x=\phi_{e'}\circ \pi_{e'}(x)$ hence $L_{e'}^j(q,x)=L_e^j(q,x)$. Also, if $supp_x L_{e'}^j\subs e$ then for all $q\in \Z_1$ and $x\in V_J$
\[L_{e'}^j(q,x)=L_{e'}^j(q,\phi_e\circ\pi_e(x))=L_e^j(q,\phi_e\circ\pi_e(x))=
L_e^j(q,x)
,\]
This shows that all forms in $\LL_q$ which depend only on the variables $x_e$ are the forms of $\LL_{q,e}$. Finally, if $\LL_{q,e}$ is a pairwise linearly independent family then so is $\LL_q$, as linearly dependent forms must depend on the same set of variables. We will refer to the family of forms $\LL_q$ as the \emph{symmetrization} of the family $\LL_{q,e}$. If $f\in\HH$ for some edge $|f|=d'\leq d$ and $\LL_{q,f}$ is a family of forms defined on $V_f$ then the above construction can be applied to obtain a symmetric family $\LL_q$ simply by choosing an $e\in\HH_d$ such that $f\subs e$ and considering $\LL_{q,f}$ as a family of forms on $V_e$. Note that the construction is independent of the choice of $e\supseteq f$, as if $f\subs e'$ as well then $L_e^j=L_{e'}^j$ for all $1\leq j\leq s$.
\\\\
In the next section, following \cite{T1} we will start an energy increment process to obtain a regularity lemma for weighted hypergraphs. At each stage we will pass to an extension of a symmetric, well-defined and pairwise independent parametric family $\LL_q$ defined for $q\in Z$ as follows. We choose an edge $e\in\HH$ and consider the extension of the family $\LL_{q,e}$ as given in \eqref{2.2.6}, that is replacing the forms $L^j(q,x_f)$ with the forms $L^j(q,\om_f(p_f,x_f))$, $\om\in \{0,1\}^f$. This gives an extension $\tilde{\LL}_{q,p,f}$ defined on the parameter space $(q,p)\in Z\times V_f$, which we symmetrize to obtain a new symmetric, well-defined and pairwise independent family $\tilde{\LL}_{q,p}$. The first step of this process was described in the introduction in the special case $e=(1,2)$.

\medskip

\noindent
\section{Regularization of Parametric Systems}

\subsection{A Koopman-von Neumann type decomposition} Let $e\subs J$ and let $\B_f$ be a $\si$-algebra on $V_f$ for $f\in\partial e$, where $\partial e=\{f\subs e:\ |f|=|e|-1\}$ denotes the boundary of the edge $e$. Let $\B:=\bigvee_{f\subs \partial e} \B_f$ be the $\si$-algebra generated by the sets $\pi_{ef}^{-1}(\B_f)$ where $\pi_{ef}:\,V_e\to V_f$ is the canonical projection. The atoms of $\B$ are the sets $G=\bigcap_{f\subs \partial e} \pi_{ef}^{-1}(G_f)$ with $G_f$ being an atom of $\B_f$, which may be interpreted as the collection of simplices $x_e\in V_e$ whose faces $x_f$ are in $G_f$ for all $f\in \partial e$.

The starting point of the proof of the Regularity Lemma, given in \cite{T1}, is to show that if a set $G_e\subs V_e$ is not sufficiently regular with respect to $\B$, that is if
\eq\label{3.1.1}
\big\|\1_{G_e}-\E(\1_{G_e}|\bigvee_{f \in \partial e}\B_{f}) \big\|_{\Box} \geq \eta,
\ee
then there exist $\si$-algebras $\B'_f\supseteq \B_f$ for $f\in\partial e$, such that
\[
\| \E (\1_{G_{e}}| \bigvee_{f \in \partial e} \B'_{f}) \big\|_{L^2}^2
\geq \big\| \E (\1_{G_{e}}| \bigvee_{f \in \partial e} \B_{f}) \big\|_{L^2}^2+c\eta^2.
\]
The quantity $\ \| \E (\1_{G}|\B)\|_{L^2}^2\,$ is referred to as the \emph{energy} (or \emph{index}) of the set $G$ with respect to the $\si$-algebra $\B$, thus the above inequality means that the energy of the set $G_e$ is increased by $c\eta^2$ by refining the $\si$-algebras $\B_f$. In addition the complexity of the $\si$-algebras $\B'_f$, denoted by $compl (\B'_f)$ and defined as the minimal number of sets generating the $\si$-algebra, is at most 1 larger than that of $\B_f$.

In our settings for a given $e\subs J$ we will have a parametric system of weights $\{\nu_{q,f}\}_{q\in Z,f\subs e}$ and measures $\{\mu_{q,f}\}_{q\in Z,f\subs e}$ associated to a well-defined, pairwise linearly independent family of of forms $\LL_q$ defined on $Z\times V_e$, as given in \eqref{2.1.5}. For simplicity we will refer to such systems of weights and measures as being \emph{well-defined}.


\begin{lemma}
For given $e\subs J$,  $|e|=d'$, let $\{\mu_{q,f}\}_{q \in Z, f\subs e}\ $ be a well-defined family of measures of complexity at most $K$. For $q\in Z$ let $G_{q,e}\subs V_e$ and $\{\B_{q,f}\}_{f\in\partial e}$ be a $\sigma$-algebra on $V_f$.
\\\\
Assume
\eq \label{3.1.2}
\big\| \1_{G_{q,e}}-\E_{\mu_{q,e}}(\1_{G_{q,e}}| \bigvee_{f \in \partial e}\B_{q,f}) \big\|_{\Box_{\nu_{q,e}}}^{2^{d'}} \geq \eta,
\ee
for some $\eta>0$ and for each $q \in \Omega$, where $\Omega \subseteq Z$ is a set of measure $\psi(\Omega) \geq c_0 >0$.
\\\\
Then for $N,W$ sufficiently large with respect to the parameters $c_0,\eta$,  there exists a well-defined extension $\{\mu_{q',f} \}_{q' \in Z',f\subs e}\ $ of the system $\{\mu_{q,f}\}$ of complexity $K'=O(K)$, and a set $\Omega' \subseteq \Omega \times V_e \subseteq Z'$ such that the following hold.

\begin{enumerate}

\item We have
\eq \label{3.1.3}
\psi'(\Omega') \geq 2^{-4}c_0^2\eta^2,
\ee
where $\psi'$ is the measure on the parameter space $Z'=Z\times V_e$.\\

\item
For all $q'=(q,p) \in Z'$ and $f\in\partial e$ there is a $\sigma$-algebra $\B_{q',f} \supseteq \B_{q,f}$ of complexity
\eq \label{3.1.4}
\text{compl}(\B'_{q,f}) \leq \text{compl}(\B_{q,f})+1.
\ee

\item For all $q'= (q,p) \in \Omega',$ one has
\eq\label{3.1.5}
\big\| \E_{\mu_{q',e}} (\1_{G_{q,e}}| \bigvee_{f \in \partial e} \B_{q',f}) \big\|_{L^2(\mu_{q',e})}^2
\,\geq\, \big\| \E_{\mu_{q,e}} (\1_{G_{q,e}}| \bigvee_{f \in \partial e} \B_{q,f}) \big\|_{L^2(\mu_{q,e})}^2 + 2^{-8}\,\eta^2,
\ee

\item and
\eq \label{3.1.6}
\mu_{q',e}(V_e) \leq 2.
\ee\\
\end{enumerate}
\end{lemma}

\noindent The meaning of the above lemma is that if there is a large ``bad" set $\Omega$ of parameters $q$ for which the set $G_{q,e}$ is not sufficiently uniform with respect to the $\sigma$-algebra $\bigvee_{f \in \partial e}\B_{q,f}$, then its energy will increase by a fixed amount when passing to a well defined extension $\{\B_{q',f}\},\,\{\mu_{q',e}\}$, for all $q'=(q,p)\in \Om'$.

\begin{proof} Let

\eq \label{3.1.7}
g_{q} := \1_{G_{q,e}}-\E_{\mu_{q,e}}(\1_{G_{q,e}}|\bigvee_{f \in \partial e}\B_{q,f}).
\ee

\noindent Then by \eqref{2.2.5} we have for each $q \in \Omega$
\eq \label{3.1.8}
\big\| g_{q}\big\|_{\Box_{\nu_{q,e}}}^{2^{d'}} = \int_{V_e}\langle g_{q},\prod_{f \in \partial e}u_{q,p,f} \rangle_{\mu_{(q,p),e}} d\mu_{q,e}(p) \geq \eta,
\ee
where $u_{q,p,f}:V_e \rightarrow [-1,1]$ are functions, and $\{\mu_{(q,p),e}\}_{(q,p)\in Z'}$ is the family of measures
\[\mu_{(q,p),e}(x)= \prod_{f\subs e}\prod_{\substack{\om\in\{0,1\}^f\\\om\neq 0}} \nu_{q,f}(p_f,x_f).\]
\noindent As explained after \eqref{2.1.5} the measures $\mu_{(q,p),e}$ are defined by a pairwise independent family of forms $\LL_{(q,p),e}$ depending on the parameters $(q,p)\in Z\times V_e$, which is a well-defined extension of the family $\LL_{q,e}$ defining the measures $\mu_{q,e}$. It is clear from \eqref{3.1.8} that the measure $\psi'$ on $Z'$ has the form $\psi'(q,p)=\mu_{q,e}(p)\,\psi(q)$.
\\\\
For $q'=(q,p)$, let

\eq \label{3.1.9}
\Gamma(q,p) := \langle g_{q}, \prod_{f \in \partial e}u_{q,p,f} \rangle_{\mu_{q,p,f}}.
\ee

\noindent We show that there is a set $\Omega_1' \subseteq \Omega \times V_e$ of measure

\eq \label{3.1.10}
\psi'(\Omega_1') \geq 2^{-3}c_0^2\,\eta^2,
\ee

\noindent such that for every $(q,p) \in \Omega_1'$ one has

\eq \label{3.1.11}
\Gamma(q,p) \geq \frac{\eta}{4}.
\ee
By Lemma 2.2 we have that $\mu_{q,e}(V_e)=1+o_{K}(1) \leq 2$ for $q \notin \e_1$ where $\e_1 \subseteq \Omega$ is a set of measure $\psi(\e_1)=o_{K}(1)$. Thus for $q \in \Omega \backslash \e_1= \Omega_{1}$ we have by \eqref{3.1.8} that

\eq\label{3.1.12}
\int_{V_e} \1_{\{\Gamma(q,p) \geq \eta/4\} }\Gamma(q,p) d\mu_{q,e}(p) \geq \frac{\eta}{2},
\ee

\noindent where by \eqref{3.1.8} and  \eqref{3.1.9} we have

\[
\Gamma(q,p)= \int_{V_e}g_{q}(x)\prod_{f \in \partial e}u_{q,f}\ \w_{q,p}(x) d\mu_{q,e}(x).
\]\\
The function $\w_{q,p}(x)$ is the product of weight functions of the form $\nu(L(q,p,x))$ depending on both variables $p$ and $x$. Thus, using the bounds  $|g_{q}|\leq 1,|u_{q,p,f}| \leq 1$, one has

\begin{align}\label{3.1.13}
\int_Z \int_{V_e} |\Gamma(q,p)|^2 d\mu_{q,e}(p)d\psi(q)
&\leq \int_Z \int_{V_e} \int_{V_e} \int_{V_e} \w_{q,p}(x) \w_{q,p}(x')  d\mu_{q,e}(x)d\mu_{q,e}(x')d\mu_{q,e}(p)d\psi(q) \\
&=1+o_{K}(1) \leq 2 \nonumber
\end{align}
by the linear forms condition, as the factors in the product depend on different sets of variables. Let \\$\Om_1':=\{(q,p)\in\Om_1\times V_e:\ \Gamma(q,p)\geq \eta/4\}$. Thus by \eqref{3.1.12}, \eqref{3.1.13} and the Cauchy-Schwartz inequality
\[
\frac{c_0^2\eta^2}{4}\leq \left(\int_{\Om_1'} \Ga(q,p) d\mu_{q,e}(p)\,d\psi(q)\right)^2 \leq \int_{\Om_1'} \Ga(q,p)^2 d\mu_{q,e}(p)\,d\psi(q)\ \psi'(\Om_1')\leq 2\,\psi'(\Om_1').
\]
This shows $\psi'(\Om_1')\geq 2^{-3}c_0^2 \eta^2$ as claimed.
\\\\
Since $|u_{q',f}| \leq 1$, decomposing each function $u_{q',f}$ into its positive and negative parts yields that

\eq \label{3.1.14}
\langle g_{q}, \prod_{f \in \partial e} v_{q',f}\rangle_{\mu_{q',e}} \geq 2^{-2}\eta
\ee\\
for some functions $v_{q',f}:V_f \rightarrow [0,1]$. For a given $f \in \partial e$ and some $ 0 \leq t_f \leq 1\ $, let \[\U_{q',t_f}:= \{ x_f \in V_f : v_{q',f}(x_f) \geq t_f\}\] be the level set of the functions $v_{q',f}$. Then $v_{q',f}(x_f)=\int_{0}^1 \1_{\U_{q',t_f}}(x_f) dt_f$, and for each term in \eqref{3.1.14} we have

\[
\int_0^1 \cdots \int_0^1 \langle g_{q}, \prod_{f \in \partial e}\1_{\U_{q',t_f}}  \rangle_{\mu_{q',e}} dt \geq 2^{-2}\eta,
\]
where $t=(t_f)_{f \in \partial e}.$ Accordingly the integrand must be at least $2^{-d-2}\eta$ for some value of the parameter $t$. Fix such a $t=(t_f)$ and write $\U_{q',f}$ for $\U_{q',t_f}$ for simplicity of notation. For $q'=(q,p)\in\Om_1'$, define $\B_{q',f}$ to be the $\sigma$-algebra generated by $\B_{q,f}$, and the $\U_{q',t_f}$. For $q' \notin \Omega_1'$, set $\B_{q',f}=\B_{q,f}.$\\\\
The function $\prod_{f \in \partial e}\1_{\U_{q',f}}$ is constant on the atoms of the $\sigma$-algebra $\bigvee_{f \in \partial e}\B_{q',f}$, and therefore we have

\[
\langle \1_{G_{q,e}}-\E_{\mu_{q',e}}(\1_{G_{q,e}}| \bigvee_{f \in \partial e}\B_{q',f}), \prod_{f \in \partial e}\1_{\U_{q',f}} \rangle_{\mu_{q',e}}=0
\]
for $q' \in \Omega_1'$. Hence, by \eqref{3.1.7} and \eqref{3.1.14} it follows that

\eq \label{3.1.16}
\langle
\E_{\mu_{q',e}}(\1_{G_{q,e}}| \bigvee_{f \in \partial e}\B_{q',f})-\E_{\mu_{q,e}}(\1_{G_{q,e}}| \bigvee_{f \in \partial e}\B_{q,f}),  \prod_{f \in \partial e}\1_{\U_{q',f}} \rangle_{\mu_{q',e}} \geq 2^{-2}\eta.
\ee\\
By Lemma 2.2 there is a set $\e_1 \subseteq Z'$ such that $\psi'(\e_1)=o_{K}(1)$ and
\[
\big\| \prod_{f \in \partial e}\1_{\U_{q',f}} \big\|_{L^2(\mu_{q',e})} \leq \mu_{q',e}(V_e)^{1/2}=1+o_{K}(1) \leq 2
\]
for $q' \in \Omega_1' \backslash \e_1=: \Omega_2'.$ Then by the Cauchy-Schwartz inequality,

\[
\big\|\E_{\mu_{q',e}}(\1_{G_{q,e}}| \bigvee_{f \in \partial e}\B_{q',f})-\E_{\mu_{q,e}}(\1_{G_{q,e}}| \bigvee_{f \in \partial e}\B_{q,f})  \big\|_{L^2(\mu_{q',e})} \geq\,2^{-3}\eta,
\]\\
for $q' \in \Omega_2'$. By Lemma 2.6 there is an exceptional set $\e_2 \subseteq Z'$ of measure $\psi'(\e_2)=o_{K,M}(1)$ such that for $q'=(q,p) \in \Omega_3' := \Omega_2' \backslash \e_2$ we have

\eq \label{3.1.17}
\big\|\E_{\mu_{q',e}}(\1_{G_{q,e}}| \bigvee_{f \in \partial e}\B_{q',f})-\E_{\mu_{q',e}}(\1_{G_{q,e}}| \bigvee_{f \in \partial e}\B_{q,f})  \big\|_{L^2(\mu_{q',e})} \geq\, 2^{-3}\eta -o_{K,M}(1) \geq\,2^{-4}\eta.\ee
\\
Since $\B_{q,f} \subseteq \B_{q',f}$, for $q'=(q,p)$, \eqref{3.1.17} is equivalent to

\eq
\big\|\E_{\mu_{q',e}}(\1_{G_{q,e}}| \bigvee_{f \in \partial e}\B_{q',f})\big\|_{L^2(\mu_{q',e})}^2-
\big\|\E_{\mu_{q',e}}(\1_{G_{q,e}}| \bigvee_{f \in \partial e}\B_{q ,f})\big\|_{L^2(\mu_{q',e})}^2
\geq\,2^{-8}\eta^2.
\ee
\\
Finally, by a further invocation of Lemma 2.6 there is a set $\e_3 \subseteq Z'$ of measure $\psi'(\e_3)=o_{K,M}(1)$ such that for $q' \in \Omega_4' :=\Omega_3' \backslash \e_3$ we have (for $N,W$ sufficiently large)

\eq\label{3.1.18}
\big\|\E_{\mu_{q',e}}(\1_{G_{q,e}}| \bigvee_{f \in \partial e}\B_{q',f})\big\|_{L^2(\mu_{q',e})}^2-
\big\|\E_{\mu_{q,e}}(\1_{G_{q,e}}| \bigvee_{f \in \partial e}\B_{q ,f})\big\|_{L^2(\mu_{q,e})}^2
\geq 2^{-9}\eta^2.
\ee

This proves the lemma by choosing $\Om'=\Om_4'$.
\end{proof}

\noindent Iterating the above lemma leads to a parametric family of $\sigma-$algebras and measures such that the sets $G_{q,e}$  become sufficiently uniform with respect to them. The associated decomposition of their indicator functions is sometimes referred to as a Koopman-von Neumann type decomposition \cite{T1}. We will replace sets $G_e\subs V_e$ by $\si$-algebras $\B_e$ on $V_e$ for $e\in \HH_{d'}$ and for that it is useful to define the \emph{total energy} of the family $\{\B_e\}_{e\in\HH_{d'}}$ with respect to a family of lower order $\si$-algebras $\{\B_f\}_{f\in\HH_{d'-1}}$ and a family of measures $\{\mu_e\}_{e\in\HH_{d'}}$ as

\eq\label{3.1.19}
\sum_{e \in \HH_{d'}, G_e \in \B_e} \big\| \E_{\mu_e}(\1_{G_e}| \bigvee_{f \in \partial e} \B_f) \big\|^2_{L^2(\mu_e)}.
\ee

Assuming the measures $\mu_e$ are normalized i.e. $\mu_e(V_e)=1+o(1)\leq 2$, a crude upper bound for the total energy is $2^{d+1} 2^{2^M}=O_M(1)$, where $M$ is the complexity of the $\si$-algebras $\B_e$.

\begin{lemma}[Koopman-von Neumann decomposition]
Let $\{\mu_{q,f} \}_{q \in Z, f \in \HH}$ be a well-defined, symmetric family of measures of complexity at most $K$. Let $1 \leq d' \leq d$, and let $\{\B_{q,e}\}_{q \in Z, e \in \HH_{d'}}$ and $\{\B_{q,f}\}_{q \in Z, f \in \HH_{d'-1}}$ be families of $\sigma$-algebras of complexity at most $M_{d'}$ and $M_{d'-1}$. Finally let $\Omega \subseteq Z$ with $\psi(\Omega) \geq c_0>0$, and let $\de>0$ be a constant.
\\\\
Then for $N,\,W$ sufficiently large with respect to the constants $\de,c_0,M_{d'},M_{d'-1}$ and $K$, there exists a well-defined, symmetric extension $\{\mu_{q',f} \}_{q' \in Z', f \in \HH}$ of the system $\{\mu_{q,f} \}$ of complexity at most $K'=O_{M_{d'},K,\,\de}(1)$ and a family of $\sigma$-algebras $\{\B_{q',f} \}_{q' \in Z', f \in \HH_{d'-1}}$ such that the following hold.\\

\begin{enumerate}
\item For all $q'=(q,p) \in Z'$ and $f \in \HH_{d'-1}$ we have

\eq \label{3.1.20}
\B_{q,f} \subseteq \B_{q',f} ,\quad \text{compl}(\B_{q',f}) \leq \text{compl}(\B_{q,f})+O_{M_{d'},\,\de}(1).
\ee\\

\item There exists a set $\Omega' \subseteq \Omega \times V \subseteq Z'$ of measure $\psi'(\Omega') \geq c(c_0, \de,M_{d'})>0$ such that for all $q'=(q,p) \in \Omega'$ and for all $G_{q,e}\in \B_{q,e}$ one has

\eq \label{3.1.21}
\big\|\1_{G_{q,e}}-\E_{\mu_{q',e}}(\1_{G_{q,e}}| \bigvee_{f \in \partial e}\B_{q',f}) \big\|_{\Box_{\nu_{q',\,e}}} \leq \de,
\ee
and

\eq \label{3.1.22}
\big\|\E_{\mu_{q',e}}(\1_{G_{q,e}}| \bigvee_{f \in \partial e}\B_{q,f})\big\|_{L^2(\mu_{q',e})}^2
=\big\|\E_{\mu_{q,e}}(\1_{G_{q,e}}| \bigvee_{f \in \partial e}\B_{q,f})\big\|_{L^2(\mu_{q,e})}^2+o_{M_{d'},K,\,\de}(1).
\ee
\end{enumerate}
\end{lemma}

\medskip

\begin{proof} Initially set $Z'=Z$, then \eqref{3.1.20} and \eqref{3.1.22} trivially hold for $q'=q$.
If there is a set $\Omega_1 \subseteq \Omega$ of measure $\psi(\Omega_1) \geq \frac{c_0}{2}$ such that inequality \eqref{3.1.21} holds for all $q\in \Omega_1$ and $G_{q,e}\in\B_{q,e}$ then the conclusions of the lemma hold for the initial system of measures and $\si$-algebras $\{\mu_{q,f}\},\,\{B_{q,f}\}$ and the set $\Omega_1$. Otherwise there is a set $\Omega_2 \subseteq \Omega$ of measure $\psi(\Omega_2) \geq \frac{c_0}{2}$ such that for each $q \in \Omega_2$ there is an $e \in \HH_{d'}$ and a set $G_{q,e} \in \B_{q,e}$ for which the inequality \eqref{3.1.21} fails. By the pigeonholing we may assume that $e\in \HH_{d'}$ is independent of $q$.
Then by Lemma 3.1, with $\eta:=\de^{2^{d'}}$, there is a well-defined extension $\{\mu_{q',f} \}_{q' \in Z',f\subs e}\ $, a family of $\sigma$-algebras $\{ \B_{q',f} \}_{q' \in Z', f\subs e}$ and a set $\Om'\subs \Om_2$ for which \eqref{3.1.3}-\eqref{3.1.5} hold. Let $\{\mu_{q',f}\}_{q'\in Z',f\in \HH}$ be the symmetrization of the system $\{\mu_{q',f} \}_{q' \in Z',f\subs e}\ $ as described in section 2.3, and set $\B_{q',f}:=\B_{q,f}$ for $q'\notin \Om'$ or $f\nsubseteq e$. By Lemma 2.5 one may remove a set $\e$ of measure $\psi'(\e)=o_{M_{d'},K}(1)$ such that for all $q'\in\Om'\backslash \e$
\eqref{3.1.20} and \eqref{3.1.22} hold for the extended system, whose total energy is at least $2^{-10}\de^{2^{d'}}$ larger than that of the initial system $\{\mu_{q,f}\}_{q \in Z, f \in \HH}$ .
\\\\
Based on the above argument we perform the following iteration. Let $\{\mu_{q',f}\}_{q'\in Z',f \in \HH}\ $ be a well-defined, symmetric extension of the initial system $\{\mu_{q,f}\}_{q\in Z,f \in \HH}\ $, $\{\B_{q',f}\}_{q'\in Z',f\in\HH_{d'-1}}\ $ be a family of $\si$-algebras and let $\Om'\subs \Om\times V'\subs Z'$ for which \eqref{3.1.20} and \eqref{3.1.22} hold. If there is a set $\Omega_1' \subs \Om'$ of measure $\psi'(\Omega_1') \geq\psi(\Omega')/2$ such that for all $q \in \Omega_1'$,  $e \in \HH_{d'}$ and $G_{q,e} \in \B_{q,e}$ inequality \eqref{3.1.22} holds, then the system $\{\mu_{q',f}\}$, $\{\B_{q',f}\}$ together with the set $\Om_1'$ satisfies the conclusions of the lemma.
\\\\
Otherwise there is a well-defined, symmetric extension $\{\mu_{q'',f} \}_{q'' \in Z'', f \in \HH}$ together with a family of $\sigma$-algebras $\{ \B_{q'',f} \}_{q'' \in Z'', f \in \HH_{d'-1}}$ and a set $\Omega''\subs \Omega' \times \Z_N^{d'}$ such that for all $q''\in\Om''$ inequalities \eqref{3.1.20} and \eqref{3.1.22} hold, and total energy of the system $(\mu_{q'',f},\B_{q,e},\B_{q'',f})$ is at least $2^{-10}\de^{2^{d'}}$ larger than that of the system $(\mu_{q',f},\B_{q,e}, \B_{q',f}).$ Set $Z':=Z$, $\mu_{q',f}:=\mu_{q'',f}$ and $\B_{q',f}:=\B_{q'',f}.$
\\\\
By \eqref{3.1.19} the iteration process must stop in $O_{M_{d'},\de}(1)$ steps and the system obtained satisfies \eqref{3.1.20}-\eqref{3.1.22}.

\end{proof}

\medskip

\subsection{Hypergraph regularity lemmas.}

The shortcoming of Lemma 3.2 is that the complexity of the $\sigma$-algebras $\B_{q,f}$ might be very large with respect to the parameter $\de$, which measures the uniformity of the graphs $G_{q,e}$. This issue can be taken care of with an iteration process using Lemma 3.2 repeatedly, along the lines it was done in \cite{T1}. In the weighted settings we have to pass to a new system of weights and measures at each iteration and have to exploit the stability properties of well-defined extensions to show that the iteration process terminates.

\begin{lemma} (Preliminary regularity lemma). Let $1 \leq d' \leq d$ and $M_{d'}>0$ be a constant.
Let $\{\mu_{q,f} \}_{q \in Z, f \in \HH\ }$ be a well-defined, symmetric family of measures of complexity at most $K$, and $1 \leq d' \leq d$ and $\{\B_{q,e} \}_{q \in Z, e \in \HH_{d'}}\ $ be a family of $\sigma$-algebras on $V_e$ so that for all $q \in Z,\ e\in\HH_{d'}$
\eq\label{3.2.1}
compl\, (\B_{q,e})\leq M_{d'}.
\ee
Let $\eps>0$ and $F:\R_+ \rightarrow \R_+$ be a non-negative, increasing function, possibly depending on $\eps$ and $\Omega \subseteq Z$ be a set of measure $\psi(\Omega) \geq c_0>0.$
\\\\
If $N,W$ are sufficiently large with respect to the parameters $\eps,c_0,M_{d'},K$, and $F$, then there exists a well-defined, symmetric extension $\ \{\mu_{\aq,f} \}_{\aq \in \aZ}\ $ of complexity at most $O_{K,M_{d'},F,\,\eps}(1)$, and families of $\sigma$-algebras $\B_{\aq,f}\subs\B'_{\aq,f}$ defined for $\aq \in \aZ,\, f \in \HH_{d-1}$ and a set $\aO \subseteq \aZ$ such that the following hold.\\

\begin{enumerate}
\item We have that $\aO \subs \Omega \times \aV \subs \aZ = Z \times \aV$ where $\aV= \Z_N^k$ of dimension  $k=O_{M_{d'},F,\,\eps}(1)$. Moreover
\eq \label{3.2.2}
\aP(\aO) \geq c(c_0,F,M_{d'}, \eps)>0.
\ee

\item There is a constant $M_{d'-1}=O_{M_{d'},F, \eps}(1)$ such that for all $\aq \in\aZ$ and  $f \in \HH_{d'-1}$ we have
\eq \label{3.2.3}
\textit{compl}(\B_{\aq,f}) \leq M_{d'-1}.
\ee

\item For all $\aq=(q,p) \in \aO,\ e \in \HH_{d'}$ and  $G_{q,e} \in \B_{q,e}$, we have

\eq \label{3.2.4}
\big\| \E_{\mu_{\aq,e}}(\1_{G_{q,e}}| \bigvee_{f \in \partial e}\B'_{\aq,f})-
 \E_{\mu_{\aq,e}}(\1_{G_{q,e}}| \bigvee_{f \in \partial e}\B_{\aq,f}) \big\|_{L^2(\mu_{\aq},e)} \leq \eps
\ee
and
\eq \label{3.2.5}
\big\| \1_{G_{q,e}}-\E_{\mu_{\aq,e}}(\1_{G_{q,e}} | \bigvee_{f \in \partial e}\B'_{\aq,f})   \big\|_{\Box_{\mu_{\aq,e}}} \leq \frac{1}{F(M_{d'-1})}.
\ee\\

\end{enumerate}
\end{lemma}

\begin{proof}
Let $\{ \mu_{q',f} \}_{q' \in Z',\,f \in \HH}\ $ be a well-defined, symmetric extension of the initial system $\{\mu_{q,f} \}$ defined on a parameter space $Z'=Z\times V'$ of complexity at most $K'$. Also for $q'\in Z'$ and $f\in \HH_{d'-1}\ $ let $\{\B_{q',f} \}_{q' \in Z',f \in \HH_{d'-1}}\,$ be a family of $\sigma$-algebras of complexity at most $M_{d'-1}$.
Set $\B_{q',e}:=\B_{q,e}$ for $q'=(q,p)\in Z'$, $e\in\HH_{d'}$, and apply Lemma 3.2 to the system $(\mu_{q'f},\,\B_{q',e},\,\B_{q',f})$, with $\de=F(M_{d'-1})^{-1}$.\\\\
This generates a well-defined, symmetric extension $\{\mu_{\aq,f} \}_{\aq \in\aZ, f\in \HH}\ $ and a family of $\sigma$-algebras\\ $\{\B'_{\aq,f}\}_{\aq\in\aZ,f\in\HH_{d'-1}}\ $ and a set $\aO\subs\aZ$. Set $\B_{\aq,f}:=\B_{q',f}$ for $\aq=(q',p)\in \aZ$, $f\in\HH_{d'-1}$. The new system $(\mu_{\aq,f},\B_{\aq,f},\B'_{\aq,f})$ satisfies \eqref{3.2.2}-\eqref{3.2.3} and \eqref{3.2.5} as long as the parameters $K',\,M_{d'-1}$ are of magnitude $O_{K,M_{d'},F,\,\eps}(1)$. There are two possibilities.\\

\begin{itemize}
\item\emph{Case 1:} There exists a set $\aO_1 \subseteq \aO$ of measure $\aP(\aO_1) \geq \aP(\aO)/2$ such that \eqref{3.2.4} holds for all $\aq\in\aO_1$. In this case the conclusions of the lemma hold for the system $(\mu_{\aq,f},\B_{\aq,f},\B'_{\aq,f})$ and the set $\aO_1$.\\

\item\emph{Case 2:} There is a set $\aO_2 \subseteq \aO$ of measure $\aP(\aO_2) \geq \frac{1}{2} \aP(\aO)$ so that inequality \eqref{3.2.4} fails for all $\aq\in\aO_2$. Then, thanks to the stability condition \eqref{3.1.22} and the fact that $\B_{q',f}=\B_{\aq,f}\subs\B'_{\aq,f}$, we have for $\ \aq\in\aO_2$, $q'=\pi'(\aq)$, and $q=\pi(\aq)$ that
\end{itemize}
\begin{align}\label{3.2.6}
\quad&\quad\sum_{e, G_{q,e}}\big\|\E_{\mu_{\aq,e}}(\1_{G_{q,e}}| \bigvee_{f \in \partial e} \B'_{\aq,f})  \big\|^2_{L^2_{\mu_{\aq,e}}} -\sum_{e,G_{q,e}} \big\| \E_{\mu_{q',e}}(\1_{G_{q,e}}|\bigvee_{f \in \partial e}\B_{q',f})  \big\|^2_{L^2_{\mu_{q',e}}}\notag\\
&\geq \sum_{e,G_{q,e}}(\big\|\E_{\mu_{\aq,e}}(\1_{G_{q,e}}|\bigvee_{f \in \partial e} \B'_{\aq,f})\big\|^2_{L^2_{\mu_{\aq,e}}} - \big\|\E_{\mu_{\aq,e}}(\1_{G_{q,e}}| \bigvee_{f \in \partial e} \B_{q',f}) \big\|^2_{L^2_{\mu_{\aq,e}}})- o_{M_{d'},K',F}(1)\notag \\
&= \sum_{e,G_{q,e}}\big\| \E_{\mu_{\aq,e}}(\1_{G_{q,e}}| \bigvee_{f \in \partial e} \B'_{\aq,f}) -  \E_{\mu_{\aq,e}}(\1_{G_{q,e}}| \bigvee_{f \in \partial e} \B_{q',f}) \big\|^2_{L^2_{\mu_{\aq,e}}}\geq\eps^2-o_{M_{d'},K',F}(1),
\end{align}
where the summation is taken over all $e \in \HH_{d'}$ and $G_{q,e} \in \B_{q,e}$.\\\\
Thus, for sufficiently large $N,\,W$, we have for all $\aq=(q,p)\in \aO_2$ that the total energy of the system $\ (\mu_{\aq,f},\ \B_{\aq,e},\ \B'_{\aq,f})$ is at least $\frac{\eps^2}{2}$ larger than that of the system
$(\mu_{q',f},\ \B_{q,e},\ \B_{q',f})$. In this case, set $Z' := \aZ,\ \Omega':= \aO_3$, $\mu_{q',f} := \mu_{\aq,f}$,  and $\,\B_{q',f}:=\B'_{\aq,f}\,$ and repeat the above argument. Starting with the original system $\mu_{q,f},\ \B_{q,e}$ and $\si$-algebras, $\B_{q,f}=\{\emptyset,V_f\}$, the iteration process must stop in
at most $\eps^{-2}2^{2^{M_{d'}+1}}2^{d+1}=O_{M_{d'},\eps}(1)$ steps, generating a system $(\mu_{\aq,f},\B_{\aq,f},\B'_{\aq,f})$ which satisfies the conclusions of the lemma.
\end{proof}

\noindent
In order to obtain a counting and a removal lemma starting from a given measure system $\{\mu_{q,e}\}$ and $\sigma$-algebras $\{\B_{q,e} \}$ we need to regularize the elements of the $\sigma$-algebras $\B_{\aq,e}$ for all $e\in\HH$ with respect to the lower order $\sigma$-algebras $\bigvee_{f\in\partial e} \B_{q,f}$. This is done by applying Lemma 3.3 inductively, and provides the final form of the regularity lemma we need. Let us call a function $F: \R_+ \rightarrow \R_+$ a \emph{growth function} if it is continuous, increasing, and satisfies $F(x) \geq 1+x$ for $x \geq 0.$

\begin{thm}\label{Thm3.1}(Regularity lemma). Let $1 \leq d' \leq d$ and $M_{d'}>0$ be a constant.
Let $\{\mu_{q,f} \}_{q \in Z, f \in \HH\ }$ be a well-defined, symmetric family of measures of complexity at most $K$, and $1 \leq d' \leq d$ and $\{\B_{q,e} \}_{q \in Z, e \in \HH_{d'}}\ $ be a family of $\sigma$-algebras on $V_e$ so that for all $q \in Z,\ e\in\HH_{d'}$
\eq\label{3.2.7}
compl\, (\B_{q,e})\leq M_{d'}.
\ee
Let $F:\R_+ \rightarrow \R_+$ be a growth function, and $\Omega \subseteq Z$ be a set of measure $\psi(\Omega) \geq c_0>0.$
\\\\
If $N,W$ are sufficiently large with respect to the parameters $c_0,M_{d'},K$, and $F$, then there exists a well-defined, symmetric extension $\ \{\mu_{\aq,f} \}_{\aq \in \aZ}\ $ of complexity at most $O_{K,M_{d'},F}(1)$, and families of $\sigma$-algebras $\B_{\aq,f}\subs\B'_{\aq,f}$ defined for $\aq \in \aZ,\, f \in \HH_{d-1}$ and a set $\aO \subseteq \aZ$ such that the following hold.\\

\begin{enumerate}
\item We have that $\aO \subs \Omega \times \aV \subs \aZ = Z \times \aV$ where $\aV= \Z_N^k$ of dimension  $k=O_{M_{d'},F}(1)$. Moreover

\eq \label{3.2.8}
\aP(\aO) \geq c(c_0,F,M_{d'})>0.
\ee

\item
There exist numbers
\eq \label{3.2.9}
M_{d'} < F(M_{d'}) \leq M_{d'-1} < F(M_{d'-1}) \leq \cdots \leq M_1 < F(M_1) \leq M_0 =O_{M_{d'},F}(1)
\ee
such that for all $1 \leq j <d',\ f \in \HH_j,$ and $\aq \in \aZ$,

\eq \label{3.2.10}
   \text{compl}(\B'_{\aq,f}) \leq M_j.
\ee
\\
\item
For all $\ 1 \leq j \leq d',\ e\in \HH_j,\ \aq=(q,p)\in \aO,\ $ and  $G_{\aq,e} \in \B_{\aq,e}$ (with  $\B_{\aq,e}:=\B_{q,e}$, if $j=d'$), one has

\eq \label{3.2.11}
\big \| \E_{\mu_{\aq,e}}(\1_{G_{\aq,e}}| \bigvee_{f \in \partial e}\B'_{\aq,f})-  \E_{\mu_{\aq,e}}(\1_{G_{\aq,e}}| \bigvee_{f \in \partial e}\B_{\aq,f})      \big\|_{L^2(\mu_{\aq,e})} \leq \frac{1}{F(M_j)}
\ee
and
\eq \label{3.2.12}
\big\| \1_{G_{\aq,e}}-\E_{\mu_{\aq,e}}(\1_{G_{\aq,e}}| \bigvee_{f \in \partial e}\B'_{\aq,f})   \big\|_{\Box_{\nu_{\aq.e}}} \leq \frac{1}{F(M_1)}.
\ee\\
\end{enumerate}
\end{thm}

\begin{proof}
We proceed by an induction on $d'$. If $d'=1$ the statement follows from Lemma 3.3 with $\eps=\frac{1}{F(M_1)}\,$, so  assume that $d'\geq 2$ and the theorem holds for $d'-1$. Apply Lemma 3.3 with a growth function $F^{*} \geq F$ (to be specified later) and with $\eps =\frac{1}{2F^{*}(M_{d'})}$. This gives a well-defined, symmetric
extension $\{\mu_{q',f} \}$ and a family of $\sigma$-algebras $\B_{q',f} \subs \B'_{q',f}$ defined on a parameter space $Z'=Z\times V$, such that

\eq \label{3.2.13}
\big\| \E_{\mu_{q',e}}(\1_{G_{q',e}}| \bigvee_{f \in \partial e}\B'_{q',f})-\E_{\mu_{q',e}}(\1_{G_{q',e}}| \bigvee_{f \in \partial e}\B_{q',f})  \big\|_{L^2(\mu_{q',e})} \leq \frac{1}{2F^*(M_{d'})}
\ee
and
\eq \label{3.2.14}
\big\| \1_{G_{q',e}}-\E_{\mu_{q',e}}(\1_{G_{q',e}}| \bigvee_{f \in \partial e}\B'_{q',f})    \big\|_{\Box_{\nu_{q',e}}} \leq \frac{1}{F^*(M_{d'-1})},
\ee\\
hold for all $q'=(q,p) \in \Omega' , e \in \HH_{d'},$ and  $G_{q',e} \in \B_{q',e}=\B_{q,e}\,$, where $\Omega' \subseteq \Omega \times V \subseteq Z'$ is a set of measure $\psi'(\Omega') \geq c(c_0,F,M_{d'})>0$.
\\\\
Applying the induction hypothesis to the system $\{\mu_{q',f} \}_{q' \in Z'},\{\B_{q',f} \}_{q' \in Z',f \in \HH_{d'-1}}$, the growth function $F$, and the set $\Omega'$, one obtains an extension $\{\mu_{\aq,f}\}_{\aq \in \aZ, f \in \HH}$ and families of $\sigma$-algebras $\{\B_{\aq,f} \subs \B'_{\aq,f}\}_{\aq \in \aZ,\ f \in \HH_j}\ $ such that \eqref{3.2.10}--\eqref{3.2.12} hold for $j < d'-1$, with constants

\eq \label{3.2.15}
M_{d'-1} < F(M_{d'-1}) \leq \cdots \leq M_1 < F(M_1) =O_{M_{d'-1},F}(1).
\ee\\
For $\ \aq=(q',p)\in \aZ\,$, $f\in\HH_{d'-1}\,$ set $\,\B_{\aq,f} := \B_{q',f},$ and $\B'_{\aq,f}:=\B'_{q',f}$. We show that inequalities \eqref{3.2.11} and \eqref{3.2.12} hold for $j=d'$. Indeed, by the stability property \eqref{2.2.12}, one has

\begin{align}
&\quad\,\big\| \E_{\mu_{\aq,e}}(\1_{G_{q',e}}| \bigvee_{f \in \partial e}\B'_{q',f})-\E_{\mu_{\aq,e}}(\1_{G_{q',e}}| \bigvee_{f \in \partial e}\B_{q',f})  \big\|_{L^2(\mu_{\aq,e})} \notag \\
&=\big\| \E_{\mu_{q',e}}(\1_{G_{q',e}}| \bigvee_{f \in \partial e}\B'_{q',f})-\E_{\mu_{q',e}}(\1_{G_{q',e}}| \bigvee_{f \in \partial e}\B_{q',f})  \big\|_{L^2(\mu_{q',e})} + o_{K,M_{d'}, F, F^*}(1) \notag \\
&\leq \frac{1}{2F^*(M_{d'})}+o_{K,M_{d'},F,F^*}(1),
\end{align}
for all $\aq=(q',p)\in \aO \backslash \aE_1,\ e \in \HH_{d'}$, and $G_{q',e} \in \B_{q',e}$. Here $\aE_1 \subseteq \aO$ is a set of measure $\aP(\aE_1)=o_{K,M_{d'},F,F^*}(1).$\\\\
Similarly using the stability properties \eqref{2.2.10} and \eqref{2.2.17} of the box norms, we have

\begin{align}
&\quad\,\big\| \1_{G_{q',e}}-\E_{\mu_{\aq,e}}(\1_{G_{q',e}}| \bigvee_{f \in \partial e}\B'_{q',f})    \big\|_{\Box_{\nu_{\aq,e}}} = \big\| \1_{G_{q',e}}-\E_{\mu_{q',e}}(\1_{G_{q',e}}| \bigvee_{f \in \partial e}\B'_{q',f})    \big\|_{\Box_{\nu_{\aq,e}}}+o_{K,M_{d'},F,F^*}(1) \notag \\
&=\big\| \1_{G_{q',e}}-\E_{\mu_{q',e}}(\1_{G_{q,e}}| \bigvee_{f \in \partial e}\aB_{q',f})    \big\|_{\Box_{\nu_{q',e}}} +o_{K,M_{d'},F,F^*}(1)
\leq \frac{1}{2F^*(M_{d'-1})}+o_{K,M_{d'},F,F^*}(1),  \label{3.2.17}
\end{align}
for all $\aq=(q',p) \in \aO \backslash \aE_2,\ e \in \HH_{d'}$ and $A_{q',e} \in \B_{q',e}=\B_{\aq,e}\ $, where $\aE_2 \subseteq \aO$ is a set of measure $\aP(\aE_2)=o_{K,M_{d'},F,F^*}(1)$.
\\\\
With $F(M_1)=O_{M_{d'-1},F}(1)$, we have that $F(M_1) \leq C(M_{d'-1},F) =: \frac{1}{2}F^*(M_{d'-1})$ for a sufficiently rapidly growing function $F^*$ depending only on $F$. Assuming $N,\ W$ are sufficiently large with respect to $M_{d'}$ and $K$, inequalities \eqref{3.2.11}, \eqref{3.2.12} for $j=d'$ and $\aq\in\aO\backslash (\aE_1\cup\aE_2)$ follow from \eqref{3.2.13} and \eqref{3.2.14}. The rest of the conclusions of the theorem are clear from the construction.
\end{proof}

\section{Counting and the Removal Lemmas.}
\subsection{The Removal Lemma} In this section we formulate a so-called Counting Lemma and show how it implies Theorem \ref{Thm1.4}. Our arguments will closely follow and are straightforward adaptations of those in \cite{T1} to the weighted settings; for the sake of completeness we will include the details.
\\\\
For $e\in \HH_d$ let $G_e\subs V_e$ be a hypergraph, and let $\B_e=\{A_e,A^C_e, \emptyset, V_e \}$ be the $\sigma$-algebra generated by it. Let $\{\nu_e \}_{e \in \HH}$ and $\{\mu_e\}_{e \in \HH}$ be the weights and measures associated to a well-defined, symmetric family forms $\LL=\{L_e^k:\ e\in\HH_d,\,0\leq k\leq d\}$.
Take $M_d>0$ and $F:\R_+ \rightarrow \R_+$ be a growth function to be determined later and apply Theorem \ref{Thm3.1} with $d'=d$ to obtain a well-defined, symmetric parametric extension $\{ \mu_{q,e}\}_{q \in Z, e \in \HH}\,$ together with $\sigma$-algebras $\B_{q,e}\subs \B'_{q,e}$ and a set $\Omega \subseteq Z$ such that \eqref{3.2.8}-\eqref{3.2.12} hold.\footnote{The family $\{\nu_e\}$ can be considered as a parametric family of weights in a trivial way, setting $Z=\Om=\{0\}$, and $\psi(0)=1$.} Note that the complexity of the system as well as the $\si$-algebras is $O_{M_d,F}(1)$. We consider the system of measures $\mu_{q,e}$ and the $\si$-algebras $\B_{q,e},\ \B'_{q,e}$ fixed for the rest of this section.
\\\\
It will be convenient to define all our $\si$-algebras on the same space $V_J$ and use only the measures\\ $\mu_q:=\mu_{q,J}=\prod_{f\in\HH} \nu_{q,f}$ instead of using a whole ensemble of measures $\{\mu_{q,e}\}_{e\in\HH}$. Thanks to the stability conditions \eqref{2.1.6}-\eqref{2.1.7} this can be done at essentially no cost. Indeed for any $e\in\HH$ there is an exceptional set $\e_e\subs\Om$ of measure $\psi(\e_e)=o_{M_d,F}(1)$, such that for any family of sets $G_{q,e}\subs V_e$ we have that
\eq\label{4.1.1}
\mu_{q}(\pi_e^{-1}(G_{q,e}))=\mu_{q,e}(G_{q,e})+o_{M_d,F}(1),
\ee
uniformly for $q\in \Om\backslash\e_e$. Let $\e=\bigcup_{e\in\HH}\e_e$, $\Om':=\Om\backslash\e$ define the $\si$-algebras $\aB_{q,e}:=\pi_e^{-1}(\B_{q,e})$ and $\aB'_{q,e}:=\pi_e^{-1}(\B'_{q,e})$ on $V_J$. Note that  $\aB_{q,e}=\aB_{e}$ for $e\in\HH_d$ as the initial $\si$-algebras $\B_e$ are not altered in Theorem \ref{Thm3.1}. Let $\aB_q:=\bigvee_{e\in \HH} \aB_{q,e}$ be the $\si$-algebra generated by the algebras $\aB_{q,e}$, and define similarly the $\si$-algebra $\aB'_q$. The atoms of $\aB_q$ are of the form $A_q=\bigcap_{f\in\HH} A_{q,e}$ where $A_{q,e}$ is an atom of $\aB_{q,e}$. In particular if $E_e\in\aB_e$ then $\bigcap_{e\in\HH_d} E_e$ is the union of the atoms of $\aB_q$.
\\\\
The so-called counting lemma \cite{Go1,Ro,T1}, gives an approximate formula for the measure of ``most" atoms $A_{q}$ and as consequence it shows that their measure is bounded below by a positive constant depending only on the initial data $F$ and $M_d$. If, as in Theorem \ref{Thm1.4}, one assumes that the measure of $\bigcap_{e\in\HH_d} E_e$ is sufficiently small then it cannot contain most of the atoms thus removing the exceptional atoms from the sets $E_e$, the intersection of the remaining sets becomes empty, leading to a proof of Theorem \ref{Thm1.4}.
\\\\
To make this heuristic precise let us start by defining the \emph{relative density}
$\de_q(A_q|B_q):=\mu_{q}(A\cap B)/\mu_{q}(B)$ for $A,B\in\aB_q$,
with the convention that $\de_q(A_q|B_q):=1$ if $\mu_{q}(B)=0$.

\begin{definition} \label{def 4.1}
Let $A_{q}=\cap_{e \in \HH}A_{q,e}$ be an atom of $\B_{q}$. We say that the atom $A_{q}$ is \emph{regular} if the following hold.
\begin{enumerate}

\item For all atoms $A_{q,e}$
\eq \label{4.1.2}
\delta_{q}(A_{q,e}\big|\bigcap_{f\in\partial e} A_{q,f}) \geq \frac{1}{\log F(M_j)}.
\ee

\item Moreover
\eq \label{4.1.3}
\int_{V_J}\big| \E_{\mu_{q}}(\1_{A_{q,e}}|\bigvee_{f \in \partial e}\aB'_{q,e})- \E_{\mu_{q}}(\1_{A_{q,e}}|\bigvee_{f \in \partial e}\aB_{q,e})\big|^2 \prod_{f \subsetneq e}\1_{A_{q,f}}\, d\mu_{q}  \leq \frac{1}{F(M_j)} \int_{V_J} \prod_{f \subsetneq e}\1_{A_{q,f}}\, d \mu_{q}.
\ee
\end{enumerate}
\end{definition}

\noindent This roughly means that all atoms $A_{q,e}$ are both somewhat large and regular on the intersection of the lower order atoms $A_{q,f}$, ($f\in\partial e$).
\\

\begin{prop}\label{prop4.1} (Counting lemma). There is a set $\e \subseteq \Omega$ of measure $\psi(\e)=o_{N,W\to\infty;M_d,F}(1)$ such that if $q \in \Omega \backslash \e$ and if $\ A_{q}=\bigcap_{e \in \HH}A_{q,e} \in \bigvee_{e \in \HH}\aB_{q,e}$ is a regular atom, then

\eq \label{4.1.4}
\mu_{q}(A_{q})= (1+o_{M_d \rightarrow \infty}(1)) \prod_{e \in \HH} \delta_{q}(A_{q,e})+O_{M_1}\left(\frac{1}{F(M_1)}\right) + o_{N,W\to\infty;M_d,F}(1).
\ee
\end{prop}

\medskip

\noindent Next, following \cite{T1}, we show that the total measure of irregular atoms is small. For any atom $A_{q,e} \in \aB_{q,e}$, let $B_{q,e,A_{q,e}}$ be the union of all sets of the form $\bigcap_{f \subsetneq e} A_{q,f}$ for which \eqref{4.1.2} or \eqref{4.1.3} fails. Note that if an atom $A_q=\cap_{e\in\HH} A_{q,e}$ is irregular
then $A_q\subs A_{q,e}\cap B_{q,e,A_{q,e}}$ for some $e\in\HH$. We claim that
\eq \label{4.1.5}
\mu_{q}(A_{q,e} \cap B_{q,e,A_{q,e}}) \lesssim \frac{1}{\log F(M_j)}
\ee
for $q \notin \e_1$, where $\e_1 \subseteq \Omega$ is a set of measure $\psi(\e_1)=o_{M_d,F}(1)$. To see this, estimate first the contribution of those sets $\bigcap_{f \subsetneq e}A_{q,f}$ to the left side of \eqref{4.1.5} for which \eqref{4.1.2} fails. This quantity is bounded by

\begin{align*}
\sum_{\{A_{q,f}\}_{f \in\partial e},\ \ \eqref{4.1.2} \hspace{1 mm} \text{fails}}\,\mu_{q}(\A_{q,e} \cap \bigcap_{f\in \partial e} A_{q,f})
&\leq\ \frac{1}{\log F(M_j)}\ \sum_{\{A_{q,f}\}_{f \in\partial e}} \ \mu_q(\bigcap_{f \in\partial e}\,A_{q,f}) \notag \\
&\leq\ \frac{1}{\log F(M_j)}\ \mu_{q}(V_J)\ \lesssim\ \frac{1}{\log F(M_j)},
\end{align*}

\noindent as the summation is taken over the disjoint atoms of the $\si$-algebra $\bigvee_{f\in\partial e} \aB_{q,e}$.
\\\\
Similarly, one estimates the total contribution of the disjoint atoms $\bigcap_{f \subsetneq e}A_{q,f}$ for which \eqref{4.1.3} fails as follows.

\begin{align*}
&\sum_{\{A_{q,f}\}_{f \subs e},\ \eqref{4.1.3} \hspace{1 mm} \text{fails}} \mu_{q}(\bigcap_{f \subsetneq e}\,A_{q,f})\ \leq\notag\\
&\leq F(M_j)\ \int_{V_e}|\E_{\mu_{q,e}}(\1_{A_{q,e}}| \bigvee_{f \in \partial e}\aB'_{q,f})-  \E_{\mu_{q,e}}(\1_{A_{q,e}}| \bigvee_{f \in \partial e}\aB_{q,f})|^2\ d\mu_{q}\notag\\
&\leq F(M_j)\,\frac{1}{F(M_j)^2}\ =\ \frac{1}{F(M_j)}.
\end{align*}

\noindent Since the sets $A_{q,e} \cap B_{q,e,A_{q,e}}$ contain all irregular atoms, and for given $e \in \HH_{j}$ the number of all atoms of the $\si$-algebra $\aB_{q,e}$ is at most $2^{2^{M_j}}$, one estimates the total measure of all irregular atoms as

\eq \label{4.1.6}
\sum_{j=1}^d \sum_{e \in \HH_j} \sum_{A_{q,e \in \B_{q,e}}} \mu_{q}(A_{q,e} \cap \B_{q,e,A_{q,e}})
\leq \sum_{j=1}^d \binom{d}{j}2^{2^{M_j}}\frac{1}{\log F(M_j)} \leq \frac{1}{\sqrt{\log F(M_d)}} \leq 2^{-2^{M_d}}
\ee
\\
if, say $F(M) \geq 2^{2^{2^{M_d+d}}}$. This shows, choosing $M_d$ sufficiently large, that most atoms are regular.
\\\\
Another fact we need is that the measure of regular atoms is not too small. Indeed by \eqref{4.1.2}, \eqref{4.1.4}, we have that for $q \in \Omega$ and a regular atom $A_{q}=\cap_{e \in \HH} A_{q,e},$

\eq \label{4.1.7}
\mu_{q}(A_{q})\ \geq\ \prod_{j \leq d}\prod_{e \in \HH_j}\frac{1}{F(M_j)^{1/10}}-O_{d,M_1}\left(\frac{1}{F(M_1)}\right) +o_{M_d,F}(1)\  \geq\ \frac{1}{F(M_1)}>0,
\ee
as long as $F$ is sufficiently rapid growing and $M_d$ is sufficiently large with respect to $d.$ It is clear from \eqref{3.2.9} that $F(M_1)\leq F^*(M_d)$ for a function $F^\ast$ depending only on $F$ and $M_d$.
\\\\
After these preparations, assuming the validity of Proposition \ref{prop4.1}, it is easy to obtain the

\begin{proof}[Proof of Theorem \ref{Thm1.4}]
Let $\de>0$, $E_e\in\A_e$ and $g_e:V_e\to [0,1]$ for $e\in\HH_d$ be given. Let $\e_1\subs \Om$ be a set of measure $\psi(\e_1)=o_{M_d,F}(1)$ so that \eqref{4.1.1}, \eqref{4.1.6} and \eqref{4.1.7} hold for $q\in\Om/\e_1$. Also by \eqref{2.2.4} conditions \eqref{1.4.11}-\eqref{1.4.12} hold for
\eq\label{4.1.8}
\tilde{\mu}_J:=\mu_{q,J}\ \ \ \textit{and}\ \ \tilde{\mu}_e:=\mu_{q,e}\ \ (e\in\HH_d),
\ee
for $q\notin \e_2$, for a set $\e_2\subs \Om$ be a set of measure $\psi(\e_2)=o_{M_d,F}(1)$.
\\\\
Now fix $q\notin \e_1\cup\e_2$ and define $\tilde{\mu}_J$ and $\tilde{\mu}_e$ for $e\in\HH_d$ as is \eqref{4.1.8}. We claim that this system of measures satisfy the conclusions of the theorem. By construction the system is symmetric so it remains to construct the sets $E'_e$ and show \eqref{1.4.13}-\eqref{1.4.14} hold. For given $e \in \HH$ define the sets

\eq\label{1.4.9}
E'_{q,e}=V_J\backslash\,(B_{q,e,A_{e}} \cup \bigcup_{f \subsetneq e,\,A_{q,f}} (A_{q,f} \cap B_{q,f,A_{q,f}})),
\ee
\\
where $A_{q,f}$ ranges over the atoms of $\B_{q,f}$. As we have $\B_{q,e}=\B_e$, which is generated by a single set $E_e$, if $\bigcap_{e\in\HH_d} E_e$ contains an atom $A_q=\bigcap_{f\in\HH} A_{q,f}$ then $A_{q,e}=E_e$ for $e\in\HH_d$. If such an atom would be regular then by \eqref{1.4.10} its measure would satisfy
\[\frac{1}{F^*(M_d)}\leq \tilde{\mu}_J\,(\bigcap_{e\in\HH_d} E_e\,)=  \mu_J\,(\bigcap_{e\in\HH_d} E_e\,)+o_{M_d,F}(1)\ <\ 2\de.\]
Choosing $M_d$ to be the largest positive integer so that $F^\ast (M_d)\leq (2\de)^{-1}$ we see that $\bigcap_{e\in\HH_d} E_e$ contains only irregular atoms. From \eqref{4.1.9} and \eqref{4.1.6} we have

\eq \label{4.1.10}
\tilde{\mu}_J\,(E_e\backslash E'_{q,e})\,=\,\tilde{\mu}_{J}\,(\bigcup_{f \subs_e,\,A_{q,f}} (A_{q,f} \cap B_{q,f,A_{q,f}}))\, \leq\, 2^{-2^{M_d}}.
\ee
Also, all irregular atoms $A_q=\bigcap_{f\in\HH} A_{q,f}\subs \bigcap_{e\in\HH_d} E_e$ are contained in one of the sets $E_e\backslash E'_{q,e,}$ thus
\[\bigcap_{e\in\HH_d}\, (E_e\cap E'_{q,e})\,=\,\emptyset.\]
Finally, choosing $\eps:=2^{-2^{M_d}}$, \eqref{1.4.14} holds by \eqref{4.1.10}. Moreover $\de\to 0$ implies $M_d\to\infty$ and hence $\eps\to 0$. This proves Theorem \ref{Thm1.4}.
\end{proof}

\medskip

\subsection{Proof of Proposition \ref{prop4.1}.}
The proof  proceeds by induction and  uses the Cauchy-Schwartz inequality, causing to double certain sets of variables. As a consequence,  we need a generalization of Proposition \ref{prop4.1}  which requires the following definition.

\begin{definition}[Weighted hypergraph bundles over \HH]
Let $K$ be a finite set together with a map $\pi:K \rightarrow J$, called the projection map of the bundle. Let $\G_K$ be the set of edges $g\subs K$ such that $\pi$ is injective on any $g$ and $\pi(g) \in \HH$.\\\\
For any $g \in \G_K$, write
\[
V_{g}:= V_{\pi(g)} = \prod_{k \in g}V_{\pi(k)},
\]
and define the weight functions $\anu_{q,g}:V_g\to\R_+$ by
\[
\anu_{q,g}(x_g) := \nu_{q,\pi(g)}(x_g),
\]
and the associated measures $\amu_{q,g}$ with density functions
\[
d\amu_{q,g}(x_g)=\prod_{g' \subseteq g}\anu_{q,g'}(x_{g'}).
\]
The total measure measure $\amu_{q,K}$ on $V_K$ is given by
\[
d\amu_{q,K}(x_g)=\prod_{g' \subseteq \G}\anu_{q,g'}(x_{g'}).
\]
\\\\
A hypergraph $\G\subs \G_K$ which is closed in the sense that $\partial g\in\G$ for every $g\in\G$, together with the spaces $V_g$ and the weight functions $\amu_{q,g}$ for $g\in\G$ is called a weighted hypergraph bundle over $\HH$. The quantity $d'=\sup_{g\in\G}|g|$ is called the order of $\G$.\\
\end{definition}

\noindent Note that the underlying linear forms defining the weight system $\{\nu_{q,g}\}_{q\in Z,g\in\G_K}$,
\[
\bL^s_g(q,x_g)=L^s_{\pi(g)}(q,x_g),
\]
are pairwise linearly independent and the system of weights
\[
\anu_{q,g} := \prod_{s\in s(\pi(g))} \aL^s_g(q,x_g)
\]
is well-defined.

\begin{lemma}[Generalized Counting Lemma]
Let $\G \subs \G_K$ be a closed hypergraph bundle over $\HH$ with the projection map $\pi:K \rightarrow J$, and $d':=\sup_{g \in \G}|g|$ be the order of $\G$. Then, for $F$ growing sufficiently rapidly  with respect to $d$ and $K$, there exists a set $\e \subseteq \Omega$ of measure $\psi(\e)=o_{\n;M_d,K,F}(1)$ such that for $q\in\Om\backslash\e$ we have

\eq \label{4.1.5}
\amu_{q,K}(\bigcap_{g \in \G}\aA_{q,\pi(g)})=(1+o_{M_d \rightarrow \infty,K}(1))\prod_{g \in \G}\de_{q,f}(A_{q,\pi(g)})+O_{K,M_1}(\frac{1}{F(M_1)})+
o_{\n,K,M_d}(1).
\ee
\end{lemma}

\noindent Note that Proposition 4.1 is the special case when $K=J$, $\pi$ is the identity map and $\G=\HH$.

\begin{proof}
We  use a double induction. First we induct on $d'$, the order of $\G$ (note that $d' \leq d$), and then, fixing $K$ and $\pi$, we induct  on the number of edges $r:=|\{g \in \G: |g|=d'\}|$.\\\\
To start, assume that $d'=r=1$, so that $\G=\{k\}$ and $j =\pi(k)\in J.$ The left hand side of \eqref{4.1.7} becomes
\begin{align*}
\amu_{q,K}(\aA_{q,j}) &= \int_{V_K} \1_{\aA_{q,j}}(x_K)\ d\amu_{q,K}(x_K)=\int_{V_k}\1_{A_{q,j}}(x_k)\ d\amu_{q,k}(x_k)+o_{\n;M_d,K,F}(1) \\
&= \int_{V_j}\1_{A_{q,j}}(x_j)\ d\amu_{q,j}(x_j)+o_{\n,K}(1)=\delta_{q,j}(A_{q,j})+o_{\n;M_d,K,F}(1)
\end{align*}
for $\ q \notin \e \subseteq Z$, where $\e$ is a set of measure $\psi(\e)=o_{\n;M_d,K,F}(1)$. The second and fourth equalities follow from Lemma 2.2 using the fact that the family of measures $\{\amu_{q,g} \}$ is well-defined. The third equality holds because  $V_K=V_j$ and
\[
d\amu_{q,K}(x_j)=\prod_{s \in s(j)}\anu^s_{q,{k}}(x_j)=\prod_{s \in s(j)}\nu_{q,{j}}^s(x_j)=d\mu_{q,j}(x_j).
\]
\\
Let $\{A_{q,e} \}_{e \in \HH}$ be a regular collection of atoms for $q \in \Omega$, and define the functions $b_{q,e},c_{q,e}:V_e \rightarrow \R$ for $e \in \HH$ by

\eq \label{4.1.6}
b_{q,e} := \E_{\mu_{q,e}}(\1_{A_{q,e}}| \bigvee_{f \in \partial e} \B'_{q,f})-\E_{\mu_{q,e}}(\1_{A_{q,e}}| \bigvee_{f \in \partial e} \B_{q,f})
\ee

\eq \label{4.1.7}
c_{q,e}:= \1_{A_{q,e}} - \E_{\mu_{q,e}}(\1_{A_{q,e}}| \bigvee_{f \in \partial e} \B'_{q,f}).
\ee
\\
Note that if $x_e \pi_e (\in \bigcap_{f \in \partial e}\aA_{q,f}) \cap V_e$ then

\eq \label{4.1.8}
\de_{q,e} = \E_{\mu_{q,e}}(\1_{A_e}| \bigvee_{f \in \partial e}\B_{q,f}),
\ee\\
and thus one has the decomposition
\eq \label{4.1.9}
\1_{A_{q,e}}(x_e)= \delta_{q,e}+b_{q,e}(x_e)+c_{q,e}(x_e)
\ee
on this set. Let $g_0 \in \G$ such that $|g_0|=d'$ and use \eqref{4.1.11} to write
\[
\prod_{g \in \G}\1_{A_{q,\pi(g)}}(x_g)=(\delta_{q,\pi(g_0)}+b_{q,\pi(g_0)}(x_{g_0})+c_{q,\pi(g_0)}(x_{g_0}))\prod_{g \in \G \backslash \{ g_0\}}\1_{A_{q,\pi(g)}}(x_g),
\]\\
and consider the contribution of the terms separately

\begin{align}
\amu_{q}({\bigcap_{g \in \G} \aA_{q,\pi(g)}})&= \int_{V_K}\prod_{g \in \G}\1_{\aA_{q,\pi(g)}}(x)d\amu_{q}(x) =\int_{V_K} \prod_{g \in \G} \1_{A_{q,\pi(g)}}(x_g) d\amu_{q}(x) \notag \\
&=\int_{V_K} (\de_{q,\pi(g_0)}+b_{q,\pi(g_0)}(x_{g_0})+c_{q,\pi(g_0)}(x_{g_0}))\prod_{g \in \G \backslash \{g_0\}}\1_{A_{q,\pi(g)}}(x_g)d\amu_{q}(x) \notag \\
&=M_{q}+E^1_{q}+E^2_{q}. \label{4.1.10}
\end{align}

\noindent Consider the first main term $M_{q}$. By the second induction hypothesis we have

\begin{align*}
M_{q}&=\de_{q,\pi(g_0)}\ \amu_{q}(\bigcap_{g \in \G \backslash \{g_0 \}}\aA_{q,\pi(g)})=\de_{q,\pi(g_0)}\,(1+o_{M_d \rightarrow \infty}(1)) \prod_{g \in \G}\delta_{q,\pi(g)}\\
&+\ \ O_{K,M_1}(\frac{1}{F(M_1)})+\ o_{\n,K,M_d}(1),
\end{align*}\\
and then $M_{q}$ agrees with the right side of \eqref{4.1.4}. We continue to estimate the second error term by

\begin{align}
E_q^2 &= \int_{V_K} c_{q,\pi(g_0)}(x_{g_0}) \prod_{g \in \GG}\1_{A_{q,\pi(g)}}(x_g) d \amu_{q}(x)   
=\E_{x \in V_K}(c_{q,\pi(g_0)}\anu_{q,g_0})(x_{g_0})\prod_{g \in \GG}\1_{A_{q,\pi(g)}}\anu_{q,g}(x_g)  \notag \\
&=\E_{x \in V_K} \prod_{|g|=d',g \in \G}f_{q,g}(x_g) \prod_{g' \in \G,|g'|<d'}\anu_{q,g'}(x_{g'}) \label{4.1.11},
\end{align}
\\
where $f_{q,g_0}=c_{q,\pi(g_0)}\anu_{q,g_0}$ and $f_{q,g}=h_{q,g}\anu_{q,g}$, for $g \in \G, g \neq g_0$ and $|g|=d'$ for a function $h_{q,g}$ of magnitude at most $1$. Thus we have $|f_{q,g}| \leq \anu_{q,g}$ for all $g \in \G,|g|=d'.$ Note that we are essentially in the situation of Proposition \ref{neumann}, the generalized von Neumann inequality. Indeed applying the Cauchy-Schwartz inequality $d'$ times successively in the variables $x_j,\ j \in g_0$, as in the appendix,
to clear all functions $f_{q,g}(x_g)$ which does not depend on at least one of these variables, we obtain

\eq \label{4.1.12}
|E^2_{q}|^{2^{d'}} \lesssim \big\| c_{q,\pi(g_0)} \big\|^{2^{d'}}_{\Box_{\anu_{q,g_0}}}+\E_{x_{g_0},y_{g_0}}|\WW_{q}(x_{g_0},y_{g_0})-1|\prod_{h \subseteq g_0}\prod_{\omega_h \in \{0,1\}^h} \anu_{q,h}(\omega_h(x_h,y_h)),
\ee
where setting $K':=K \backslash g_0$
\eq \label{4.1.13}
\WW_{q}(x_{g_0},y_{g_0})=\E_{x \in V_{K'}}\prod_{g \in \G,g\nsubseteq g_0}\ \prod_{\omega_g \in \{0,1\}^{g\cap g_0}} \anu_{q,g}(\omega_{g}(x_{g\cap g_0},y_{g\cap g_0}),x_{g\backslash g_0}).
\ee
\\
Note that the first term on the right hand side of \eqref{4.1.12} is $O(F(M_1)^{2^{-d'}})$ by \eqref{3.1.12} and \eqref{4.1.7}. To estimate the second term we apply the Cauchy-Schwartz inequality one more time to see that it is $o_{\n;M_d,K,F}(1)$ for $q \notin \e_1$, $\e_1$ being a set of measure $o_{\n,M_d,K,F}(1)$ using the fact that the underlying linear forms are pairwise linearly independent in the variables $(q,x_{g_0},y_{g_0},x_{K'})$.
\\\\
Finally we estimate the error term $E_{q}^1$ defined as
\[
E_{q}^1 = \int_{V_K} b_{q,\pi(g_0)}(x_{g_0})\prod_{g \in \GG}\1_{A_{q,\pi(g)}}(x_g)d\amu_{q}(x).
\]
Taking absolute values and discarding all factors $\1_{A_{q,\pi(g)}}(x_g)$ for $|g|=d'$, $g\neq g_0$, one estimates

\begin{align*}
|E_{q}^1| \leq &\int_{V_{g_0}}(|b_{q,\pi(g_0)}(x_{g_0})|\prod_{g \subsetneq g_0}\1_{A_{q,\pi(g)}}(x_g)) \\
&\times (\E_{x_{ K'}}\prod_{g \in \G', |g|<d'}\1_{A_{q,\pi(g)}}\anu_{q,g}(x_g)\prod_{h \in \G',|h|=d'}\anu_{q,h}(x_h))\ d\amu_{q,g_0}(x_{g_0}),
\end{align*}
where $\G'=\{g\in\G:\ g\nsubseteqq g_0\}$.
Since $\anu_{q,g_0}(V_{g_0})=1+o_{\n;M_d,K,F}(1)$ outside a set $\e_2 \subseteq \Omega$ of measure  $\psi(\e_2)=o_{\n;M_d,K,F}(1)$, we may write
\[
E_{q}^1= \int_{V_{g_0}}A(x_{g_0})B(x_{g_0})\ d \amu_{q,g_0}(x_{g_0}),
\]
and apply Cauchy-Schwartz's inequality to get

\eq \label{4.1.14}
|E_{q}^1|^2 \lesssim \left(\int_{V_{g_0}}A(x_{g_0} )^2\ d\amu_{q,g_0}(x_{g_0})\right) \left(\int_{V_{g_0}}B(x_{g_0})^2\ d\amu_{q,g_0}(x_{g_0}) \right).
\ee
\\
The first factor on the left side of \eqref{4.1.14} is estimated by

\eq \label{4.1.15}
\E_{x_{g_0} \in V_{g_0}} b_{q,\pi(g_0)}(x_{g_0})^2 \prod_{g \subsetneq g_0}\1_{A_{q,\pi(g)}}(x_g)\prod_{g \subseteq g_0}\nu_{q,\pi(g)}(x_g).
\ee
\\
Let $f_0=\pi(g_0)$. Since $\pi:g_0 \rightarrow f_0$ is injective and $V_{g_0}=V_{f_0}$,  we may write the expression in \eqref{4.1.15}, by re-indexing the variables $x_g$ to $x_f$ for $g \subs g_0,\ f= \pi(g)$, as

\eq \label{4.1.16}
\int_{V_{f_0}}b^2_{q,f_0}(x_{f_0}) \prod_{f \subsetneq f_0} \1_{A_{q,f}}(x_f)\ d\mu_{q,f_0}(x_{f_0}) \lesssim \frac{1}{F(M_{d'})}\int_{V_{f_0}} \prod_{f \subsetneq f_0}\1_{A_{q,f}}(x_f) d\mu_{q,f_0}(x_{f_0}),
\ee
where the inequality follows from by assumption \eqref{4.1.2} on regular atoms. By the induction hypothesis we further estimate the right side \eqref{4.1.16} as

\eq \label{4.1.17}
\frac{1}{F(M_{d'})}\ (1+o_{M_{d} \rightarrow \infty}(1)) \prod_{f \subsetneq f_0}\de_{q,f}(A_{q,f})+O_{M_d}(\frac{1}{F(M_1)})+o_{\n;M_d,K,F}(1).
\ee
\\
The second factor in \eqref{4.1.14} may be expressed in terms of a hypergraph bundle $\tilde{K}$ over $K$, involving the construction in \cite{T1}. Let $\tilde{K}=K_0 \oplus_{g_0} K$, the set $K \times \{0,1\}$ with the elements $(k,0)$ and $(k,1)$ are identified for $k \in g_0.$ Let $\phi: \tilde{K} \rightarrow K$ be the natural projection, and $\pi \circ \phi: \tilde{K} \rightarrow J$ be the associated map down to $J$. Let $\G_0=\{g \in \G, g \subseteq g_0 \}$ and $\G'=\{g \in \G, g \not\subseteq g_0, |g|<d' \}$ and define the hypergraph bundle $\tilde{\G}$ on $\tilde{K}$ to consist of the edges $g \times \{0\}$ and $g \times \{1\}$ for $g \in \G_0 \cup \G'$, two edges being identified for $g \in \G_0$. Define the weights
\eq \label{4.1.18}
\tilde{\nu}_{q,g\times \{i\}}(x_{g \times \{i\}}) := \anu_{q,g}(x_{g \times\{i\}}),
\ee
for $q \in Z,\ g \in \G_K, i=0,1,$ that is for all edges $\tilde{g}\in\G_{\tilde{K}}$, and let $\tilde{\mu}_{q,g \times \{i\}}$ be the associated family of measures. Then we have for the second factor appearing in \eqref{4.1.14} (with $h_0 := K \backslash g_0$)

\begin{align}
&\int_{V_{g_0}} B(x_{g_0})^2 d\amu_{q,g_0}(x_{g_0}) \notag \\
&=\int_{V_{g_0}}\bigg[ \prod_{g \in g_0}\1_{A_{q,\pi(g)}(x_g)} \bigg]\bigg[\E_{x_{h_0} \in V_{K \backslash g_0}}\prod_{g \in \GG} \1_{A_{q,\pi(g)}}\anu_{q,g}(x_g) \prod_{h \not\subseteq g_0, |h'|=d'}\anu_{q,h}(x_h) \bigg]^2\,d\amu_{q,g_0}(x_{g_0})\notag \\
&=\int_{V_{\tilde{K}}}\prod_{\tilde{g} \in \tilde{\G}}\1_{A_{q,\pi \circ \phi (\tilde{g})}}(x_{\tilde{g}})\  d\tilde{\mu}_{q,\tilde{K}}(x_{\tilde{K}}). \label{4.1.19}
\end{align}
\\
Indeed, when expanding the square of inner sum in \eqref{4.1.17} we double all points in $K \backslash g_0$ thus we eventually sum over $x_{\tilde{K}} \in V_{\tilde{K}}$, also double all edges $g \in \tilde{G}$ to obtain the edges $g \times \{0\}, g \times \{1\}.$ As for the weights,
the procedure doubles all weights $\anu_{q,g}(x_g)$ for $g \not\subset g_0$, $g\in \G_K$ to obtain the weights $\anu_{q,g}(x_{g \times \{i\}})$ for $i=0,1$ while leaves the weights $\anu_{q,g}(x_g)$ for $g \subseteq g_0$ unchanged.
The order of $\tilde{g}$ is less than $d'$ thus by the first induction hypothesis, we have

\begin{align}
&\tilde{\mu}_{q,\tilde{K}}(\bigcap_{\tilde{g }\in \tilde{\G}}\aA_{q,\pi \circ \phi(\tilde{g})})
=(1+o_{M_d \rightarrow \infty}(1))\prod_{\tilde{g} \in \tilde{\G}}\de_{q,\pi \circ \phi(\tilde{g})}(A_{q,\pi \circ \phi(\tilde{g})})+O_{K,M_1}(\frac{1}{F(M_1)})+o_{\n;,M_d,K,F}(1) \notag \\
&=(1+o_{M_d \rightarrow \infty}(1))\prod_{g \in \G_0}\delta_{q,\pi(g)}(A_{q,\pi(g)})\prod_{g \in \G'}\delta^2_{q,\pi(g)}(A_{q,\pi(g)})    +O_{d,K,M_1}(\frac{1}{F(M_1)})+o_{\n:M_d,K<F}(1), \label{4.1.20}
\end{align}
\\
for $q \notin \e_{\tilde{K},\phi}$ where $\e_{\tilde{K},\phi} \subseteq \Omega$ is a set of measure $\psi(\e_{\tilde{K},\phi})=o_{\n;M_d,K,F}(1)$.
Note that there are only $O_{K}(1)$ choices for choosing the set $\tilde{K}$ and the projection map $\phi: \tilde{K} \rightarrow K$ thus taking the union of all possible exceptional sets $\e_{\tilde{K},\phi}$ we have that \eqref{4.1.20} holds for $q \notin \e'_K$ if measure $\psi(\e'_K)=o_{\n;M_d,K,F}(1)$. Combining the bounds \eqref{4.1.17} and \eqref{4.1.20} we obtain the error estimate $E_q^1=o_{\n;M_d,K,F}(1)$ outside a set $\e'_K$ of measure $o_{\n;M_d,K,F}(1)$. This closes the induction and the proposition follows.
\end{proof}

\numberwithin{equation}{section}
\section{Proof of The Main Results} \noindent In this section we finish the proof of our main result Theorem \ref{MainThm2}. Since we have already shown the validity of Theorem \ref{Thm1.4} and hence by the argument in the introduction that of Theorem \ref{Thm1.3} it remains to show that counting affine copies of $\Delta$ in a set $A\subs \Z_N^d$ with weights $w$ translates to counting copies in $A\subs \PP^d$ of
positive relative density. This is standard, we include the details for the sake of completeness, following the arguments given in \cite{CoMa}.
\\\\
First, let us identify $[1,N]^d$ with $\Z_N^d$, let us recall that constellations in $\Z_N^d$ defined by the simplex $\De$ which are contained in a box $B\subs [1,N]^d$ of size $\eps N$, are in fact genuine constellations contained in $B$. Note that we can assume that the simplex $\De$ is \emph{primitive} in the sense that $t\De\nsubseteq\Z^d$ for any $0<t<1$, as any simplex is a dilate of a primitive one. To any simplex $\De\subs\Z^d$ there exists a constant $\tau(\De)>0$ depending only on $\De$ such that the following holds.

\begin{lemma}\cite{CoMa} Let $E\subs\Z^d$ be a primitive simplex. Then there is constant $0<\eps<\tau(\De)$ so that the following holds.
\\\\Let $N$ be sufficiently large, and let $B=I^d$ be a box of size $\eps N$ contained in $[1,N]^d\simeq\Z_N^d$.
If there exist $x\in\Z_N^d$ and $1\leq t<N$ such that $x\in B$ and $x+tE\subs B$ as a subset on $\Z_N^d$, then either  $x+t\De\subs B$ or $x+(t-N)\De\subs B$, also as a subset of $\Z^d$.
\end{lemma}

\emph{Proof} [Theorem \ref{Thm1.3} implies Theorem \ref{MainThm2}]\\

Fix $\al>0$ and let $\De=\{0,v_1,\ldots,v_d\}\subs\Z^d$ be non-degenerate simplex. Let $W=W(\al)$ be a sufficiently large integer and let $A\subs\PP_N^d$ be a set with $|A| \geq\, \al\, |\PP_N|^d$, where $N$ is assumed to be sufficiently large with respect to both $W$ and $\al$.
By the pigeonhole principle choose $b=(b_j)_{1\leq j\leq d}$ so that $b_j$ is relative prime to $W$ for each $j$, and
\eq\label{5.1}
|A\cap( (W\Z)^d+b)|\geq \al\ \frac{N^d}{(log\,N)^d\ \phi(W)^d},
\ee
where $\phi$ is the Euler totient function. Set $N_1:=N/W$ and $\ A_1:=\{n\in [1,N_1]^d;\ Wn+b\in A\}\,$. Choose $\eps_2>0$ so that $\eps_2<\tau(\De)$. By the Prime Number Theorem there is a prime $N'$ so that $\eps_2 N'=N_1(1+o_{N_1\to\infty}(1))$, thus we have
\eq\label{5.2}
|A_1\cap [1,\eps_2 N']^d|\geq \frac{\al\,\eps_2^d}{2}\ \frac{(N')^d\, W^d}{(log\,N')^d\ \phi(W)^d}.
\ee
By Dirichlet's theorem on primes in arithmetic progressions the number of $\ n\in [1,N']^d\backslash [\eps_1 N',N']^d\,$ for which $\ Wn+b\in\PP^d\,$ is of $\ O_d(\eps_1 \frac{N'^d W^d}{(log\,N')^d\ \phi(W)^d})\,$, thus \eqref{5.2} holds for the set $\ A':= A_1\cap [\eps_1 N',\eps_2 N']^d\,$ as well, if $\eps_1 \leq c_d\,\eps_2^d\al$ for a small enough constant $c_d>0$.
\\\\
Fix $b=(b_j)_{1\leq j\leq d}$, $0<\eps_1 <\eps_2$ as above, and let $\nu_j=\nu_{b_j}$ be the Green-Tao measures on $\Z_{N'}$. For $n=(n_j)_{1\leq j\leq d}\in A'$ one has that $Wn+b\in A\subs \PP^d$ hence by (1.3.1) $w(x)=c_d (\frac{\phi(W)}{W}\log\,N)^d$. Thus
\eq\label{5.3}
\EE_{x\in\Z_{N'}^d} \1_{A'}(x)\,w(x)\,\geq\,c(d,\De)\, \al,
\ee
for some constant $c(d,\De)>0$. Applying Theorem 1.3 for the set $A'\subs \Z_{N'}^d$ it follows that
\eq\label{5.4}
\EE_{(x,t)\in\Z_{N'}^{d+1}}\,\prod_{i=1}^d \1_{A'}(x+tv_i)\,w(x+t\De)\,\geq c(d,\De,\al)>0.
\ee
If $x+t\De\subs A'$ then by (1.3.1) and (1.3.3) we have that $w(x+t\De) = c_\De\,(\frac{\phi(W)}{W}\log\,N)^{l(\De)}$, thus the set $A'\subs \Z_{N'}^d$ must contain at least $\ c(d,\De,\al) (N')^{d+1} (\frac{\phi(W)}{W}\log\,N)^{-l(\De)}\ $ simplices of the form $\De'=x+t\De$. By Lemma 5.1 this remains true in $\Z^d$ as $A'\subs [\eps_1 N,\eps_2 N]^d$. Since $W\De'+b\subs WA'+b\subs A$ and $N'\approx N/\eps_2W$ it follows that the set $A$ contains at least
\[\ c(d,\De,\al) W^{l(\De)-d-1}\phi(W)^{-l(\De)}\, N^{d+1}(\log\,N)^{-l(\De)}= c'(d,\De,\al)\,N^{d+1}(\log\,N)^{-l(\De)}\]
simplices of the form $\De'=y+s\De$, as we have chosen a sufficiently large but fixed integer $W=W(\al)$. This proves Theorem 1.2. $\Box$


\appendix
\section{Basic properties of weighted box norms}
\noindent In this appendix we describe some basic facts about the weighted version of Gowers's box norms defined in (1.6.2) for functions $F:V_e\to \R$. We will assume $e=\{1,\ldots,d\}=:[d]$, and $V=V_\dd=\Z_N^d$ without loss of generality. To show that these are indeed norms (for $d\geq 2$) let us define a multilinear form referred to as the weighted Gowers's inner product. Let $F_\om:V_e\to\R$ for $\om\in \{0,1\}^e$, be a given family of functions and define

\begin{align*}
\left\langle F_{\omega},\omega \in \{0,1\}^d \right\rangle_{\Box_{\nu}} &:= \E_{x_\dd,y_\dd\in V}\prod_{\omega\in\{0,1\}^d} F_{\omega}(\omega(x_\dd,y_\dd))\prod_{|I|<d}\prod_{\omega_I\in\{0,1\}^I}
\nu_I(\omega_I(x_I,y_I)).
\end{align*}
So $\left\langle F_{\omega},\omega \in \{0,1\}^d \right\rangle_{\Box_{\nu}}=\left\|F\right\|_{\Box_{\nu}}^{2^d}$, if $F_\om =F$ for all $\om\in\{0,1\}^e$.

\begin{lemma}[Gowers-Cauchy-Schwartz's inequality] $|\left\langle F_{\omega}; \omega \in \{0,1\}^d \right\rangle| \leq  \displaystyle{ \prod_{\omega_{[d]}} } \left\|F_{\omega}\right\|_{\Box^d_{\nu}}.$
\end{lemma}

\begin{proof} We will use Cauchy-Schwartz inequality several times and the linear forms condition.
\begin{align*}
\left\langle F_{\omega} ; \omega \in \{0,1\}^d \right\rangle_{\Box_{\nu}^d}
&=\E_{x_{[2,d]},y_{[2,d]}}\bigg[ \bigg( \prod_{|I|<d,1 \notin I} \prod_{\omega_I} \nu_I(\omega_I(x_I,y_I))   \bigg)^{1/2} \\
&\times \bigg( \E_{x_1} \nu(x_1) \prod_{\omega_{[2,d]}} F_{\omega_{(0,[2,d])}}(x_1,\omega_{[2,d]}(x_{[2,d]},y_{[2,d]}))\prod_{|I|<d-1,1 \notin I}\nu_{\{1\} \cup I}(x_1,\omega_I(x_I,y_I)) \\
&\times \bigg( \prod_{|I|<d,1 \notin I} \prod_{\omega_I} \nu_I({\omega_I}(x_I,y_I))   \bigg)^{1/2} \\
&\times \bigg( \E_{y_1} \nu(y_1) \prod_{\omega_{[2,d]}} F_{\omega_{(1,[2,d])}}(y_1,{\omega_{[2,d]}}(x_{[2,d]},y_{[2,d]}))\prod_{|I|<d-1,1 \notin I}\nu_{\{1\} \cup I}(y_1,{\omega_I}(x_I,y_I))\bigg].
\end{align*}
Applying the Cauchy Schwartz inequality in the $x_1$ variable, one has
\[
|\left\langle F_{\omega} ; \omega \in \{0,1\}^d \right\rangle_{\Box_{\nu}^d} |^2 \leq A \cdot B
\]
here,
\begin{align*}
A &=\E_{x_{[2,d]},y_{[2,d]}} \bigg[\prod_{|I|<d,1 \notin I} \prod_{\omega_I} \nu_I({\omega_I}(x_I,y_I))   \ \\
&\times \bigg( \E_{x_1,y_1}\nu(x_1)\nu(y_1) \prod_{\omega_{[2,d]}} F_{\omega_{(0,[2,d])}}(x_1,{\omega_{[2,d]}}(x_{[2,d]},y_{[2,d]}))
F_{\omega_{(0,[2,d])}}(y_1,{\omega_{[2,d]}}(x_{[2,d]},y_{[2,d]}))  \\
&\times \prod_{|I|<d-1,1 \notin I}\nu_{\{1\} \cup I}(x_1,{\omega_I}(x_I,y_I))  \bigg] \ =\ \left\langle F^{(0)}_{\omega}({\omega}(x_{[d]},y_{[d]})) \right\rangle_{\Box^d_{\nu}},
\end{align*}
where
\begin{align*}
F^{(0)}_{(0,\omega_{[2,d]})}(x_1,{\omega_{[2,d]}}(x_{[2,d]},y_{[2,d]})) &= F^{(0)}_{(1,\omega_{[2,d]})}(y_1,{\omega_{[2,d]}}(x_{[2,d]},y_{[2,d]})) \\
&:=F(x_1,{\omega_{[2,d]}}(x_{[2,d]},y_{[2,d]}))
\end{align*}
for any $\omega_{[2,d]}$. Similarly,
\[
B =\left\langle F^{(1)}_{\omega}{\omega}(x_{[d]},y_{[d]})) \right\rangle_{\Box^d_{\nu}}
\]
where
\begin{align*}
F^{(1)}_{(0,\omega_{[2,d]})}(x_1,{\omega_{[2,d]}}(x_{[2,d]},y_{[2,d]})) &= F^{(1)}_{(1,\omega_{[2,d]})}(y_1,{\omega_{[2,d]}}(x_{[2,d]},y_{[2,d]})) \\
&:=F(y_1,{\omega_{[2,d]}}(x_{[2,d]},y_{[2,d]}))
\end{align*}
for any $\omega_{[2,d]}.$
In the same way, applying Cauchy-Schwartz's inequality in $x_2$ variable, we end up with
\[
|\left\langle  F_{\omega}; \omega \in \{0,1\}^d \right\rangle_{\Box^d_{\nu}}| \leq
\prod_{\omega_{[0,1]}}\left\langle  F^{\omega_{[0,1]}}_{\omega}; \omega \in \{0,1\}^d \right\rangle_{\Box^d_{\nu}}
\]
and continuing this way with $x_3,...,x_d$ variables, we end up with
\[
|\left\langle  F_{\omega}; \omega \in \{0,1\}^d \right\rangle_{\Box^d_{\nu}}| \leq
\prod_{\omega\in\{0,1\}^d}\left\langle F_\om,...,F_\om \right\rangle_{\Box^d_{\nu}} = \prod_{\omega\in\{0,1\}^d} \left\| F_{\omega} \right\|^{2^d}_{\Box^d_{\nu}}.
\]
\end{proof}

\begin{cor}
$\left\| \cdot \right\|_{{\Box^d_{\nu}}}$ is a semi-norm  for $d\geq 1$.
\end{cor}
\begin{proof}
By the Gowers-Cauchy-Schwartz inequality we have that $\|F\|_{\Box_\nu}\geq 0$,  moreover
\begin{align*}
\left\| F+G \right\|_{\Box_{\nu}^d}^{2^d} &= \left\langle  F+G,...,F+G \right\rangle_{\Box_{\nu}^d} \\
&=\sum_{\omega \in \{0,1\}^d}\left\langle h^{\omega_1},...,h^{\omega_d} \right\rangle_{\Box_{\nu}^d},\quad h^{\omega}=\begin{cases} F  &,\omega=0 \\ G &,\omega=1 \end{cases}\\
&\leq\sum_{\omega \in \{0,1\}^d} \left\| h^{\omega_1} \right\|_{\Box_{\nu}^d}...\left\| h^{\omega_d} \right\|_{\Box_{\nu}^d} =(\left\| F \right\|_{\Box_{\nu}^d}+\left\| G \right\|_{\Box_{\nu}^d})^{2^d}.
\end{align*}
Also it follows directly from the definition that $\left\|\lambda F\right\|_{\Box_{\nu}^d}^{2^d}=\lambda^{2^d}\left\| f \right\|_{\Box_{\nu}^d}^{2^d}$, hence $\left\| \lambda F \right\|_{\Box_{\nu}^d}=|\lambda|\left\|F \right\|_{\Box_{\nu}^d}.$
\end{proof}

\begin{proof}[Proof of Proposition \ref{neumann}] Let $\HH'=\{f\in\HH;\ |f|<d$, and write the left side of \eqref{neumannineq} as
\[\E=\EE_{x\in V_J}\prod_{e\in\HH_d} F_e(x_e)\prod_{f\in\HH'}\nu_f(x_f).\]
Fix $e_0=[d]$ and write $e_j:=[d+1]\backslash\{j\}$ for the rest of the faces. The idea is to apply the Cauchy-Schwartz inequality successively in the $x_1,x_2,\ldots,x_d$ variables to eliminate the functions $F_{e_1}\leq \nu_{e_1},\ldots,F_{e_d}\leq \nu_{e_d}$, using the linear forms condition at each step. Using $F_{e_1}\leq\nu_{e_1}$ we have
\[|E|\leq \EE_{x_2,\ldots,x_{d+1}} \nu_{e_1}(x_1)\prod_{1\notin f\in\HH'} \nu_f(x_f)\big|\EE_{x_1}\prod_{j\neq 2}F_{e_j}(x_j)\prod_{1\in f\in\HH'}\nu_f(x_f)\big|.
\]
By the linear forms condition $\EE_{x_2,\ldots,x_{d+1}} \nu_{e_1}(x_1)\prod_{1\notin f\in\HH'} \nu_f(x_f)=1+o_{\n}(1)$, thus by the Cauchy-Schwartz inequality
\eq\label{6.1}
E^2 \ls \EE_{x_2,\ldots,x_{d+1}} \nu_{e_1}(x_1)\prod_{1\notin f\in\HH'} \nu_f(x_f)\ \EE_{x_1,y_1} \prod_{j\neq 2}F_{e_j}(x_1,x_{e_j\backslash\{1\}})F_{e_j}(y_1,x_{e_j\backslash\{1\}})
\ee
\[\times \prod_{1\in f\in\HH'}\nu_f(y_1,x_{f\backslash\{1\}})\,\nu_f(x_1,x_{f\backslash\{1\}}).
\]
Note that, what happened is that we have replaced the function $F_{e_1}$ by the measure $\nu_{e_1}$, doubled the variable $x_1$ to the pair of variables $(x_1,y_1)$ and also doubled each factor of the form $G_e(x_e)$ (which is either $F_e(x_e)$ or $\nu_e(x_e)$, for $e\in\HH$) depending on the $x_1$ variable. To keep track of these changes as we continue with the rest of that variables, let us introduce some notations. Let $g\subs [d]$ and for a function $G_e(x_e)$ define
\eq
\label{6.2}G_e^*(x_{e\cap g},y_{e\cap g},x_{e\backslash g}) := \prod_{\om_e\in \{0,1\}^{e\cap g}} G_e(\om_e(x_{e\cap g},y_{e\cap g}),x_{e\backslash g}).
\ee
We claim that after applying the Cauchy-Schwartz inequality in the $x_1,\ldots,x_i$ variables we have with $g=[i]$
\begin{align}\label{6.3}
E^{2^i} &\ls \EE_{x_{[i]},y_{[i]},x_{J\backslash [i]}} \prod_{j\leq i}
\nu_{e_j}^*(x_{[i]\cap e_j},y_{[i]\cap e_j},x_{e_j\backslash [d]})
\prod_{j>i} F_{e_j}^* (x_{[i]\cap e_j}, y_{[i]\cap e_j}, x_{e_j\backslash [d]})\\
&\times \prod_{f\in\HH'} \nu_f^* (x_{f\cap [i]},y_{f\cap [i]},x_{f\backslash [i]}).
\end{align}
\\\\For $i=1$ this can be seem from \eqref{6.1}. Note that the linear forms appearing in any of these factors are pairwise linearly independent as our system is well-defined. Assuming it holds for $i$ separating the factors independent of the $x_{i+1}$ variable, replacing the function $F_{e_{i+1}}$ with $\nu_{e_{i+1}}$, and applying the Cauchy-Schwartz inequality we double the variable $x_{i+1}$ to the pair $(x_{i+1},y_{i+1})$ and each factor $G^*_e(x_{e\cap [i]},y_{e\cap [i]},x_{e\backslash [i]})$ depending on it, to obtain the factor $G^*_e(x_{e\cap [i+1]},y_{e\cap [i+1]},x_{e\backslash [i+1]})$, thus the formula holds for $i+1$.
After finishing this process we have by \eqref{6.2} and \eqref{6.3}
\[
E^{2^d}\ls \EE_{x_{[d]},y_{[d]}} \prod_{\om\in \{0,1\}^d} F_{e_0}
(\om(x_{[d]},y_{[d]}))\ \prod_{f\subs [d],f\neq e_0}\prod_{\om_f\in\{0,1\}^f}
\nu_f(\om_f(x_f,y_f))\,\WW(x_{[d]},y_{[d]}),
\]
where
\[\WW(x_{[d]},y_{[d]})=\EE_{x_{d+1}} \prod_{d+1\in e\in\HH} \prod_{\om_e\in \{0,1\}^{e\cap [d]}} \nu_e (\om_e (x_{e\cap [d]},y_{e\cap [d]},x_{e\backslash [d]})).
\]
Thus, as $F_{e_0}\leq \nu_{e_0}$, to prove \eqref{neumannineq} it is enough to show that
\[\EE_{x_{[d]},y_{[d]}} \prod_{f\subs [d]}\prod_{\om_f\in\{0,1\}^f}
\nu_f(\om_f(x_f,y_f))\,|\WW(x_{[d]},y_{[d]})-1|=o_{\n}(1).\]
This, similarly as in \cite{GT1}, can be done with one more application of the Cauchy-Schwartz inequality leading to 4 terms involving the "big" weight functions $\WW$ and $\WW^2$. Each term is however $1+o_{\n}(1)$ by the linear forms condition, as the underlying linear forms are pairwise linearly independent. Indeed the forms $L_f(\om_f(x_f,y_f)$ are pairwise independent for $f\subs [d]$, and depend on a different set of variables then the forms $L_e(\om_e (x_{e\cap [d]},y_{e\cap [d]},x_{e\backslash [d]}))$ for $e\nsubseteq [d]$ defining the weight function $\WW$. The new forms appearing in $\WW^2$ are copies of the forms in $\WW$ with the $x_{d+1}$ variable replaced by a new variable $y_{d+1}$ hence are independent of each other and the rest of the forms. This proves the proposition.
\end{proof}


\bigskip
\bigskip

\end{document}